\documentclass[a4paper,12pt]{amsproc}
\pdfoutput=1
\usepackage{mathrsfs}
\usepackage{appendix}
\usepackage{amssymb}
\usepackage{fancyhdr}
\usepackage{charter}
\usepackage{typearea} 
\usepackage{pdfsync}
\usepackage{mathrsfs}
\usepackage{dsfont}
\usepackage{appendix}
\usepackage{amssymb}
\usepackage{fancyhdr}
\usepackage{charter}
\usepackage{typearea}
\usepackage{color}
\usepackage{hyperref}
\usepackage{amsmath,amstext,amsthm,amscd}
\usepackage{mathrsfs}
\usepackage[margin=2.5cm]{geometry}
\usepackage{enumitem}
\usepackage{setspace}
\usepackage{bbm} 
\usepackage{charter}

\DeclareMathOperator{\Ker}{Ker}
\DeclareMathOperator{\Rang}{Rang}
\DeclareMathOperator{\Dom}{Dom}
\DeclareMathOperator{\End}{End}
\DeclareMathOperator{\Id}{Id}
\DeclareMathOperator{\supp}{supp}
\newcommand{\Tl}{\mathcal{T}}
\newcommand{\N}{\mathbb{N}}
\newcommand{\R}{\mathbb{R}}
\newcommand{\Cinf}{\mathcal{C}^{\infty}}
\newcommand{\Ct}{\mathcal{C}^{2}}

\newcommand{\Cs}{\mathbb{C}}
\newcommand{\Cn}{\mathbb{C}^{n}}

\newcommand{\ii}{\sqrt{-1}}
\newcommand{\p}{\partial}
\newcommand{\pb}{\bar{\partial}}
\newcommand{\la}{\langle}
\newcommand{\ra}{\rangle}

\newcommand{\tchi}{\Tilde{\chi}}
\newcommand{\To}{\rightarrow}
\newcommand{\ov}{\overline}

\newcommand{\Pk}{\mathcal{P}^{(q)}_{k,c_k}}
\newcommand{\Pks}{\mathcal{P}^{(q)}_{k,c_k,s}}
\newcommand{\Plk}{\mathcal{P}^{(q)}_{(k),c_k}}
\newcommand{\Plks}{\mathcal{P}^{(q)}_{(k),c_k,s}}
\newcommand{\Pkk}{P^{(q)}_{k,c_k}}

\newcommand{\Pksk}{P^{(q)}_{k,c_k,s}}

\newcommand{\Plkski}{P^{(q),I,J}_{(k),c_k,s}}

\newcommand{\Plksk}{P^{(q)}_{(k),c_k,s}}

\newcommand{\Bk}{\mathcal{B}^{(q)}_{k}}
\newcommand{\Bks}{\mathcal{B}^{(q)}_{k,s}}

\newcommand{\Blks}{\mathcal{B}^{(q)}_{(k),s}}
\newcommand{\Bkk}{B^{(q)}_{k}}

\newcommand{\Bksk}{B^{(q)}_{k,s}}

\newcommand{\Blksk}{B^{(q)}_{(k),s}}
\newcommand{\Bs}{\mathcal{B}^{(q)}_{s}}
\newcommand{\Bsk}{B^{(q)}_{s}}

\newcommand{\Bsko}{B^{(q)}_{0,s}}

\newcommand{\tBlks}{\Tilde{\mathcal{B}}^{(q)}_{(k),s}}
\newcommand{\tBlksk}{\Tilde{B}^{(q)}_{(k),s}}
\newcommand{\tBlkso}{\Tilde{\mathcal{B}}^{(q_0)}_{(k),s}}

\newcommand{\tpb}{\Tilde{\bar{\partial}}}

\newcommand{\tU}{\Tilde{U}}
\newcommand{\ts}{\Tilde{s}}
\newcommand{\tL}{\Tilde{L}}
\newcommand{\tphi}{\Tilde{\phi}}
\newcommand{\tvarphi}{\Tilde{\varphi}}
\newcommand{\tomega}{\Tilde{\omega}}
\newcommand{\T}{T^{*,(0,q)}\Cn}

\newcommand{\TM}{T^{*,(0,q)}M}

\newcommand{\zb}{\bar{z}}

\newcommand{\pzi}{\dfrac{\partial}{\partial z^i}}
\newcommand{\pzj}{\dfrac{\partial}{\partial z^j}}

\newcommand{\pzbi}{\dfrac{\partial}{\partial \bar{z}^i}}

\newcommand{\pzbj}{\dfrac{\partial}{\partial \bar{z}^j}}
\newcommand{\wb}{\bar{w}}

\newcommand{\pwb}{\dfrac{\partial}{\partial \bar{w}}}

\numberwithin{equation}{section}
\theoremstyle{plain}
\newtheorem{thm}{Theorem}[section]
\newtheorem{prop}[thm]{Proposition}
\newtheorem*{claim}{Claim}
\newtheorem{cor}[thm]{Corollary}
\newtheorem{lem}[thm]{Lemma}

\theoremstyle{definition}
\newtheorem{defin}{Definition}[section]
\newtheorem{assumption}{Assumption}[section]
\newtheorem{statement}{Statement}[section]

\theoremstyle{remark}
\newtheorem{rmk}{Remark}[section]

\makeatletter
\def\l@subsection{\@tocline{2}{0pt}{2.5pc}{5pc}{}}
\def\l@subsubsection{\@tocline{2}{0pt}{5pc}{7.5pc}{}}
\makeatother

\begin{document}

\title{Semi-Classical Asymptotics of Bergman and
Spectral Kernels for $(0,q)$-forms}
\author{Yueh-Lin Chiang}
\thanks{The thesis is presented for the Master's degree in Mathematics at National Taiwan University and supervised by Professor Chin-Yu Hsiao.}

\begin{spacing}{1.5}
\begin{abstract}
In this paper, we develop a new scaling method to study spectral and Bergman kernels for the $k$-th tensor power of  a line bundle over a complex manifold under local spectral gap condition. In particular, we establish a simple proof of the pointwise asymptotics of spectral and Bergman kernels. As a new result, in the function case, we obtain the leading term of Bergman kernel under spectral gap with exponential decay. Moreover, in the general cases of $(0,q)$-forms, the asymptotics remain valid while the curvature of the line bundle is degenerate.  
\end{abstract}
\end{spacing}
\begin{spacing}{1}
\maketitle \tableofcontents
\end{spacing}
\newpage
\section{Introduction}
Let $M$ be a Hermitian complex manifold with $\text{\rm dim}_{\Cs}M=n$ and equip $M$ with a positive Hermitian $(1,1)$-form $\omega$. Consider a holomorphic line bundle $L$ over $M$ with a locally defined weight function $\phi$ that gives $L$ a Hermitian metric $h$. The Hermitian form $\omega$ and the metric $h$ endow the space of $L$-valued $(0,q)$-forms with a $L^2$-inner product. By taking the completion of this space with respect to the inner product, we obtain the Hilbert space $L^2_{\omega,\phi}(M,\TM \otimes L)$. Consider $\Box^{(q)}_{\omega,\phi}$ to be the Kodaira Laplacian induced by the Hermitian structures $\omega$ and $h$. The \textbf{Bergman projection} \[\mathcal{B}^{(q)}_{\omega,\phi}: L^2_{\omega,\phi}(M,\TM \otimes L)\To \Ker \Box^{(q)}_{\omega,\phi}\] is the orthogonal projection from the space of $L^2$-integrable sections of $\TM\otimes L$ onto the space of harmonic sections with respect to Kodaira Laplacian $\Box^{(q)}_{\omega,\phi}$. For a Borel set $B\subset \R$, we denote by $\mathbbm{1}_B(\Box^{(q)}_{\omega,\phi})$ the functional calculus of $\Box^{(q)}_{\omega,\phi}$ with respect to the indicator function $\mathbbm{1}_{B}$ (cf. \cite[section 2]{Davies}). Given a nonnegative constant $c$, the \textbf{spectral projection} 
\[\mathcal{P}^{(q)}_{\omega,\phi,c}:=\mathbbm{1}_{[0,c]}(\Box^{(q)}_{\omega,k\phi}):L^2_{\omega,\phi}(M,\TM \otimes L)\To E^{(q)}_{\leq c}\] is the orthogonal projection onto the space $\Rang \left(\mathbbm{1}_{[0,c]}(\Box^{(q)}_{\omega,k\phi})\right)$ denoted by $E^{(q)}_{\leq c}$. The \textbf{Bergman kernel} $B^{(q)}_{\omega,\phi}(z,w)$ is the Schwartz kernel of $\mathcal{B}^{(q)}_{\omega,\phi}$ and the \textbf{spectral kernel} $P^{(q)}_{\omega,\phi,c}(z,w)$ is the Schwartz kernel of $\mathcal{P}^{(q)}_{\omega,\phi,c}$. \par 
The Bergman kernel is a fundamental object in complex analysis and geometry, which plays a central role in some important problems in complex geometry, geometric quantization, and mathematical physics. However, it is challenging to study the Bergman kernel directly. Inspired by quantum mechanics and semi-classical analysis, if we consider the $k$-th tensor power $L^k$ of $L$ and replace the Hermitian metric $\phi$ by $k\phi$, it is possible to handle the asymptotic behavior of the Bergman kernel as $k$ goes to infinity. Therefore, the study of the large $k$ behavior of the Bergman kernel $B^{(q)}_{\omega,k\phi}(z,w)$ has become prominent in modern research. The asymptotic behavior of the Bergman kernel $B^{(q)}_{\omega,k\phi}(z,w)$ is rich in geometrical meaning and closely related to index theory and algebraic geometry. In \cite{Be01}, R. Berman obtained the local holomorphic Morse inequalities by analyzing the Bergman kernel on the diagonal part. In \cite{Hsia01}, C.-Y. Hsiao illustrated a proof of the Kodaira embedding theorem by the full expansion. Furthermore, the approximation of Kähler
metrics(e.g.,\cite{app2},\cite{app1}),
existence of canonical Kähler metrics (e.g., \cite{app3},\cite{app4},\cite{app5},\cite{app6}) and the Berezin-Toeplitz quantization (e.g.,\cite{app8},\cite{app10},\cite{app7},\cite{app9}) are impressive applications. We refer readers to the book \cite{book01} of X. Ma and G. Marinescu for a comprehensive study of Bergman kernel and relative subjects.\par 

 For a compact manifold $M$ with a positive line bundle $L$,  T. Bouche (1990,\cite{app2}) and G. Tian (1990,\cite{app1}) obtained the leading term of the Bergman kernel, and D. Catlin proved the full expansion (1997,\cite{catlin}) later. More precisely, D.Catlin claimed that 
\begin{equation}\label{1.0 asymptotic}
    B^{(q)}_{\omega,k\phi}(z,z)\sim k^n b^{(q)}_n+k^{n-1}b^{(q)}_{n-1}+\cdots+b^{(q)}_0 \quad \text{\rm as }\,  k\To\infty
\end{equation} for the case $q=0$. Furthermore, X. Dai, K. Liu and X. Ma gave another proof of the full expansion based on localized techniques and heat kernel methods (2004,\cite{DaiLiuMa01}),(2008,\cite{DaiLiuMa02}) and B. Berndtsson, R. Berman and J. Sjöstrand also offered a different proof (2008,\cite{Be02}). \par

In the case of non-degenerate line bundle $L$ which may not be positive, if $M$ is compact and $M=M(q)$ (cf. Def.\ref{1.1 def curvature}), there is a full asymptotic expansion of $B^{(q)}_{\omega,k\phi}(z,w)$ proven by R. Berman and J. Sjöstrand (2007,\cite{Be03}). Moreover, in (2006,\cite{Ma01}), X. Ma and G. Marinescu established similar results in the context of $\text{spin}^{c}$-Dirac operators in compact symplectic manifolds. In a later work,  C.-Y. Hsiao and G. Marinescu (2014,\cite{Hsia02})  demonstrated that the Bergman kernel has a local asymptotic expansion at all non-degenerate points under the local spectral gap condition (cf. Def.\ref{1.2 spectral gap 1}). Also, they showed that the spectral kernel $P^{(q)}_{\omega,k\phi,k^{-N}}$ has an analogous result. \par 

In this paper, we derive the leading term $b^{(q)}_n$ of the asymptotic expansion (cf.(\ref{1.0 asymptotic})) by scaling method under the local spectral gap condition (cf.  Def.\ref{1.2 spectral gap 1}). For the function case, we can loosen the spectral gap condition to an exponential decay rate (cf. Def.\ref{1.2 spectral gap 2}). It is noteworthy that we do not require the curvature to be non-degenerate.\par  
As for the spectral kernel, we fix a sequence $c_k$ satisfying $\limsup_{k\To\infty} k^{-1}c_k=0$ and consider the asymptotic behavior of $P^{(q)}_{\omega,k\phi,c_k}(z,z)$. If there exists an integer $d$ such that $\liminf_{k\To \infty} k^d c_k>0$, then we can also obtain the leading term of the expansion of $P^{(q)}_{\omega,k\phi,c_k}(z,z)$. Furthermore, in the case $q=0$, we only need a weaker condition of $c_k$ that \[\exists c<1 \,\text{such that }\, \liminf e^{2c\min{\lambda_i}\cdot k^{1/2}} c_k>0\quad \text{where }\,\lambda_i\,\text{are defined in }(\ref{lll}) .\]

\subsection{Set-up and the main results}\label{section 1.1}

Let $(M,\omega)$ be a Hermitian manifold with complex dimension $n$ where $\omega$ is a positive Hermitian $(1,1)$-form. Denote by $\la\cdot|\cdot\ra_{\omega}$ the pointwise Hermitain inner product induced by $\omega$ on $T_{\Cs}M$ and $dV_{\omega}$ the induced Riemannian volume form given by $\dfrac{\omega^n}{n!}$.\par 

We consider a holomorphic Hermitian line bundle $(L, h^L)$ over the manifold $M$, and denote its $k$-th tensor power $L^{\otimes k}$ by $L^k$. Let $s$ be a local holomorphic trivializing section of $L$ over an open subset $U$ of $M$. The Hermitian metric $h^{L}$ corresponds locally to a \textbf{weight function} $\phi:U \To \R$ such that
$|s|^2_{h^L}=e^{-2\phi}$. Denote by $s^k$ the $k$-th tensor power $s^{\otimes k}$ of $s$. Then the metric of $L^k$ in $U$ can be described as
$
|s^{k}|^2_{k\phi}:=|s^{k}|^2_{h^{L^k}}=e^{-2k\phi}
$ where $s^k$ trivializes $L^k$ in $U$ with its weight function $k\phi$. Denote by $\la \cdot|\cdot\ra_{k\phi}:=\la \cdot|\cdot\ra_{h^{L^k}}$ the pointwise Hermitian inner product $h^{L^k}$ on $L^k$ for convenience.\par

We also introduce the holomorphic Hermitian connection $\nabla^L$ on $(L, h^L)$ that has a curvature form denoted by $\Theta^L$. We identify $\Theta^L$ with a Hermitian matrix $\dot{\Theta}^L \in \mathcal{C}^{\infty}(M,\End(T^{(1,0)}M))$ that satisfies the following equation:
\begin{equation*}
\langle \dot{\Theta}^L(z)v_1\mid v_2\rangle_{\omega}:=\Theta^L(z)(v_1\wedge\ov{v_2}) \quad \text{\rm for all }\, v_1, v_2\in T^{(1,0)}_zM ,\quad z\in M.
\end{equation*}Next, we set the notation describing the signature of the curvature.
 \begin{defin}\label{1.1 def curvature}
 For any $q\in\{0,1,\dots,n\}$, we denote 
 \begin{multline*}
    M(q):=\{ z\in M \mid   \dot{\Theta}^L(z)\in \End(T_z^{(1,0)}M)  \, \text{\text is non-degenerate} \\ \text{\text and has exactly $q$ negative eigenvalues}\}. 
 \end{multline*}
  \end{defin} \par 
There is a natural Hermitian structure denoted by $\la \cdot | \cdot  \ra_{\omega,k\phi}$ on the vector bundle $\TM \otimes L^k$ over $M$ obtained by the Hermitian pointwise inner product on $\TM$ induced by $\omega$ (cf.(\ref{2.2 eq 1})) and the local weight functions $k\phi$ of the Hermitian metric $h^{L^k}$ of $L^k$, where $\TM$ denotes the bundle of $(0,q)$-forms on $M$ (cf.(\ref{2.2 0,q form})). Let $\Omega^{(0,q)}(M,L^k)$ be the space of smooth $(0,q)$-forms on $M$ with values in $L^k$, and let $\Omega^{(0,q)}_c(M,L^k)$ be the subspace of $\Omega^{(0,q)}(M,L^k)$ consisting of elements with compact support in $M$.  The pointwise inner product $\la\cdot|\cdot\ra_{\omega,k\phi}$ on $\TM\otimes L^k$ induces a $L^2_{\omega,k\phi}$-inner product $(\cdot|\cdot)_{\omega,k\phi}$ on the space $\Omega^{(0,q)}_c(M,L^k)$ (cf.(\ref{2.2 inner product 1})). Denote $L^2_{\omega,k\phi}(M,\TM\otimes L^k)$ as the completion of $\Omega^{(0,q)}_c(M,L^k)$ with respect to $(\cdot|\cdot)_{\omega,k\phi}$ and denote $\|\cdot\|_{\omega,k\phi}$ as its norm. \par 
Let $\pb^{(q)}_{k}:\Omega^{(0,q)}(M,L^k)\To\Omega^{(0,q+1)}(M,L^k)$ be the Cauchy-Riemann operator with values in $L^k$ and $\pb^{*,(q+1)}_{k}:\Omega^{(0,q+1)}(M,L^k)\To\Omega^{(0,q)}(M,L^k)$ be the formal adjoint of $\pb^{(q)}_{k}$ with respect to $(\cdot|\cdot)_{\omega,k\phi}$. Recall that the \textbf{Kodaira Laplacian} is given by  \[\Box^{(q)}_{\omega,k\phi}:=\pb^{*,(q+1)}_{k}\pb^{(q)}_{k}+\pb^{(q-1)}_{k}\pb^{*,(q)}_{k}:\Omega^{(0,q)}(M,L^k)\To \Omega^{(0,q)}(M,L^k)\] and it has the Gaffney extension(cf.(\ref{2.3 Gaffney extansion})): \[\Box^{(q)}_{\omega,k\phi}:\Dom\Box^{(q)}_{\omega,k\phi}\subset L^2_{\omega,k\phi}(M,\TM\otimes L^k)\To L^2_{\omega,k\phi}(M,\TM\otimes L^k).\] Denote by $E^{(q)}_{k,\leq c}$ the image of $\mathbbm{1}_{[0,c]}\left(\Box^{(q)}_{\omega,k\phi}\right)$ which is the functional calculus of $\Box^{(q)}_{\omega,k\phi}$ with respect to the indicator function $\mathbbm{1}_{[0,c]}$. We specify a nonnegative sequence $c_k$ and denote by  (cf.(\ref{2.3 spectral projection}))  
\[
\mathcal{P}^{(q)}_{k,c_k}:=\mathbbm{1}_{[0,c_k]}\left(\Box^{(q)}_{\omega,k\phi}\right):L^2_{\omega,k\phi}(M,\TM\otimes L^k)\To E^{(q)}_{k,\leq c_k}
\]the \textbf{spectral projection} which is the orthogonal projection. Specifically, in the case $c_k=0$, denote \[\mathcal{B}^{(q)}_{k}:L^2_{\omega,k\phi}(M,\TM\otimes L^k)\To \Ker\Box^{(q)}_{k}\] to be the \textbf{Bergman projection}. Define $P^{(q)}_{k,c_k}(z,w)$ to be the \textbf{spectral kernel} and $B^{(q)}_{k}(z,w)$ to be the \textbf{Bergman kernel} which are the Schwartz kernels of $\mathcal{P}^{(q)}_{k,c_k}$ and $\mathcal{B}^{(q)}_{k}$, respectively.

Now, we choose a suitable holomorphic coordinate chart $U$ centered at $p\in M$ and a holomorphic trivialization $s$ on $U$ such that (cf. Lemma \ref{2.2 coordinate chart})
    \begin{align}\label{lll}\phi&=\sum_{i=1}^{n}\lambda_i|z^i|^2+O(|z|^3) \quad ; \quad 
        \omega= \ii\sum_{i=1}^{n}dz^i \wedge d\zb^i + O(|z|).
    \end{align}Moreover, if $\lambda_i> 0$ for all $i=1,\cdots,n$, we take the  trivialization such that  \[\phi=\sum_{i=1}^{n}\lambda_i|z^i|^2+O(|z|^4).\] Note that if $p\in M(q')$ for some $q'\in\{0,\cdots,n\}$, then 
    \[
    q'=\#\{i\,;\lambda_i<0\}\quad 
\text{\rm and }\quad  n-q'=\#\{i\,;\lambda_i>0\}. \] In this paper, we always assume that $\lambda_i<0$ for all $i=1,\cdots,q'$ by rearrangement. Next, we introduce the spectral gap conditions.
\begin{defin}[spectral gap condition 1]\label{1.2 spectral gap 1}
    For any $q\in\{0,\cdots,n\}$ and an open set $U \subset M$, we say $\Box^{(q)}_{\omega,k\phi}$ has a small local spectral gap condition of polynomial rate on $U$ if there exist $d\in \N$ and $C >0$ such that for all large enough $k$,
    \[
        \|\left(I-\mathcal{B}^{(q)}_{k}\right)u\|^2_{\omega,k\phi}\leq  C k^{d}\left( \Box^{(q)}_{\omega,k\phi}u\mid u\right)_{\omega,k\phi} \quad \text{for all }\, u\in\Omega^{(0,q)}_c(U,L^k).
    \]
\end{defin}
For the function case $q=0$, we introduce a relaxed condition that allows for a \textbf{narrower} spectral gap. 
\begin{defin}[spectral gap condition 2]\label{1.2 spectral gap 2}
For an open set $U \subset M$, we say $\Box^{(0)}_{\omega,k\phi}$ has a small local spectral gap condition of suitable exponential rate on $U$ if  there are constants $0<c<1$ and $C>0$ such that for large enough $k$, 
    \begin{equation*}
        \|\left(I-\mathcal{B}^{(0)}_{k}\right)u\|^2_{\omega,k\phi}\leq  Ce^{2c\min{\lambda_i}\cdot k^{1/2}}\left( \Box^{(0)}_{\omega,k\phi}u\mid u\right)_{\omega,k\phi}\quad \text{for all }\, u\in\Cinf_c(U,L^k).
    \end{equation*}
\end{defin}
Let $s:U\To L$ be a local non-vanishing holomorphic section defined on an open set $U\subset M$. We can locally express the spectral and Bergman kernels on $U \times U$ as 
\begin{align}
 P^{(q)}_{k,c_k}(z,w)=P^{(q),s}_{k,c_k}(z,w)\,s^k(z)\otimes (s^k(w))^*;\label{8.1}\\
B^{(q)}_{k}(z,w)=B^{(q),s}_{k}(z,w)\,s^k(z)\otimes (s^k(w))^*.\notag
\end{align}Here, $P^{(q),s}_{k,c_k}(z,w)$ and $B^{(q),s}_{k}(z,w)$ are elements in 
 $\Cinf(U \times U,\TM \boxtimes \TM )$ where $T^{*,(0,q)}M \boxtimes T^{*,(0,q)}M$ is the vector bundle over $U\times U$ whose fiber at  $(z,w)\in U\times U$ is the space of linear transformations from $T^{*,(0,q)}_wM$ to $T^{*,(0,q)}_zM$. We now introduce the primary object in our approach. 
\begin{defin}\label{8.2}
 We treat $U$ as a subset in $\Cn$ and assume that $U$ is convex. The scaled spectral kernel $P^{(q),s}_{(k),c_k}\in\Cinf\left(\sqrt{k}U\times\sqrt{k}U, T^{*,(0,q)}\Cn\boxtimes T^{*,(0,q)}\Cn\right)$ is defined by
 \[
 P^{(q),s}_{(k),c_k}(z,w):=k^{-n}P^{(q),s}_{k,c_k}(\dfrac{z}{\sqrt{k}},\dfrac{w}{\sqrt{k}}).
 \]Similarly, the scaled Bergman kernel is defined by 
 \[
 B^{(q),s}_{(k)}(z,w):=k^{-n}B^{(q),s}_{k}(\dfrac{z}{\sqrt{k}},\dfrac{w}{\sqrt{k}}).
 \]
\end{defin}
We are ready to illustrate the main results of this paper.
\begin{thm}[main theorem for Bergman kernel]\label{main thm 1}
     If $p \notin M(q)$, the scaled Bergman kernel $B^{(q),s}_{(k)}(z,w) \To 0$ locally uniformly in $\Cinf$ on $\Cn$.
If $p\in M(q)$ and $\Box^{(q)}_{\omega,k\phi}$ has local small spectral gap condition of polynomial rate in $U$ (cf. Def. \ref{1.2 spectral gap 1}), then $B^{(q),s}_{(k)}(z,w)$ converges to
    \begin{equation*}\label{1.1 main thm 1.1}
       \dfrac{|\lambda_1 \cdots \lambda_n|}{\pi^n}\,e^{2(\sum_{i=1}^{q}|\lambda_i|\zb^i w^i+\sum_{i=q+1}^{n}|\lambda_i|z^i\wb^i-\sum_{i=1}^{n}|\lambda_i||w^i|^2})(d\zb^1\wedge \cdots \wedge d\zb^q) \otimes (\dfrac{\p}{\p \wb^{1}} \wedge \cdots \wedge \dfrac{\p}{\p \wb^{q}})
    \end{equation*}locally uniformly in $\Cinf$ on $\Cn$. Here, we identify $(d\zb^1\wedge \cdots \wedge d\zb^q) \otimes (\dfrac{\p}{\p \wb^{1}} \wedge \cdots \wedge \dfrac{\p}{\p \wb^{q}})$ as a section of $\T\boxtimes\T$ over $\Cn$ defined by \[\eta \mapsto (d\zb^1\wedge \cdots \wedge d\zb^q)\otimes \eta(\dfrac{\p}{\p \wb^{1}} \wedge \cdots \wedge \dfrac{\p}{\p \wb^{q}}) \quad \text{ for all }\, \eta\in \T.\]
    In particular, in the case $p\in M(0)$, the convergence above for the function case $q=0$ remains valid if 
 $\Box^{(0)}_{\omega,k\phi}$ has only local small spectral gap condition of suitable exponential rate in $U$ (cf. Def. \ref{1.2 spectral gap 2}). 
\end{thm}
Next, the second main theorem is the spectral kernel version. The spectral gap conditions can be dropped and conditions can be imposed on the sequence $c_k$ since \[
\|\left(I-\Pk\right)u\|^2_{\omega,k\phi}\leq c_k \left(\Box^{(q)}_{\omega,k\phi}u\mid u\right) \quad \text{\rm for all }\, u \in L^2_{\omega,k\phi}(M,\TM\otimes L^k). 
\]This estimate plays the role of a spectral gap condition.
\begin{thm}[main theorem for spectral kernel]\label{main thm 2}
    Assume that the nonnegative sequence $c_k$ satisfies
    \[
    \limsup_{k\To\infty}\dfrac{c_k}{k}=0.
    \] If $p\notin M(q)$, the scaled spectral kernel $P^{(q),s}_{(k),c_k}(z,w) \To 0$ locally uniformly in $\Cinf$ on $\Cn$. 
    If $p\in M(q)$ and there exists $d\in \N$ such that $\liminf_{k\To\infty} k^dc_k>0$, then $P^{(q),s}_{(k),c_k}(z,w)$ converges to
    \begin{equation*}
         \dfrac{|\lambda_1 \cdots \lambda_n|}{\pi^n}\,e^{2(\sum_{i=1}^{q}|\lambda_i|\zb^i w^i+\sum_{i=q+1}^{n}|\lambda_i|z^i\wb^i-\sum_{i=1}^{n}|\lambda_i||w^i|^2})(d\zb^1\wedge \cdots \wedge d\zb^q) \otimes (\dfrac{\p}{\p \wb^{1}} \wedge \cdots \wedge \dfrac{\p}{\p \wb^{q}})
    \end{equation*}locally uniformly in $\Cinf$ on $\Cn$.\par  Particularly, in the case $p\in M(0)$, the convergence above of the function case $q=0$  still holds under a weaker condition of $c_k$ that $\liminf e^{2c\min\{\lambda_i\}k^{1/2}}c_k>0$ for some $c<1$.
\end{thm}  

\begin{rmk}
For a fixed point $p\in M$, observe that
$
B^{(q),s}_{k}(p,p)=k^{n}B^{(q),s}_{(k)}(0,0)$. By Theorem \ref{main thm 1}, under spectral gap conditions, we deduce
\begin{align*}
    B^{(q),s}_{k}(p,p)&= k^n \dfrac{|\lambda_1 \cdots \lambda_n|}{\pi^n}(d\zb^1\wedge \cdots \wedge d\zb^q) \otimes (\dfrac{\p}{\p \wb^{1}} \wedge \cdots \wedge \dfrac{\p}{\p \wb^{q}})+o(k^n) \quad \text{\rm if }\, p\in M(q);\\
    B^{(q),s}_{k}(p,p)&= o(k^n) \quad \text{\rm if }\, p\notin M(q).
\end{align*}In a similar way, by Theorem \ref{main thm 2}, we are able to conclude the same asymptotic behavior for the diagonal part of the spectral kernels $P^{(q),s}_{k,c_k}(p,p)$ under the suitable conditions on $c_k$. From our results, if the expansion (\ref{1.0 asymptotic}) exists, we can conclude that
\begin{align*}
    b^{(q)}_{n}(p,p)&=\dfrac{|\lambda_1 \cdots \lambda_n|}{\pi^n}(d\zb^1\wedge \cdots \wedge d\zb^q) \otimes (\dfrac{\p}{\p \wb^{1}} \wedge \cdots \wedge \dfrac{\p}{\p \wb^{q}})\otimes s^{k}\otimes \left(s^{k}\right)^*\quad\text{\rm if }\, p\in M(q);\\
    b^{(q)}_{n}(p,p)&=0 \quad \text{\rm if }\, p\notin M(q). 
\end{align*}
\end{rmk}
\begin{rmk}
     Theorem \ref{main thm 1} provides a purely analytic proof of the Kodaira embedding theorem (cf.\cite{Hsia01}), while Theorem \ref{main thm 2} can be used to establish the Demailly’s Morse inequality (cf.\cite[section 10.5]{Hsia02}). 
\end{rmk}
We divide the proof of the main theorems into two steps. First, in Chapter \ref{chapter 3}, we try to establish local uniform bounds of $B^{(q),s}_{(k)}(z,w)$ and $P^{(q),s}_{(k),c_k}(z,w)$ on $\Cn$ (cf. Theorem \ref{3.3 loc bdd thm}). In this way, we can infer that any subsequence of $B^{(q),s}_{(k)}$ ( or $P^{(q),s}_{(k),c_k}$)  has a $\Cinf$ uniformly convergent subsequence by the Arzelá-Ascoli theorem. \par 
Next, in Chapter \ref{chapter 4}, we prove that every convergent subsequence of $B^{(q),s}_{(k)}$ (or $P^{(q),s}_{(k),c_k}$) must converge to the Bergman kernel of the model case on $\Cn$ (cf. Theorem \ref{4.2 main thm}, Theorem \ref{4.3 main thm 2}, Theorem \ref{4.5 main thm}), which is exactly
\begin{equation*}
        \dfrac{|\lambda_1 \cdots \lambda_n|}{\pi^n}\,e^{2(\sum_{i=1}^{q}|\lambda_i|\zb^i w^i+\sum_{i=q+1}^{n}|\lambda_i|z^i\wb^i-\sum_{i=1}^{n}|\lambda_i||w^i|^2})(d\zb^1\wedge \cdots \wedge d\zb^q) \otimes (\dfrac{\p}{\p \wb^{1}} \wedge \cdots \wedge \dfrac{\p}{\p \wb^{q}}). 
    \end{equation*}

    \newpage

\section{Preliminaries and terminology}In this chapter, we provide the foundational knowledge and notation necessary for the paper. Specifically, we first discuss the basic structure of complex analytic geometry and line bundles in Section \ref{section 2.2}. Then, in Section \ref{section 2.3}, we discuss the Kodaira Laplacian and the relevant spectral theory. Also, we introduce the spectral and Bergman kernels. In Section \ref{section 2.4}, we introduce Sobolev theory, which will be utilized in Section \ref{section 3.3}.
\subsection{Standard notations}\label{section 2.1}
 Let $\N_{0}$ be the set $\N \cup \{ 0 \}$, and a multi-index $\alpha$ is of the form $\alpha = (\alpha_1,\alpha_2,...,\alpha_n) \in (\N_0)^n $. Denote $|\alpha| := \sum_i \alpha_i$ and $\alpha!=\alpha_1!\alpha_2!\cdots\alpha_n!$. For $\xi=(\xi_1,\cdots,\xi_n)\in \R^n$, $\xi^{\alpha}:=\xi_1^{\alpha_1}\cdots \xi_n^{\alpha_n}$.
 \par
Let $M$ be a $n$-dimensional complex manifold and $TM$ be the real tangent bundle of the underlying smooth manifold. Denote by $T_\Cs M$ the complexified tangent bundle $TM \otimes \Cs$ and $\bigwedge^l T_{\Cs}^*M$ the $l$-th exterior algebra of the cotangent bundle $T_\Cs^* M$. For a local holomorphic coordinate $(z^1,\cdots,z^n)$ that has an underlying real coordinate $(x^1,\cdots,x^{2n})$ with $z^j=x^{2j-1}+\ii x^{2j}$, let \[\pzj:=\dfrac{1}{2}\left(\dfrac{\p}{\p x^{2j-1}}-\ii\dfrac{\p}{\p x^{2j}}\right)\quad \text{\rm and}\quad\pzbj:=\dfrac{1}{2}\left(\dfrac{\p}{\p x^{2j-1}}+\ii\dfrac{\p}{\p x^{2j}}\right)\]as sections of $T_{\Cs}M$. Therefore, $dz^j:=dx^{2j-1}+\ii dx^{2j}$ and $d\Bar{z}^j:=dx^{j-1}-\ii dx^{2j}$ are sections of $T^{*}_{\Cs}M$.  For a multi-index $\alpha=(\alpha_1,\cdots,\alpha_{n})\in (\N_0)^n$, we denote $\dfrac{\p}{\p z^{\alpha}}:=(\dfrac{\p}{\p z^1})^{\alpha_1}\cdots(\dfrac{\p}{\p z^n})^{\alpha_n}$ and $\dfrac{\p}{\p \zb^{\alpha}}:=(\dfrac{\p}{\p \zb^1})^{\alpha_1}\cdots(\dfrac{\p}{\p \zb^n})^{\alpha_n}$. Sometimes, we simply write them as $\p^{\alpha}_{z}$ and $\p^{\alpha}_{\zb}$,  respectively. Also, for $\alpha\in \N_0^{2n}$, we denote $\dfrac{\p}{\p x^{\alpha}}:=(\dfrac{\p}{\p x^1})^{\alpha_1}\cdots(\dfrac{\p}{\p x^{2n}})^{\alpha_{2n}}$ and sometimes write it as $\p^{\alpha}_x$.\par
Define \[\mathcal{J}_{q,n} := \{I=(i_1,\ldots,i_q):1\leq i_1<i_2<\cdots<i_q\leq n\} \subset (\N_0)^{q}.\] For any element $I=(i_1,\ldots,i_q) \in \mathcal{J}_{q,n}$, we denote the $q$-forms $d z^{I}:=d z^{i_1}\wedge  \cdots \wedge d z^{i_q}$ and $d \Bar{z}^{I}:=d \Bar{z}^{i_1}\wedge \cdots \wedge d \Bar{z}^{i_q}$.
 \par
Consider an open subset $U$ of $M$. Denote by $\Cinf(U)$ the space of smooth functions on $U$ and by $\Cinf_c(U)$ the subspace of $\Cinf(U)$ whose elements have compact support in $U$. For a vector bundle $E$ over $M$, we denote $\Cinf(U,E)$ as the space of smooth sections of $E$ over $U$ and $\Cinf_c(U,E)$ as the subspace of $\Cinf(U,E)$ whose every element has compact support in $U$. 
Let $dm$ be the standard Lebesgue measure on $\Cn$, and let $B(r)$ be the set $\{z\in\Cn; |z|<r\}$.
 
\subsection{Complex geometry and Hermitian holomorphic line bundle}\label{section 2.2}

  Let $M$ be a complex manifold of dimension $n$. There is a natural complex structure $J:TM \rightarrow TM$ such that $J^2=-\Id$. Then $T_{\Cs}M= T^{(1,0)}M \oplus T^{(0,1)}M$  where $T^{(1,0)}M$ and $T^{(0,1)}M$ are the $i$-eigenbundle and $-i$-eigenbundle of $J$, respectively. Similarly, $T_\Cs^*M = T^{*,(1.0)}M\oplus T^{*,(0,1)}M$ where $T^{*,(1.0)}M$ and $T^{*,(0,1)}M$ are dual bundles of $T^{(1,0)}M$ and $T^{(0,1)}M$, respectively. \par  
 The splitting of the complexified tangent bundle can be extended to the exterior algebra of the complexified cotangent bundle. Namely,
  \begin{equation}\label{2.2 0,q form}
      \bigwedge^k T_{\Cs}^*M=\bigoplus_{p+q=k}\left(\bigwedge^p T^{*,(1,0)}M \right)\bigwedge\left(\bigwedge^q T^{*,(0,1)}M\right).
  \end{equation}
   Define $T^{*,(p,q)}M:=\left(\bigwedge^p T^{*,(1,0)}M \right)\bigwedge\left(\bigwedge^q T^{*,(0,1)}M\right)$ and hence $\bigwedge^k T_{\Cs}^*M=\bigoplus_{p+q=k}T^{*,(p,q)}M$.
 Let $\Omega^{(p,q)}(M)$ be the space of smooth $(p,q)$-forms which are smooth sections of $T^{*,(p,q)}M$ and $\Omega^{(p,q)}_c(M)$ be the subspace of $\Omega^{(p,q)}(M)$ consisting of elements with compact support in $M$. For a local holomorphic coordinate $(z^1,\cdots,z^n)$ in $U\subset M$, we have a local frame for $T^{*,(p,q)}M$ given by
\begin{equation*}
    T^{*,(p,q)}M\mid_U = \text{\rm span}\{d z^I \wedge d \zb^J\}_{I\in \mathcal{J}_{p,n} , J\in \mathcal{J}_{q,n}}.
\end{equation*} 
Next, we call $\omega$ a positive Hermitian $(1,1)$-form if: 
 \begin{enumerate}[label=(\roman*)]
     \item $\omega \in \Omega^{(1,1)}(M)$ ;
     \item For any local holomorphic coordinate $(z^1,\cdots,z^n)$, $\omega$ can be written as 
     \begin{equation*}
         \omega=\ii\sum_{i,j=1}^{n}h_{i,j}d z_i \wedge d \zb_j
     \end{equation*}
     where $ [h_{i,j}] $ is a positive Hermitian matrix.
 \end{enumerate}
 A positive Hermitian $(1,1)$-form $\omega$ induces pointwise Hermitian inner products $\la\cdot|\cdot\ra_\omega$ on $T^{(1,0)}M$ and $T^{(0,1)}M$ that are locally given by $\la \pzi | \pzj\ra_{\omega} := h_{i,j}$ and $\la \pzbi|\pzbj \ra_{\omega}:=\bar{h}_{i,j}$, respectively. Thus, we have a Hermitian inner product $\la\cdot|\cdot\ra_{\omega}$ on the complexified cotangent bundle $T_{\Cs}M=T^{(1,0)}M\oplus T^{(0,1)}M$. When we restrict the domain of $\la\cdot|\cdot\ra_{\omega}$ to the subbundle $TM\subset T_{\Cs}M$, we obtain a Riemannian metric called $g_\omega$ on the underlying real manifold. The Riemannian volume form $dV_{\omega}$ associated with $g_{\omega}$ is given by $dV_\omega = \frac{\omega^n}{n!}$. Moreover,
 The Hermitian inner product $\la\cdot|\cdot\ra_\omega$ can be naturally extended to $\TM$ by
\begin{equation}\label{2.2 eq 1}
    \la d \zb_{i_1}\wedge \cdots \wedge d \zb_{i_q} \mid d \zb_{j_1}\wedge \cdots
    \wedge d \zb_{j_q} \ra_{\omega}=\frac{1}{q!}  \overline{\text{det}[h^{i_l,j_k}]_{l,k=1...q}} 
\end{equation}
 where $[h^{i,j}]$ is the inverse matrix of $[h_{i,j}]$.

 We can now define the $L^2$-inner product on the space $\Omega^{(0,q)}_c(M)$ by
 \begin{equation}\label{2.2 inner product 0}
\left(\eta_1\mid\eta_2\right)_{\omega}=\int_{M}\la\eta_1\mid\eta_2\ra_{\omega}dV_{\omega} \quad \text{\rm for all }\, \eta_1,\eta_2 \in \Omega^{(0,q)}_c(M).    
 \end{equation}
 Let $L^2_{\omega}(M,\TM)$ be the completion of $\Omega^{(0,q)}_c(M)$ with respect to the inner product $\left(\cdot|\cdot\right)_{\omega}$ and denote by $\|\cdot\|_{\omega}$ the corresponding norm. \par 
 For an open set $U\subset M$, we define the restriction of the $L^2$-inner product by  \begin{equation}\label{2.2 inner product 0 restriction}
 \left(\eta_1\mid\eta_2\right)_{\omega,U}:=\int_U\la\eta_1\mid\eta_2\ra_{\omega}dV_{\omega} \quad \text{\rm for all }\,\eta_1,\eta_2\in\Omega^{(0,q)}_c(U).    
 \end{equation} In the same manner, we can define $L^2_{\omega}(U,\TM)$ to be the completion of $\Omega^{(0,q)}_c(U)$ with respect to $\left(\cdot|\cdot\right)_{\omega,U}$ and denote $\|\cdot\|_{\omega,U}$ to be the corresponding norm. \par 
A complex vector bundle $E \rightarrow M $ is holomorphic if the transition functions are holomorphic. We define $\Omega^{(p,q)}(M,E)$ as the space of $E$-valued smooth $(p,q)$-forms on $M$. Let $\pb^{E} : \Omega^{(p,q)}(M,E) \rightarrow \Omega^{(p,q+1)}(M,E)$ be the Cauchy-Reimann operator with values in $E$ (cf.\cite{Griff}, p.70). Given a Hermitian inner product $h^E$ on $E$, a connection $\nabla^E : \Cinf(M,E) \rightarrow \Cinf(M,T^*_\Cs M\otimes E)$ is called a \textbf{Chern connection} compatible with $h^E$ if for all $X \in T_\Cs M$ , $Y \in T^{(0,1)}M$ and $s_i \in \Cinf(M,E)$, the following two properties hold:
 \begin{enumerate}[label=(\roman*)]
     \item $d\la s_1 | s_2 \ra_{h^E}(X) = \la \nabla^E_X s_1 |s_2\ra_{h^E}+\la s_1 | \nabla^E_{\overline{X}} s_2\ra_{h^E}$;\label{2.2 condition 1}
     \item $\nabla^E_Y s=(\pb^E s)(Y)$.\label{2.2 condition 2} 
 \end{enumerate}
Given a holomorphic local frame $\{s_i\}$, we locally have connection 1-forms $\theta^i_j$ such that $\nabla^Es_i=\sum_{j}\theta^i_j \otimes s_j$. 
Denote $h^E_{i,j}:=\la s_i | s_j \ra_{h^E}$. Observe that the condition \ref{2.2 condition 1} implies $dh^E_{i,j}=\theta^i_k h^E_{k,j}+\overline{\theta^j_k}h^E_{i,k}$ and condition \ref{2.2 condition 2} implies that $\theta^i_j$ are $(1,0)$-forms. After combining the above conditions, we solve the matrix of connection $1$-forms by 
$[\theta^i_j] = [\p h^E_{i,j}][h^E_{i,j}]^{-1}$. In conclusion, there exists a \textbf{unique Chern connection with respect to a holomorphic Hermitian vector bundle $(E,h^E)$} 
 (cf.\cite{Griff}, p.73). \par

Now, we can define the curvature $\Theta^E \in \Omega^{(1,1)}(M,E)$ as the anti-symmetrization of $(\nabla^E)^2:\Cinf(M,E) \rightarrow \Cinf(M,T^*_\Cs M \otimes T^*_\Cs M \otimes E)$\footnote{We refer the readers to chapter 1 of \cite{Griff} for details on the elementary objects of complex manifolds.}. 
 If we fix a holomorphic local frame and denote $[\theta^{E}]$ as the matrix of connection forms with respect to the local frame, $\Theta^E$ can be locally expressed as a $\text{\rm rank } E \times \text{\rm rank } E$ 
 matrix of $2$-forms by $[\Theta^E]=[d\theta^E]-[\theta^E] \wedge [\theta^E]$.\par
 Recall that a holomorphic Hermitian line bundle $(L,h^E)$ is a 1-dimensional holomorphic Hermitian vector bundle. Let $(U,s)$ be a local trivialization where $U$ is a holomorphic chart and $s:U \subset M \rightarrow L $ is a holomorphic local non-vanishing section. Then there exists a local weight $\phi : U \rightarrow \R$ such that $\la s | s \ra_{h^L}=e^{-2 \phi(z)}$ . The  Chern connection is locally given by the connection 1-form $\theta=-2\p \phi$ and the curvature $\Theta^L$ is locally given by the $(1,1)$-form \[\Theta^L=- 2\pb \p \phi=2\sum_{i,j=1}^n\dfrac{\p^2 \phi}{\p z^i \p \zb^j}dz^i\wedge d\zb^j.\]
 Define $\dot{\Theta}^L\in\Cinf(M,\End(T^{(1,0)}M))$ to be the curvature operator such that
 \begin{equation*}
\langle \dot{\Theta}^L(p)v_1\mid v_2\rangle_{\omega}=\Theta^L(p)(v_1\wedge\ov{v_2}) \quad \text{\rm for all }\, v_1, v_2\in T^{(1,0)}_pM ,\quad p\in M.
\end{equation*}
Now, we introduce a lemma that allows us to simplify the information on curvature.
\begin{lem}\label{2.2 coordinate chart}(cf. \cite[Lemma III,2.3]{for coordinate})
Let $L \rightarrow M$ be a holomorphic line bundle over a complex manifold $M$. For any fixed $p \in M$, there exists a trivialization $(U,s)$ where $U \subset \Cn$ is a holomorphic chart centered at $p$ and $s:U \rightarrow L$ is a non-vanishing holomorphic section such that the Hermitian form $\omega$ and the local weight $\phi$ with respect to $s$ can be written as 
\[
\omega(z)=\ii\sum_{i=1}^{n}dz^i \wedge d\zb^i+O(|z|)\quad;\quad \phi(z)=\sum_{i=1}^{n}\lambda_i|z^i|^2+O(|z|^3).
\]
\end{lem}
\begin{rmk} \label{rmk coordinate chart}
    If $\lambda_i\neq 0$ for all $i=1,\cdots,n$, then the trivialization in Lemma \ref{2.2 coordinate chart} can be chosen such that \[\phi(z)=\sum_{i=1}^{n}\lambda_i|z^i|^2+O(|z|^4).
    \] 
\end{rmk}
Observe that the Hermitian metric $h^L$ on $L$ can be identified by a family of local wights $\{\phi_i\}$ with respect of a family of trivializing sections $\{s_i\}$. We will alternatively denote $\la \cdot | \cdot \ra_{\phi}:=\la \cdot | \cdot \ra_{h^L}$ if there is no risk of ambiguity. 
We can define the $k$-th tensor power of $L$ as $L^k:=L^{\otimes k}$, and denote the corresponding trivializing section as $s^k := s^{\otimes k}$. It follows that the local weight of $s^k$ with respect to the induced metric $h^{L^k}$ is given by $k\phi$. The norm of $s^{k}$ is  $|s^k|_{k\phi}:=|s^k|_{h^{L^k}}=e^{-k\phi}$.\par 
Fix a positive Hermitian $(1,1)$-form $\omega$ on $M$ and a Hermitian metric $h^L$ with local weights $\phi$ on $L$. They induce a pointwise Hermitian inner product $\la \cdot | \cdot \ra_{\omega,k\phi}$ on the bundle $\TM \otimes L^k$. For a fixed trivialization $s:U\To L$ with local weight $\phi$, if $u_1=\eta_1\otimes s^k$ and $u_2=\eta_2\otimes s^k$ where $\eta_i\in \Omega^{(0,q)}(U)$, then
\begin{equation*}
 \la u_1 \mid u_2 \ra_{\omega,k\phi}=\la \eta_1\otimes s^k \mid \eta_2\otimes s^k \ra_{\omega,k\phi}=\la \eta_1 \mid \eta_2 \ra_{\omega}e^{-2k\phi}.   
\end{equation*}
Then we can define the $L^2$-inner product on the space $\Omega^{(0,q)}_{c}(M,L^k)$ by
\begin{equation}\label{2.2 inner product 1}
 \left(u_1 \mid u_2\right)_{\omega,k\phi}:=\int_{M}\la u_1 \mid u_2 \ra_{\omega,k\phi}dV_{\omega} 
\quad\text{\rm for }\, u_1,u_2\in\Omega^{(0,q)}_c(M,L^k). 
\end{equation}
Denote $\|u\|^2_{\omega,k\phi}:=\left(u \mid u\right)_{\omega,k\phi}$ and $L^2_{\omega,k\phi}(M,\TM \otimes L^k)$ as the Hilbert space which is the completion of $\Omega^{(0,q)}_{c}(M,L^k)$ with respect to the inner product $\left(\cdot|\cdot\right)_{\omega,k\phi}$.

\subsection{The spectral and Bergman kernels}\label{section 2.3}
Let $\pb^{(q)}_k: \Omega^{(0,q)}(M,L^k) \rightarrow \Omega^{(0,q+1)}(M,L^k)$ be the Cauchy-Riemann operator and $
    \pb^{*,(q+1)}_k: \Omega^{(0,q+1)}(M,L^k) \rightarrow \Omega^{(0,q)}(M,L^k)$ be its formal adjoint with respect to $(\cdot | \cdot)_{\omega,k\phi}$. The \textbf{Kadaira Laplacian} is given by 
\begin{equation*}
\Box^{(q)}_{k}=\Box^{(q)}_{\omega,k\phi}:=\pb^{*,(q+1)}_{k} \pb^{(q)}_k+ \pb^{(q-1)}_k \pb^{*,(q)}_{k} : \Omega^{(0,q)}(M,L^k) \rightarrow \Omega^{(0,q)}(M,L^k).
\end{equation*}
Next, we define \[\Dom \, \pb^{(q)}_k := \{u \in L^2_{\omega,k\phi}(M,\TM \otimes L^k) ; \pb^{(q)}_k u \in L^2_{\omega,k\phi}(M,T^{*,(0,q+1)}M \otimes L^k)\}\] where $\pb^{(q)}_k u$ is defined in distribution sense. Then we can extend $\pb^{(q)}_k$ as
\begin{equation*}
    \pb^{(q)}_k: \Dom \, \pb^{(q)}_k \subset L^2_{\omega,k\phi}(M,\TM\otimes L^k) \rightarrow L^2_{\omega,k\phi}(M,T^{*,(0,q+1)}M \otimes L^k).
\end{equation*}
 Let
$\pb^{*,(q+1)}_k: \Dom \, \pb^{*,(q+1)}_k \subset L^2_{\omega,k\phi}(M,T^{*,(0,q+1)}M \otimes L^k) \rightarrow L^2_{\omega,k\phi}(M,T^{*,(0,q)}M \otimes L^k)$ be the $L^2_{\omega,k\phi}$-adjoint of $\pb^{(q)}_k$ and denote
\[
\Dom \Box^{(q)}_{k}:=\{u\in \Dom \pb^{(q)}_k \cap \Dom \pb^{*,(q)}_k\mid \pb^{(q)}_k u\in \Dom \pb^{*,(q+1)}_k\,\text{\rm and }\,\pb^{*,(q)}_k u\in \Dom \pb^{(q-1)}_k\}. 
\]
We have the \textbf{Gaffney extension}(cf.\cite{Gaffney})
\begin{equation}\label{2.3 Gaffney extansion}
    \Box^{(q)}_k: \Dom \,\Box^{(q)}_{k}\subset L^2_{\omega,k\phi}(M,\TM\otimes L^k) \To L^2_{\omega,k\phi}(M,\TM\otimes L^k).
\end{equation}
It is well-known that the extension is semi-positive and self-adjoint(cf.\cite[proposition 3.1.2]{book01}). Next, we introduce the spectral theorem.
\begin{thm}\label{2.5 spectrum theorem for continuous case}(\cite[theorem 2.5.1]{Davies})
    Let $A:\Dom A\subset \mathcal{H} \rightarrow \mathcal{H} $ be a self-adjoint operator on a Hilbert space $\mathcal{H}$. Then there exists a spectrum set $\text{\rm Spec}A \subset \R$ , a finite measure $\mu$  on $\text{\rm Spec}A \times \N$ and a unitary operator \begin{equation*} H:\mathcal{H} \rightarrow L^2(\text{\rm Spec}A \times \N,d \mu)
    \end{equation*}
    with the following properties: Set $h:\text{\rm Spec}\,A \times \N \rightarrow \R$ by $h(s,n):=s$. Then an element $f \in \mathcal{H}$ is in $\Dom A$ if and only if $h \cdot H(f) \in L^2(\text{\rm Spec}A \times \N,d \mu)$. In addition, we have \begin{equation*} Af = H^{-1}\circ (h \cdot Hf) \quad \text{\rm for all }\, f \in \Dom\,A. 
    \end{equation*}
\end{thm}
By Theorem \ref{2.5 spectrum theorem for continuous case}, we know that $\Box_k^{(q)}$ has the spectrum set $\text{\rm Spec}\,\Box_k^{(q)}$ that lies in $[0,\infty)$ since $\Box^{(q)}_{k}$ is semi-positive. Moreover, there is a unitary map \[H_k:L^2_{\omega,k\phi}(M,\TM\otimes L^k) \rightarrow L^2(\text{\rm Spec }\Box_k^{(q)} \times \N,d\mu_k)\] such that $\Box_k^{(q)}u=H_k^{-1} \circ (h \cdot H_ku)$ for all $u \in \Dom \Box_k^{(q)}$. \par
Given nonnegative constants $c_k$, we define the \textbf{spectral projections} by \begin{equation}\label{2.3 spectral projection}
\mathcal{P}_{k,c_k}^{(q)}u:= H_k^{-1} \circ (\mathbbm{1}_{[0,c_k] \times \N} \cdot H_ku) 
\end{equation}where  $\mathbbm{1}_{[0,c_k] \times \N}$ is the indicator function defined on $\text{\rm Spec}\,\Box_k^{(q)} \times \N$ by 
\[\begin{cases}
 \mathbbm{1}_{[0,c_k] \times \N}(s,l)=1 \quad \text{\rm if }\,s \in [0,c_k];\\
 \mathbbm{1}_{[0,c_k] \times \N}(s,l)=0\quad \text{\rm if }\,s \notin [0,c_k].
\end{cases}\] Clearly, $\mathcal{P}_{k,c_k}^{(q)}$ is an orthogonal projection since $H_k$ is a unitary map. In fact, the construction of $\mathcal{P}^{(q)}_{k,c_k}$ above coincides with the functional calculus $\mathbbm{1}_{[0,c_k]}(\Box^{(q)}_{k})$ with respect to the indicator function $\mathbbm{1}_{[0,c_k]}$ (cf.\cite[section 2]{Davies}). We may denote by $E^{(q)}_{k,\leq c_k}$ the image of $\mathbbm{1}_{[0,c_k]}(\Box^{(q)}_{k})$ and then \[
\Pk=\mathbbm{1}_{[0,c_k]}(\Box^{(q)}_{k}): L^2_{\omega,k\phi}(M,\TM\otimes L^k)\To E^{(q)}_{k,\leq c_k}.
\]  
For the case $c_k=0$, we denote  
\[
\mathcal{B}_k^{(q)}:=\mathcal{P}_{k,0}^{(q)}:L^2_{\omega,k\phi}(M,\TM \otimes L^k) \rightarrow \Ker\Box_k^{(q)}
\]to be the \textbf{Bergman projection}.
To introduce the spectral and Bergman kernels, we need the following theorem (cf.\cite[section 5.2]{Hormander}). Denote by $\mathcal{D}^{'}_{c}(M,E)$ the space of distribution sections of a vector bundle $E$ over $M$ whose elements have compact support in $M$.
\begin{thm}[Schwartz Kernel Theorem for smoothing operators]\label{2.3 kernel theorem}
Let $E$ and $F$ be two vector bundles on a manifold $M$ with a volume form $dV$. Then for any continuous linear operator $\mathcal{P}: \mathcal{D}'_c(M,E) \rightarrow \Cinf(M,F)$, there exists a unique smooth kernel $K_\mathcal{P} \in \Cinf(M \times M , F \boxtimes E)$ such that 
\[
\mathcal{P}u(x_0)=\int_{M}K_\mathcal{P}(x_0, y)(u(y))dV(y)
\]for all $u \in \mathcal{D}'_c(M;E)$. Here, we denote $F \boxtimes E$ as a vector bundle on $M \times M$ whose fiber at $(x,y) \in M \times M $ is the space of linear transformations from $E_x$ to $F_y$. 
\end{thm}
Since Kodaira Laplacian is elliptic, the spectral projection $\Pk$ and the Bergman projection $\Bk$ are smoothing operators in the sense that \begin{align*}
 \Pk&:\mathcal{D}^{'}_{c}(M,\TM \otimes L^k)\To \Cinf(M,\TM\otimes L^k), \\ \Bk&:\mathcal{D}^{'}_{c}(M,\TM \otimes L^k)\To \Cinf(M,\TM\otimes L^k)  
\end{align*} are continuous maps. In conclusion, the conditions of Theorem \ref{2.3 kernel theorem} hold for $\Pk$ and $\Bk$ and hence their distribution kernels are smooth. 

\begin{defin}
    Define the \textbf{spectral kernel} $\Pkk(z,w)$ and \textbf{Bergman kernel} $\Bkk(z,w)$ which are in $\Cinf\left(M\times M,(T^{*,(0,q)}M\otimes L^k)\boxtimes(T^{*,(0,q)}M\otimes 
 L^{k})\right)$ to be the Schwartz kernels of the spectral projection $\Pk$ and Bergman projection $\Bk$, respectively. In this way, for all $u\in L^2_{\omega,k\phi}(M,\TM\otimes L^k)$, we have
    \begin{align*}
        \Pk u(z)=\int_M\Pkk(z,w)u(w)dV_{\omega}(w)\quad;\quad
        \Bk u(z)=\int_M\Bkk(z,w)u(w)dV_{\omega}(w).
    \end{align*}
\end{defin}

\subsection{The Sobolev and Gårding inequalities}\label{section 2.4}
In this section, we consider the $\binom{n}{q}$-dimensional trivial complex vector bundle $\T$ over an open subset $U$ of $\Cn$ with a global trivializing frame $\{d\zb^I\}_{I\in\mathcal{J}_{q,n}}$. Let 
\begin{equation*}u:=\sum_{I\in \mathcal{J}_{q,n}}u_I d\zb^I \in \Omega^{(0,q)}_c(U)\end{equation*} be a smooth section of $\T$. We consider $u$ as a smooth vector-valued function \[(u_I)_{I\in\mathcal{J}_{q,n}}:U\subset \Cn\simeq \R^{2n} \To \Cs^{\binom{n}{q}}\] by fixing an order of $\mathcal{J}_{q,n}$. Recall that the Fourier transform of $u=(u_I)_{I\in\mathcal{J}_{q,n}}$ is  \[\widehat{u}(\xi) := (\widehat{u}_I(\xi))_{I\in\mathcal{J}_{q,n}}\] where $\widehat{u}_I(\xi):=(2\pi)^{-n/2}\int _{\R^{2n}}u_I(x)e^{-i\xi \cdot x}dm$. For any $s \in \R$, the Sobolev $s$-norm $\|\cdot\|_{s,U}$ is \begin{equation}\label{2.4 sobolev norm}
\|u\|^2_{s,U}=\|u\|^2_{s}:=\int_{\R^{2n}}(1+|\xi|^2)^s|\widehat{u}(\xi)|^2 dm(\xi).    
\end{equation} The \textbf{Sobolev space} $H_s(U,\T)$ is the completion of $\Omega^{(0,q)}_c(U)$ with respect to the norm $\|\cdot\|_s$. Since $\int_{\R^{2n}}|\widehat{u_I}|^2 dm=\int_{\R^{2n}}|u_I|^2 dm$ for all $I\in\mathcal{J}_{q,n}$, we have $\|\cdot\|_{0,U}=\|\cdot\|_{0}=\|\cdot\|_{dm}=\|\cdot\|_{dm,U}$ where $\|\cdot\|_{dm}$ is the $L^2$-norm with respect to standard Lebesgue measure $dm$ in Euclidean space. 
The following proposition induces a variant of the Sobolev norm.
\begin{prop}[compatibility]\label{2.4 prop 0}
Given $u=\sum_{I\in\mathcal{J}_{q,n}}u_Id\zb^I \in \Omega^{(0,q)}_c(U)$ and $s\in\N$, there exist positive constants $C_1$ and $C_2$ 
 independent of $u$ such that
\[
C_1\sum_{I\in\mathcal{J}_{q,n}}\sum_{|\alpha|\leq s}\|\p^{\alpha}_xu_{I}\|^2_0\leq \|u\|^2_{s} \leq C_2\sum_{I\in\mathcal{J}_{q,n}}\sum_{|\alpha|\leq s}\|\p^{\alpha}_xu_{I}\|^2_0.  
\]
\end{prop}
The proof of Proposition \ref{2.4 prop 0} is simply by the fact that $|\widehat{\p^{\alpha}_xu_I}(\xi)|=|\xi^{\alpha}\widehat{u_I}(\xi)|$ where $\xi^{\alpha}:=\xi_1^{\alpha_1}\cdots \xi_{2n}^{\alpha_{2n}}$. Next, we introduce a basic proposition in the Sobolev theory.
\begin{prop}\label{2.4 prop}
For any $s\in\R$, $H_{-s}(U,\T)$ is the dual space of $H_{s}(U,\T)$ and  \[
|\left(u\mid v\right)_0|\leq\|u\|_{-s}\|v\|_{s}
\]for all $u\in H_{-s}(U,\T)$ and $v\in H_{s}(U,\T)$.   
\end{prop}
Let $\mathcal{C}^d(U,\T)$ be the space of $d$-th differentiable sections. For any point $x\in U$, define 
\begin{equation}\label{2.4 C pt norm for 0q form}
|u|^2_{\mathcal{C}^d}(x):= \sum_{I\in\mathcal{J}_{q,n}}\sum_{|\alpha| \leq d}|\p^{\alpha}u_I(x)|^2\quad \text{for all }\,u\in \mathcal{C}^d(U,\T).    
\end{equation}
Next, define a norm $\|\cdot\|_{\mathcal{C}^d(U)}$ on the space $\mathcal{C}^d(U,\T)$ by \begin{equation}\label{2.4 C space}
\|u\|^2_{\mathcal{C}^d(U)}:= \sup_{x\in U}|u|^2_{\mathcal{C}^d}(x).   
\end{equation} The following theorem is well-known and will be applied in Section \ref{section 3.3}.
\begin{thm}[Sobolev inequality] Let $d\in\N_0$ and $s\in\R$ such that $s > d+n$. If $u \in H_s(U,\T)$, then $u \in \mathcal{C}^d(U,\T) $ and there exists a constant $C_{s,d}$ independent of $u$ such that 
\begin{equation*}
    \|u\|_{\mathcal{C}^d(U)} \leq C_{s,d} \|u\|_{s}. 
\end{equation*} 
\end{thm}

We now consider a second-order differential operator $P:\Omega^{(0,q)}(U)\To \Omega^{(0,q)}(U)$.\footnote{A basic reference for this section is \cite{Gil}, sections 1.1-1.3} By ordering the basis $\{d\zb^I\}_{I\in\mathcal{J}_{q,n}}$, we can treat $P$ as a $\binom{n}{q}\times\binom{n}{q}$ matrix $[P_{i,j}]$ of second-order differential operators $P_{i,j}:\Cinf(U)\To \Cinf(U)$. Let $(x^1,\cdots,x^{2n})$ be the standard coordinate on $\R^{2n}\simeq \Cn$. We can represent $P_{i,j}$ as \[P_{i,j}=\sum_{|\alpha| \leq 2}a_{i,j,\alpha}(x)\p^{\alpha}_x\quad\text{\rm where }a_{i,j,\alpha}\in\Cinf(U).\]Define the symbol $\sigma(P_{i,j})$ of $P_{i,j}$ by \[\sigma(P_{i,j})(x,\xi) := \sum_{|\alpha|\leq 2} \ii^{|\alpha|} a_{i,j,\alpha}(x)\xi^{\alpha}\quad\text{\rm where }\, x\in U\, , \, \xi\in\R^{2n}.\] $\sigma(P_{i,j})$ is a polynomial of $\xi$ of degree $2$ for any fixed  $x\in U$. Furthermore, we define the symbol $\sigma(P)(x,\xi)$ of $P$ as the $\binom{n}{q}\times\binom{n}{q}$ matrix $[\sigma(P_{i,j})(x,\xi)]$ of polynomials of $\xi$ for any $x\in U$. 
\begin{defin}[elliptic operator]
 We call a second-order differential operator $P:\Omega^{(0,q)}(U)\To\Omega^{(0,q)}(U)$ is \textbf{uniform elliptic} on $U$ if there exists $C>0$ such  that   
\begin{equation}\label{2.3 positivity}
   |\sigma(P)(z,\xi)v|\geq C|\xi|^2|v| \quad \forall x\in U,\xi\neq 0\,\,\text{and}\,\,v\in\R^{2n}.\end{equation}   
\end{defin}
The following theorem plays an important role in the proof of local uniform bound in Section \ref{section 3.3}.
\begin{thm}[Gårding inequality]\label{2.4 Gårding inequality}
Let $P$ be a second-order differential operator which is uniform elliptic on an open set $U\subset \subset \Cn$. Then for any $m\in\N$, there exists a positive constant $\Tilde{C}$ such that 
\begin{equation*}
     \|u\|_{2m,U} \leq \Tilde{C}\left(\|u\|_{0,U}+\|P^m u\|_{0,U}\right) \quad \text{\rm for all }u\in H_{2m}(U,\T).
\end{equation*}

\end{thm}

\begin{rmk}
The settings of this section can be modified to any trivial vector bundles. In particular, the cotangent bundle $T^{*,(p,q)}\Cn\To \tU\subset \Cn$ with the trivializing frame
$\{dz^I\wedge d\zb^J\}_{I\in\mathcal{J}_{p,n},J\in\mathcal{J}_{q,n}}$.  For example, if $u=\sum_{i,j=1}^{n}u_{i,j}dz^i\wedge d\zb^j$ is a smooth $(1,1)$-form with compact support in $\Cn$, we can define 
\begin{equation}\label{2.4 C norm for 11 form}
|u|^2_{\mathcal{C}^d}(x):=\sum_{i,j=1}^{n}\sum_{|\alpha|\leq d}|\p^{\alpha}u_{i,j}(x)|^2   
\end{equation}
and the norm 
\[
\|u\|^2_{\mathcal{C}^d(U)}:=\sup_{x\in U} |u|^2_{\mathcal{C}^d}(x).
\]
\end{rmk}

\newpage

\section{The local uniform bounds for scaled spectral and Bergman kernels}\label{chapter 3}
In this chapter, our aim is to analyze the behavior of the scaled spectral and Bergman kernels. Our objective is to establish local uniform bounds on the scaled kernels, which will allow us to investigate their local convergence properties. To this end, we will apply the Arzelà–Ascoli theorem.\par

In Section \ref{section 3.1}, we recall the set-up which has been mentioned in Section \ref{section 1.1} and construct the scaled bundles. In Section \ref{section 3.2}, we compute the Kodaira Laplacian on the trivial line bundle and apply the results on the cases of scaled bundles. In Section \ref{section 3.3}, under the framework set in Sections \ref{section 3.1} and \ref{section 3.2}, we can eventually control the local behavior of scaled spectral and Bergman kernels by the analytic tools of Sobolev theory. 
\subsection{The scaled bundles}\label{section 3.1} 
Let $(M,\omega)$ be a Hermitian manifold of dimension $n$ and $(L,h^L) \rightarrow M$ be a holomorphic Hermitian line bundle. Given a non-vanishing holomorphic section $s$ of $L$ on a holomorphic chart $U$ that trivializes $L$, there exists a local weight $\phi:U \rightarrow \mathbb{R}$ such that $|s|^2_{h^L} = e^{-2\phi}$. Denote $|s|_{\phi}:=|s|_{h^L}$ for convenience. \par
Recall that $\left(\cdot|\cdot\right)_{\omega}$ is the $L^2$-inner product of the Hilbert space $L^2_{\omega}(M,\TM)$ (cf.(\ref{2.2 inner product 0})) and $\left(\cdot|\cdot\right)_{\omega,U}$ is the restriction of the inner product (cf.(\ref{2.2 inner product 0 restriction})). Also, we can define another Hilbert space $L^2_{\omega,k\phi}(M,\TM\otimes L^k)$ which has the inner product   $\left(\cdot|\cdot\right)_{\omega,k\phi}$(cf.(\ref{2.2 inner product 1})).\par  
 Denote $\pb^{(q)}_k:\Omega^{(0,q)}(M,L^k)\rightarrow \Omega^{(0,q+1)}(M,L^k)$ to be the \textbf{Cauchy-Riemann operator} and $\pb^{*,(q+1)}_k:\Omega^{(0,q+1)}(M,L^k) \rightarrow \Omega^{(0,q)}(M,L^k)$ to be the formal adjoint of $\pb_k$ with respect to $(\cdot|\cdot)_{\omega,k\phi}$. Recall that we have the  \textbf{Kodaira Laplacian} $\Box^{(q)}_k$ (or $\Box^{(q)}_{\omega,k\phi}$) given by \[\Box^{(q)}_k:=\pb^{*,(q+1)}_k\pb^{(q)}_k+\pb^{(q-1)}_k\pb^{*,(q)}_k:\Dom \Box^{(q)}_k\To L^2_{\omega,k\phi}(M,\TM\otimes L^k)\] which is the \textbf{Gaffney extension}(cf.(\ref{2.3 Gaffney extansion})).\par

Fix a nonnegative sequence $c_k$. Denote by $E^{(q)}_{k,\leq c_k}$ the image of the functional calculus $\mathbbm{1}_{[0,c_k]}(\Box^{(q)}_{k})$ with respect to the indicator function $\mathbbm{1}_{[0,c_k]}$. The \textbf{spectral projection} is the orthogonal projection
 \[
 \Pk:=\mathbbm{1}_{[0,c_k]}(\Box^{(q)}_{k}):L^2_{\omega,k\phi}(M,\TM \otimes L^{k})\To E^{(q)}_{k,\leq c_k}
 \]and the \textbf{Bergman projection} is the orthogonal projection 
\[
\Bk:= \mathcal{P}^{(q)}_{k,0} : L^2_{\omega,k\phi}(M,\TM\otimes L^k)\to \Ker\Box^{(q)}_{k}
\] which is a special case of spectral projection by taking $c_k=0$. Readers may consult (\ref{2.3 spectral projection}) for an explicit definition.\par   

 For $\eta_1 \otimes s^k$ and $\eta_2 \otimes s^k$ in $L^2_{\omega,k\phi}(U,T^{*,(0,q)}M\otimes L^k)$ where $\eta_i \in \Omega^{(0,q)}_c(U)$, observe that 
\[
(\eta_1 \otimes s^k | \eta_2 \otimes s^k)_{\omega,k\phi}=\int_{U}\la \eta_1 |\eta_2 \ra_{\omega}e^{-2k\phi}dV_{\omega}=(\eta_1 e^{-k\phi}| \eta_2 e^{-k\phi})_{\omega,U}.
\]This induces a unitary identification   
\begin{equation}\label{3.1 id 1}
   L^2_{\omega,k\phi}(U,T^{*,(0,q)}M\otimes L^k) \cong L^2_{\omega}(U,T^{*,(0,q)}M) \quad \text{\rm by }\, 
   \eta \otimes s^k \leftrightarrow \eta e^{-k\phi}. 
\end{equation}

Define the \textbf{localized spectral projection}  
\[
    \Pks:L^2_{\omega}(U,T^{*,(0,q)}M) \rightarrow L^2_{\omega}(U,T^{*,(0,q)}M)
\]satisfying $\Pk(\eta \otimes s^k)=e^{k\phi}\Pks(\eta e^{-k\phi})\otimes s^k$ for all $\eta \in L^2_{\omega}(U,T^{*,(0,q)}M)$. In the case $c_k=0$, we denote $\Bks:= \mathcal{P}^{(q)}_{k,0,s}$ as the \textbf{ localized Bergman projection}.\par  

Next, let $\Pkk(z,w)$ and $\Bkk(z,w)$ be the \textbf{spectral} and \textbf{Bergman kernels} which are the Schwartz kernels of $\Pk$ and $\Bk$, respectively. We may also define the \textbf{localized spectral kernel} $\Pksk(z,w)$ and \textbf{localized Bergman kernel} $\Bksk(z,w)$ to be the Schwartz kernels of $\Pks$ and $\Bks$, respectively. The relation between $P^{(q),s}_{k,c_k}(z,w)$ and $P^{(q)}_{k,c_k,s}(z,w)$ is given by
\begin{equation}\label{kernel relation}
P^{(q)}_{k,c_k,s}(z,w)=P^{(q),s}_{k,c_k}(z,w)\cdot|s^k(z)|_{h^L}\cdot|(s^k)^*(w)|_{h^{L^*}}=e^{-k\phi(z)}P^{(q),s}_{k,c_k}(z,w)e^{k\phi(w)},    \end{equation}
where $P^{(q),s}_{k,c_k}(z,w)$ is defined in (\ref{8.1}).
\par 
From now on, we fix a point $p \in M$ throughout this paper and apply Lemma \ref{2.2 coordinate chart} to obtain a trivialization $(U,s)$ centered at $p$ such that 
\begin{equation*}
\phi(z)=\sum_{i=1}^n\lambda_i|z|^2+O(|z|^3)\quad\text{\rm and }\quad\omega(z)=\ii\sum_{i=1}^{n}dz^i\wedge d\zb^i +O(|z|).
\end{equation*}Recall that the set of points with signature $q$ is defined by
\begin{equation*}
  M(q):=\{ p'\in M \mid   \dot{\Theta}^L(p') \, \text{\rm is non-degenerate and has exactly $q$ negative eigenvalues}\}  
\end{equation*}and observe that $p\in M(q)$ means $q=\#\{i\mid \lambda_i<0\}$ and $n-q=\#\{i\mid \lambda_i>0\}$. In the case of $p\in M(0)$, we choose the trivialization  such that $\phi=\sum_{i=1}^{n}\lambda_i|z|^2+O(|z|^4)$ throughout this paper. Without loss of generality, we assume $B(1)\subset U \subset \Cn$ and make the following observations. Denote
\begin{equation*}
  \phi_0(z):=\sum_{i=1}^n \lambda_i |z^i|^2\quad\text{\rm and }\quad
  \phi_{(k)}(z):=k\phi(z/\sqrt{k}).
\end{equation*}Then for any $\epsilon<1/2$, there exists a constant $C$ independent of $k$ such that
\begin{equation}\label{3.1 ob phi}
 |\phi_{(k)}-\phi_0|_{\Ct}(z) \leq C\dfrac{|z|^3+1}{\sqrt{k}} \quad \forall \ |z|\leq k^{\epsilon}.   
\end{equation}where $|\cdot|_{\mathcal{C}^2}$ is defined in (\ref{2.4 C pt norm for 0q form}).
Also, set 
\begin{equation*}
\omega_0:=\ii\sum_{i}dz_i\wedge d\zb_i\quad\text{\rm and }\quad
\omega_{(k)}:=\omega(\dfrac{z}{\sqrt{k}}).
\end{equation*}Then there also exists a constant $C'$ such that 
\begin{equation}\label{3.1 ob omega}
   |\omega_{(k)}-\omega_0|_{\Ct}(z) \leq C'\frac{|z|+1}{\sqrt{k}} \quad \forall \ |z|\leq k^{\epsilon}.   
\end{equation} where $|\cdot|_{\mathcal{C}^2}$ for $(1,1)$-forms is defined in (\ref{2.4 C norm for 11 form}).
Furthermore, $\phi_{(k)}\To \phi_0$ and $\omega_{(k)}\To\omega_0$ locally uniformly in $\Cinf$ on $\Cn$.  $\phi_{(k)}$ and $\omega_{(k)}$ are defined on $B(\sqrt{k})$ and are called the \textbf{ scaled metric} and \textbf{scaled Hermitian form}, respectively. Inspired by the observations above, we construct the \textbf{scaled line bundles}. Define 
\[
 s^{(k)}(z):=s^{k}(\dfrac{z}{\sqrt{k}}):B(\sqrt{k}) \rightarrow L^k.   
\]
This makes $L^k$ a trivial line bundle over $B(\sqrt{k})$ with a trivializing section $s^{(k)}$ for any $k \in \N$. We denote the scaled line bundle as 
\[
    L^{(k)}:=L^k \rightarrow B(\sqrt{k}) \subset \Cn
\]which equipped with 
 the scaled metric $\phi_{(k)}$ by
\[\la s^{(k)} \mid s^{(k)} \ra_{\phi_{(k)}}:= e^{-2\phi_{(k)}}=e^{-2k\phi(z/\sqrt{k})}.\]
More generally, we consider the trivial vector bundle
\[
T^{*,(0,q)}{\Cn} \otimes L^{(k)} \rightarrow B(\sqrt{k}) \subset \Cn
\]which is a $\binom{n}{q}$-dimensional complex vector bundle with the space of smooth sections $\Omega^{(0,q)}(B(\sqrt{k}),L^{(k)})$ and  trivializing frames $\{d \zb^I \otimes s^{(k)}\}_{I\in\mathcal{J}_{q,n}}$.\par 
We endow the vector bundle $T^{*,(0,q)}\Cn \otimes L^{(k)}\To B(\sqrt{k})$ with a pointwise Hermitian structure by the scaled metric $\phi_{(k)}=k\phi(z/\sqrt{k})$ on $L^{(k)}$ and the scaled Hermitian form $\omega_{(k)}=\omega(z/\sqrt{k})$ on $T^{*,(0,q)}B(\sqrt{k})$. That is, for all $\eta_1,\eta_2 \in \Omega^{(0,q)}(B(\sqrt{k}))$,
\[
\la \eta_1 \otimes s^{(k)}| \eta_2 \otimes s^{(k)} \ra_{\omega_{(k)},\phi_{(k)}}(z) := \la \eta_1(z) | \eta_2(z) \ra_{\omega_{(k)}}\,e^{-2\phi_{(k)}(z)}.
\]
\par
Similar to the identification (\ref{3.1 id 1}), there is a unitary correspondence
\begin{align}\label{3.1 id 2}
L^2_{\omega_{(k)},\phi_{(k)}}(B(\sqrt{k}),T^{*,(0,q)}\Cn \otimes L^{(k)}) &\cong L^2_{\omega_{(k)}}(B(\sqrt{k}),T^{*,(0,q)}\Cn)\quad\text{\rm by }\\ \eta \otimes s^{(k)} &\leftrightarrow \eta e^{-k\phi(z/\sqrt{k})} \notag.    
\end{align}
\par 
In the meantime, by changing variable, there are unitary identifications
\begin{align}\label{3.1 id 3}
    L^2_{\omega,k\phi}(B(1),\T \otimes L^k) &\cong L^2_{\omega_{(k)},\phi_{(k)}}(B(\sqrt{k}),\T \otimes L^{(k)})\quad \text{by }\\
    \eta \otimes s^k &\leftrightarrow k^{-n/2}\eta(z/\sqrt{k}) \otimes s^{(k)} \notag 
\end{align}and
\begin{align}\label{3.1 id 4}
   L^2_{\omega}(B(1),\T) &\cong  L^2_{\omega_{(k)}}(B(\sqrt{k}),\T)\quad \text{by }\\
   \eta &\leftrightarrow k^{-n/2}\eta(z/\sqrt{k}). \notag
\end{align}So far, we have four unitary identifications (\ref{3.1 id 1}),(\ref{3.1 id 2}),(\ref{3.1 id 3}),(\ref{3.1 id 4}) between the spaces of sections. In fact, the identifications form a commutative diagram. We can transform the 
 localized spectral (or Bergman) kernels defined on $B(1)$ to kernels on the scaled bundles over $B(\sqrt{k})$ by (\ref{3.1 id 4}).\par 
Define the \textbf{scaled localized spectral projection} $\Plks:L^2_{\omega_{(k)}}(B(\sqrt{k}),\T) \To L^2_{\omega_{(k)}}(B(\sqrt{k}),\T)$ such that
\begin{equation}\label{3.1 scaled projection relation}
\left(\Plks u\right)(\sqrt{k}z)=\Pks \left(u(\sqrt{k}w)\right)    
\end{equation}
and the \textbf{scaled localized Bergman projection} $\Blks:=\mathcal{P}^{(q)}_{(k),0,s}$. Define the \textbf{scaled localized spectral kernel} $\Plksk(z,w)$ to be the Schwartz kernel of $\Plks$ which is given by 
\[
\Plksk(z,w)=k^{-n}P^{(q)}_{k,c_k,s}(\dfrac{z}{\sqrt{k}},\dfrac{w}{\sqrt{k}})
\]and the \textbf{scaled localized Bergman kernel} $\Blksk(z,w):=P^{(q)}_{(k),0,s}$. In this way, we have 
\[
\Plks u(z)=\int_{B(\sqrt{k})}\Plksk(z,w)u(w)dV_{\omega_{(k)}} \quad \text{\rm where }\,dV_{\omega_{(k)}}:=\dfrac{\omega_{(k)}^n}{n!}.
\]The relation between $P^{(q),s}_{(k),c_k}(z,w)$ and $P^{(q)}_{(k),c_k,s}(z,w)$ is given by 
\begin{equation}\label{kernel relation 2}
P^{(q)}_{(k),c_k,s}(z,w)=e^{-\phi_{(k)}(z)}P^{(q),s}_{(k),c_k}(z,w)e^{\phi_{(k)}(w)}    
\end{equation}where $P^{(q),s}_{(k),c_k}$ is defined in (\ref{8.2}).

\subsection{The Laplacians}\label{section 3.2}
The goal of this section is to compute Kodaira Laplacian on a trivial line bundle. We now temporarily forget the set-up in Section \ref{section 3.1}. Let $\tU$ be an open set in $\Cn$ and $\tL \rightarrow \tU$ be a trivial line bundle over $\tU$ with a trivializing section $\ts$. Fix a positive Hermitian $(1,1)$-form $\tomega$ on $\tU$ and a weight function $\tphi$ such that $\la \ts | \ts \ra_{\tphi}=e^{-2\tphi}$. Consider \[\T \otimes \tL \rightarrow \tU \subset \Cn\] to be the Hermitian vector bundle with the  pointwise Hermitian structure $\la\cdot|\cdot\ra_{\tomega,\tphi}$ induced by $\tomega$ and $\tphi$. That is, for $\eta_1,\eta_2 \in \Omega^{(0,q)}_c(\tU)$, 
\begin{equation}\label{3.2 inner product pt}
\la \eta_1\otimes \ts\mid \eta_2\otimes \ts \ra_{\tomega,\tphi}(z)=\la \eta_1(z)\mid\eta_2(z)\ra_{\tomega}e^{-2\tphi(z)} \quad \text{\rm for all }z\in\tU.     
\end{equation}
This defines an inner product on the space on $\Omega^{(0,q)}_c(\tU,\tL)$. Namely,
\begin{equation}\label{3.2 inner product}
\left(\eta_1\otimes \ts\mid\eta_2\otimes \ts\right)_{\tomega,\tphi}:=\int_{\tU}\la\eta_1\otimes\ts\mid\eta_2\otimes\ts\ra_{\tomega,\tphi}dV_{\tomega}\quad \text{where }\,dV_{\tomega}:=\dfrac{\tomega^n}{n!}.  
\end{equation}
Let $L^2_{\tomega,\tphi}(\tU,\T \otimes \tL)$ be the completion of $\Omega^{(0,q)}_c(\tU)$ with respect to $\left(\cdot|\cdot\right)_{\tomega,\tphi}$. Similarly, we have another Hilbert space $L^2_{\omega}(\tU,\T)$ with its inner product $(\cdot|\cdot)_{\tomega}$(cf.(\ref{2.2 inner product 0})). There is a unitary identification
\begin{align}\label{3.2 identification}
    L^2_{\tomega,\tphi}(\tU,\T \otimes \tL )  \cong  L^2_{\tomega}(\tU,\T) , \quad 
    \eta \otimes \ts \leftrightarrow \eta e^{-\tphi}.
\end{align}For any smooth $(0,1)$-form $\eta$ on $\tU$, we can consider the wedge operator $\eta \wedge \cdot: T^{*,(0,q)}_p\tU \rightarrow T^{*,(0,q+1)}_p\tU $. Moreover, for the positive Hermitian $(1,1)$-form $\tomega$, we let $\eta\wedge^*_{\tomega}\cdot: T^{*,(0,q+1)}_p\tU \rightarrow T^{*,(0,q)}_p\tU$ to be the adjoint of $\eta\wedge\cdot$ via the pointwise inner product $\la \cdot | \cdot \ra_{\tomega}$. For $\eta_1 , \eta_2 \in T^{*,(0,1)}_p\tU$, we have the identity $(\eta_1\wedge^*_{\tomega})\eta_2\wedge \cdot +\eta_2\wedge(\eta_1\wedge^*_{\tomega})\cdot=\la \eta_1 | \eta_2 \ra_{\tomega}\cdot.$\par
Let $\pb^{(q)}_{\tL}:\Omega^{(0,q)}(\tU,\tL) \rightarrow \Omega^{(0,q+1)}(\tU,\tL)$ be the Cauchy-Riemann operator with values in $\tL$ and $\pb^{*,(q+1)}_{\tL}:\Omega^{(0,q+1)}(\tU,\tL)\rightarrow \Omega^{(0,q)}(\tU,\tL) $ be the formal adjoint of $\pb^{(q)}_{\tL}$ with respect to $\left(\cdot|\cdot\right)_{\tomega,\tphi}$. Under identification (\ref{3.2 identification}), it is natural to define the \textbf{localized Cauchy-Riemann operator} $\pb_{\ts}:\Omega^{(0,q)}(\tU)\To\Omega^{(0,q+1)}(\tU)$ such that $\pb_{\tL}(\eta \otimes \ts)= e^{\tphi}\pb_{\ts}(\eta e^{-\tphi})\otimes \ts$. 
Denote by $\pb^{(q)}$ the standard Cauchy-Riemann operator on $\Omega^{(0,q)}(\tU)$. Note that 
\begin{equation}\label{3.2 witten relation}
    \pb^{(q)}_{\ts}=e^{-\tphi} \pb^{(q)} e^{\tphi}.  
\end{equation}By direct computation, we have
\begin{equation}\label{3.2 pbs eq}
\pb^{(q)}_{\ts}=\pb^{(q)}+(\pb \tphi)\wedge\cdot .
\end{equation}Of course, we can also define $\pb^{*,(q)}_{\ts}:\Omega^{(0,q-1)}(\tU) \rightarrow \Omega^{(0,q)}(\tU)$ satisfying $\pb^{*}_{\tL}(\eta \otimes \ts)=e^{\tphi}\ \pb^{*}_{\ts}(\eta e^{-\tphi})\otimes \ts$. Then $\pb^{*,(q)}_{\ts}$ is the formal adjoint of $\pb^{(q-1)}_{\ts}$ with respect to $\left(\cdot|\cdot \right)_{\tomega}$. Next, define $\pb^{*,(q)}_{\tomega}$ to be the formal adjoint of $\pb^{(q-1)}$ with respect to $\left( \cdot | \cdot \right)_{\tomega}$. Note that
\begin{equation}\label{3.2 pbs* eq}
\pb^{*,(q)}_{\ts}=\pb^{*,(q)}_{\tomega}+(\pb \tphi)\wedge^*_{\tomega}\cdot.
\end{equation} \par 
Recall that $\Box^{(q)}_{\tL}:=\pb^{*,(q+1)}_{\tL}\pb^{(q)}_{\tL}+\pb^{(q-1)}_{\tL}\pb^{*,(q)}_{\tL}$ is the Kodaira Laplacian. We can define the  \textbf{localized Kodaira Laplacian} $\Box^{(q)}_{\ts}:=\pb^{*,(q+1)}_{\ts}\pb^{(q)}_{\ts}+\pb^{(q-1)}_{\ts}\pb^{*,(q)}_{\ts}$ that acts on $\Omega^{(0,q)}(\tU)$. $\Box^{(q)}_{\tL}$ and $\Box^{(q)}_{\ts}$ are compatible under the identification (\ref{3.2 identification}) in the sense that
\begin{equation}\label{3.2 localized relation}
 \Box^{(q)}_{\tL}(\eta \otimes \ts)=e^{\tphi}(\Box^{(q)}_{\ts}\eta e^{-\tphi})\otimes \ts.   
\end{equation} We can consider the Gaffney extansions of $\Box^{(q)}_{\tL}$ and $\Box^{(q)}_{\ts}$ which preserve the relation (\ref{3.2 localized relation}) and $
\eta\otimes \ts \in \Dom\Box^{(q)}_{\tL} \Longleftrightarrow \eta e^{-\tphi}\in\Dom \Box^{(q)}_{\ts}
$. Next, we compute the localized Kodaira Laplacian using the settings above.\par

\begin{lem}\label{3.2 Lemma Lap}The localized Kodaira Laplacian can be expressed as
\begin{equation}\label{3.2 Lemma Lap simple Box eq}
\Box^{(q)}_{\ts}=\Delta_{\tomega}+\pb \left((\pb\tphi)\wedge^*_{\tomega}\cdot\right)+(\pb \tphi)\wedge^*_{\tomega}\pb+\pb^*_{\tomega}\left( (\pb\tphi)\wedge\cdot\right)+(\pb \tphi)\wedge\pb^*_{\tomega}+\la \pb \tphi|\pb \tphi \ra \cdot,    
\end{equation}
where $\Delta^{(q)}_{\tomega}:=\pb^{*,(q+1)}_{\tomega}\pb^{(q)}+\pb^{(q-1)}\pb^{*,(q)}:\Omega^{(0,q)}(\tU)\To\Omega^{(0,q)}(\tU)$ is the Hodge Laplacian with respect to $\tomega$. Furthermore, assume that $\tomega^n/n!=e^{\tvarphi}dm$ for some function $\tvarphi$ and let $\theta$ denote the matrix of connection forms of $\nabla$ on $T^{(q,0)}\tU$ with respect to the frame $\{dz^{I}\}_{I\in\mathcal{J}_{q,n}}$. Then for $f \in \Cinf(\tU)$ and $I \in \mathcal{J}_{q,n}$,
\begin{equation}\label{3.2 Lemma Lap pb* eq}
\pb^*_{\tomega}fd\zb^I=\left(-\dfrac{\p f}{\p z^i}-f\dfrac{\p \tvarphi}{\p z^i}\right)(d\zb^i)\wedge^*_{\tomega}d\zb^I-f \,(d\zb^i\wedge^*_{\tomega})\overline{{\theta}}_{\p/\p \zb^i}^*d\zb^I.
\end{equation}
\end{lem}

\begin{proof}
First, by (\ref{3.2 pbs eq}) and \ref{3.2 pbs* eq}),
\begin{align*}
    \Box_{\ts}&=\pb^*_{\ts}\pb_{\ts}+\pb_{\ts}\pb^*_{\ts}=\left(\pb^*_{\tomega}+(\pb \tphi)\wedge^*_{\tomega}\cdot\right)\left(\pb+(\pb \tphi)\wedge\cdot\right)+\left(\pb+(\pb \tphi)\wedge\cdot\right)\left(\pb^*_{\tomega}+(\pb \tphi)\wedge^*_{\tomega}\cdot\right)\\
    &=\Delta_{\tomega}+\pb \left((\pb\tphi)\wedge^*_{\tomega}\cdot\right)+\left((\pb \tphi)\wedge^*_{\tomega}\right)\pb+\pb^*_{\tomega}\left( (\pb\tphi)\wedge\cdot\right)+(\pb \tphi)\wedge\pb^*_{\tomega}+\la \pb \tphi|\pb \tphi \ra \cdot. 
\end{align*}
 Now, we compute $\pb^*_{\tomega}fd\zb^I$. By the locality of differential operator, we may assume $f \in \Cinf_c(\tU)$. Let $g \in \Cinf(\tU)$ and $J \in \mathcal{J}_{q-1,n} $ then
  \begin{align*}
        &\left(\pb^*_{\tomega}( f(z) d \zb^I  ) \middle| g(z) d \zb^J \right)_{\tomega} =\left(f(z) d\zb^{I}  \middle| \pb( g(z) d\zb^J ) \right)_{\tomega} \nonumber\\
        &=\int_{\tU}f(z)\overline{\left(\dfrac{\p g(z)}{\p \zb^i}\right)}\la d\zb^I|d\zb^i\wedge d\zb^J\ra_{\tomega} e^{\tvarphi}dm= \int_{\tU}(\dfrac{\p \Bar{g}}{\p z_i}(z))f(z) \, e^{\tvarphi(z)}\la d\zb^I | d\zb^i \wedge d \zb^J\ra_{\tomega}  dm \nonumber\\
        &= \int_{\tU} -\bar{g}(z)\pzi \left(f(z)   \, e^{\tvarphi(z)} \la d\zb^I | d\zb^i \wedge d \zb^J \ra_{\tomega}\right) dm \nonumber \\
        &=\int_{\tU}-\Bar{g}(z)  \left(\dfrac{\p f}{\p z^{i}}(z)+f(z)\dfrac{\p \tvarphi}{\p z^{i}}(z)\right)\la (d \zb^{i})\wedge^*_{\tomega} d\zb^I | d \zb^J \ra_{\tomega}\,e^{\tvarphi(z)} dm  \nonumber \\
        & \qquad \qquad \qquad \qquad \qquad \qquad \qquad  
        - \int_{\tU}\bar{g}(z)f(z)  \left(\pzi \la  d\zb^I \mid d\zb^i \wedge d \zb^J \ra_{\tomega}\right)e^{\tvarphi(z)} dm.
  \end{align*}By direct computation,
  \begin{multline*}
  \pzi \la d\zb^I|(d\zb^i)\wedge d\zb^J \ra_{\tomega}=\overline{\frac{\p}{\p \zb^i}\la dz^I|(dz^i)\wedge dz^J \ra_{\tomega}}=\overline{\la dz^I|\nabla_{\p/\p z^i}(dz^i)\wedge dz^J\ra_{\tomega}}\\
  =\overline{\la(dz^i\wedge^*_{\tomega})\theta_{\p/\p z^i}^*dz^I|dz^J\ra_{\tomega}}=\la(d\zb^i\wedge^*_{\tomega})\overline{{\theta}}_{\p/\p \zb^i}^*d\zb^I|d\zb^J\ra_{\tomega}.
  \end{multline*}So, we can conclude that
\begin{equation*}
       \pb^*_{\tomega}fd\zb^I=\left(-\dfrac{\p f}{\p z^i}-f\dfrac{\p \tvarphi}{\p z^i}\right)(d\zb^i)\wedge^*_{\tomega}d\zb^I-f \,(d\zb^i\wedge^*_{\tomega})\overline{{\theta}}_{\p/\p \zb^i}^*d\zb^I.
    \end{equation*}
\end{proof}
So far, we establish a framework for the localized Cauchy-Riemann operator $\pb_{\ts}$ and the localized Kodaira Laplacian $\Box_{\ts}$. They depend only on the local data $\tomega$ and $\tphi$ on $\Cn$. Next, we are going to apply this framework to the configuration in Section \ref{section 3.1}. \par

Recall the trivializing neighborhood $B(1) \subset U \subset M$ taken in the previous section. We insert $\tU=B(1)$ and $\ts=s^k$ into the above framework. Denote $\pb_{k,s}:=\pb_{s^k}$ as the localized Cauchy-Riemann operator. By (\ref{3.2 witten relation}), we have
$ \pb_{k,s}=e^{-k\phi}\pb e^{k\phi}$, which is an analogy of the Witten deformation of exterior derivative on real manifolds (cf.\cite{Witten}). Moreover, by (\ref{3.2 pbs eq}) and (\ref{3.2 pbs* eq}), 
\begin{align*}
    \pb_{k,s}= \pb+(\pb k\phi)\wedge\cdot \quad ; \quad
    \pb^*_{k,s}=\pb^*_{\omega}+(\pb k\phi)\wedge^*_{\omega}\cdot\,.
\end{align*}Denote $\Box^{(q)}_{k,s}:=\Box^{(q)}_{s^k}$ to be the \textbf{localized Kodaira Laplacian}. The expression of $\Box^{(q)}_{k,s}$ is given in Lemma \ref{3.2 Lemma Lap}. On the other hand, we can consider the scaled vector bundle $\T \otimes L^{(k)} \rightarrow B(\sqrt{k})$ and insert $\tU=B(\sqrt{k})$ and $\ts=s^{(k)}$. Thus, $\tomega=\omega_{(k)}$ and $\tphi=\phi_{(k)}$. Denote $\pb_{(k),s}:=\pb_{s^{(k)}}$ and compute that
\begin{align*}
   \pb_{(k),s}=\pb+(\pb\phi_{(k)})\wedge\cdot\quad;\quad \pb^*_{(k),s}=\pb^*_{\omega_{(k)}}+(\pb\phi_{(k)})\wedge^*_{\omega_{(k)}}\cdot\,.
\end{align*}Denote $\Box^{(q)}_{(k),s}:=\Box^{(q)}_{s^{(k)}}$ to be the \textbf{scaled localized Kodaira Laplacian}. The relations between the two sets of operators established above are given by
\begin{align}\label{3.2 rescale relation}
  (\pb^{(q)}_{(k),s}u)(\sqrt{k}z)=\dfrac{1}{\sqrt{k}}\pb^{(q)}_{k,s}(u(\sqrt{k}z))\quad ; \quad  
  (\pb^{*,(q)}_{(k),s}u)(\sqrt{k}z)=\dfrac{1}{\sqrt{k}}\pb^{*,(q)}_{k,s}(u(\sqrt{k}z)) 
\end{align}and hence
\begin{equation}\label{3.2 rescale relation 3}
\left(\Box_{(k),s}^{(q)}u\right)(\sqrt{k}z)=\dfrac{1}{k}\Box^{(q)}_{k,s}(u(\sqrt{k}z)). 
\end{equation}
Now, recall the Sobolev space $H_s(\tU,\T)$ with its norm $\|\cdot\|_{s}$(cf.(\ref{2.4 sobolev norm})). By the fact that
$\omega_{(k)} \rightarrow \omega_0$ and
$\phi_{(k)} \rightarrow \phi_0$
locally uniformly in $\Cinf$, we know that the coefficients of $\Box^{(q)}_{(k)}$ converge locally uniformly. Therefore, we can apply
Theorem \ref{2.4 Gårding inequality} to obtain the following proposition. 
\begin{prop}[k-uniform  Gårding inequalities]\label{3.2 k Gard ineq}

   For any fixed radius $r \geq 0$ and integers $m\in \N$, there is a constant $C$ independent of $k$ such that
   \begin{equation*}
       \|u\|_{2m} \leq C \left(\|u\|_{0} + \|(\Box^{(q)}_{(k)})^{m} u\|_{0}\right)
   \end{equation*}
   for all $u \in H_{2m}(B(r),\T)$ and $k \geq r^2$.
   \end{prop}
\begin{rmk}\label{3.2 k Gard ineq rmk}
    In fact, for any cut off functions $\rho \in \Cinf_c(B(r),[0,\infty))$, $\Tilde{\rho} \in \Cinf_c(B(r),[0,\infty))$ with $\supp \rho \subset\subset \supp \Tilde{\rho}$, there is a $C>0$ such that
    \begin{equation*}
       \|\rho u\|_{2m} \leq C \left(\| \Tilde{\rho} u\|_{0} + \|\Tilde{\rho}(\Box^{q}_{(k)})^{m} u\|_{0}\right)\quad \text{for all }\, u\in\Omega^{(0,q)}(B(r)).
   \end{equation*} This property makes the Gårding inequality applicable to sections without compact support.
\end{rmk}
\newpage 
\subsection{The uniform bounds}\label{section 3.3}
 
To begin with, we make some observations about the compatibility of norms. Note that there exist positive constants $C_1$ and $C_2$ such that
\begin{equation*}\label{3.3 compatibility 0}
C_1\|u\|_{\omega_0,B(1)}\leq \|u\|_{\omega,B(1)}\leq C_2\|u\|_{\omega_0,B(1)}    
\end{equation*}
 for all $u\in\Omega^{(0,q)}(B(1))$  since $\omega$ and $\omega_0$ are both positive bounded Hermitian $(1,1)$-forms on the precompact domain $B(1)$. By scaling the metric, it follows that 
\begin{equation}\label{3.3 compatibility 1}
C_1\|u\|_{\omega_0,B(\sqrt{k})}\leq \|u\|_{\omega_{(k)},B(\sqrt{k})}\leq C_2\|u\|_{\omega_0,B(\sqrt{k})}    
\end{equation}
for all $u\in\Omega^{(0,q)}(B(\sqrt{k}))$. In topological sense, $L^2_{\omega}(B(1),\TM)\approx L^2_{\omega_0}(B(1),\TM)$ and $L^2_{\omega_{(k)}}(B(\sqrt{k}),\TM)\approx L^2_{\omega_0}(B(\sqrt{k}),\TM)$. Moreover, by the fact that $\omega_{(k)}$ converges to $\omega_0$ locally uniformly, for any fixed radius $r>0$ and positive number $\varepsilon>0$, there exists $k_0\in\N$ such that
\begin{equation}\label{3.3 compatibility 2}
 (1-\varepsilon)|\left(u\mid v\right)_{\omega_{0},B(r)}|\leq |\left(u\mid v\right)_{\omega_{(k)},B(r)}|\leq (1+\varepsilon)|\left(u\mid v\right)_{\omega_{0},B(r)}|   
\end{equation}
for all $u,v\in\Omega^{(0,q)}_c(B(r))$ and $k\geq k_0$. Now, we enter the core of this section.

\begin{lem}\label{3.3 thm pf 1.5}
Given $u\in \Omega^{(0,q)}_c(B(\sqrt{k}))$ for some $k\in\N$, we have 
 \begin{equation*}\label{3.3 lemms 1.5 (1)}
      \|\Plks u\|_{\omega_{(k)},B(\sqrt{k})}\leq \|u\|_{\omega_{(k)},B(\sqrt{k})}.  
    \end{equation*}Consequently, 
   there exists a constant $C$ independent of $k$ such that the four inequalities hold:
   \begin{equation*}
     \|\Plks u\|_{\omega_{0}\, \textit{\rm or}\, \omega_{(k)},B(\sqrt{k})}\leq C\|u\|_{\omega_{0}\, \textit{\rm or}\, \omega_{(k)},B(\sqrt{k})}, \end{equation*}for all $u\in \Omega^{(0,q)}_c(B(\sqrt{k}))$ and $k\in \N$. Moreover, for any radius $r>0$ and a number $\varepsilon>0$, there exists $k_0$ such that 
   \begin{equation*}
       \|\Plks u\|_{\omega_{(k)},B(\sqrt{k})}\leq(1+\varepsilon)\|u\|_{\omega_0,B(r)},
   \end{equation*} for all $u\in\Omega^{(0,q)}_{c}(B(r))$ and $k>k_0$. 
\end{lem}
\begin{proof}
    Let $u\in\Omega^{(0,q)}_c(B(\sqrt{k}))$. Inspired by identification (\ref{3.1 id 4}), we define
    \[u_k(z):=k^{n/2}u(\sqrt{k}z)\in\Omega^{(0,q)}_c(B(1))
    \] which satisfies $\|u_k\|_{\omega}=\|u\|_{\omega_{(k)}}$. Since $B(1) \subset M$, we can treat $u_k$ as an element of $\Omega^{(0,q)}_c(M)$. By (\ref{3.1 scaled projection relation}), 
    \[
    \Plks u(z)=k^{-n/2}\Pks u_k(\sqrt{k}z).
    \]This means $\Plks u$ corresponds to $\Pks u_k$ under the identification (\ref{3.1 id 4}) and hence \[\|\Plks u\|_{\omega_{(k)},B(\sqrt{k})}=\|\Pks u_{k}\|_{\omega,B(1)}.\] By the identification (\ref{3.1 id 1}) and the definition of $\Pk$,
    \[
    \|\Pks u_k\|_{\omega,B(1)}=\|\Pk (u_k e^{k\phi}\otimes s^k)\|_{\omega,k\phi,B(1)} \leq \|\Pk (u_k e^{k\phi}\otimes s^k)\|_{\omega,k\phi}\leq \|u_k e^{k\phi} s^k\|_{\omega,k\phi}=\|u_k\|_{\omega}.
    \] Combine the above arguments and get
    \begin{equation*}
      \|\Plks u\|_{\omega_{(k)},B(\sqrt{k})}\leq \|u\|_{\omega_{(k)},B(\sqrt{k})}.  
    \end{equation*}Finally, we apply (\ref{3.3 compatibility 1}) and (\ref{3.3 compatibility 2}) to complete the proof.
\end{proof}

 Observe that \[
 dV_{\omega_0}=\frac{\omega_0^n}{n!}=2^{n}dx^1\wedge\cdots\wedge dx^{2n}=2^{n}dm.
\]
The volume form induced by $\omega_0$ coincides with the standard Lebesgue measure $dm$ up to a constant $2^n$. Consequently, we know that the induced $L^2$-norms $\|\cdot\|_{\omega_0}$ and $\|\cdot\|_{dm}$ (or $\|\cdot\|_{0}$)  are equivalent. Here, $\|\cdot\|_{0}$ means the Sobelov $0$-norm which is exactly $\|\cdot\|_{dm}$.  We now verify an essential result of this paper.

\begin{thm}[k-uniform smoothing property]\label{3.3 k-smoothing prop}
Fix functions $\chi$ and $\rho$ in $\Cinf_c(\Cn)$ and real numbers $s,t \in \R$. If  
 $\ \limsup_{k\To\infty}\dfrac{c_k}{k}<\infty$, there exists a constant $C$ independent of $k$ such that   
\begin{equation*}
    \| \chi \Plks \rho u\|_{s} \leq C \| u \|_{t}
\end{equation*} for all $u \in H_{t}(B(r),T^{*,(0,q)}\Cn)$ and $k\in\N$ with $\supp\chi \cup\supp\rho\subset B(\sqrt{k})$. 
\end{thm}
\begin{proof}
   It is sufficient to show that for each  $m \in \N$,
  \begin{equation}\label{3.3 smoothing pf 0}
        \chi \Plks \rho: H_{-2m}(B(r),T^{*,(0,q)}\Cn) \rightarrow H_{2m}(B(2r),T^{*,(0,q)}\Cn)
    \end{equation}is a $k$-uniformly bounded linear map for all $k$ with $\supp\chi \cup\supp\rho\subset B(\sqrt{k})$.\par
 We may assume $u\in\Omega^{(0,q)}_c(\Cn)$ by density argument.
By Proposition \ref{3.2 k Gard ineq} and Remark \ref{3.2 k Gard ineq rmk},
   \begin{align}\label{3.3 smooth pf 1}
    \|\chi \Plks \rho u\|_{2m}  &\lesssim \| \Tilde{\chi} \Plks\rho u \|_0 + \| \Tilde{\chi} (\Box^{(q)}_{(k),s})^m \Plks\rho u \|_{0}\\ &\lesssim \| \Tilde{\chi}\Plks\rho u\|_{\omega_{(k)},B(\sqrt{k})}+ \|\Tilde{\chi}\left(\Box^{(q)}_{(k),s}\right)^m\Plks\rho u\|_{\omega_{(k)},B(\sqrt{k})}.\notag    
   \end{align}The second inequality is from (\ref{3.3 compatibility 1}). Define $u_k(z):=k^{n/2}u(\sqrt{k}z)$, $\rho_k(z):=\rho(\sqrt{k}z)$ and $\Tilde{\chi}_k(z):=\Tilde{\chi}(\sqrt{k}z)$. For large enough $k$ with $B(\sqrt{k})\supset \supp \rho$, we observe that $\rho_k u_k\in\Omega^{(0,q)}_c(B(1))\subset \Omega^{(0,q)}_c(M)$ and $\|\rho_k u_k\|_{\omega,B(1)}=\|\rho u\|_{\omega_{(k)},B(\sqrt{k})}$. By Lemma \ref{3.3 thm pf 1.5},
\begin{align}\label{3.3 smoothing pf 1.1}
\| \Tilde{\chi}\Plks\rho u\|_{\omega_{(k)},B(\sqrt{k})}\lesssim\|u\|_0. \end{align}
Next, from the relations (\ref{3.1 scaled projection relation}) and (\ref{3.2 rescale relation 3}), we can see
\[ k^{n/2}\left(\Box^{(q)}_{(k),s}\right)^m\Plks \rho u (\sqrt{k}z) = k^{-m}\left(\Box^{(q)}_{k,s}\right)^m\Pks \rho_ku_k(z).
\]By changing variable again, we compute that
\begin{equation}\label{3.3 smooting pf 1.2}
\|\left(\Box^{(q)}_{(k),s}\right)^m\Plks\rho u\|_{\omega_{(k)},B(\sqrt{k})}=\|k^{-m}\left(\Box^{(q)}_{k,s}\right)^m\Pksk \rho_k u_k\|_{\omega,B(1)}.
\end{equation}Moreover, by the property of spectral projection, we estimate that
\begin{align}\label{3.3 smooting pf 1.3}
\|k^{-m}\left(\Box^{(q)}_{k,s}\right)^m\Pks \rho_k u_k\|_{\omega,B(1)}&\leq \|k^{-m}\left(\Box^{(q)}_{k}\right)^m\Pk \rho_k e^{k\phi} u_k\otimes s^k\|_{\omega,k\phi,M} \\
&\leq (\dfrac{c_k}{k})^m\|\rho_k e^{k\phi} u_k\otimes s^k\|_{\omega,k\phi,B(1)}\notag\\& =(\dfrac{c_k}{k})^m\|\rho_k u_k\|_{\omega,B(1)}\lesssim \|u\|_0 \notag.
\end{align}The last inequality is from the fact that $\limsup_{k\To\infty} c_k/k<\infty$. Combining estimates (\ref{3.3 smooth pf 1})-(\ref{3.3 smooting pf 1.3}), we have   
\begin{equation}\label{3.3 smooting pf 2}
\|\chi \Plks \rho u\|_{2m}\lesssim \|u\|_{0}.    
\end{equation}
Next, by the self-adjointness of spectral projection, for all $v\in\Omega^{(0,q)}_c(\Cn)$,
\begin{align*}
   \left(\chi\Plks\rho u\mid v\right)_{\omega_{(k)}}&=
   \left(\chi_k\Pk\rho_k e^{k\phi}u_k\otimes s^k\mid e^{k\phi}v_k\otimes s^k\right)_{\omega,k\phi}\\ &=
   \left(e^{k\phi}u_k\otimes s^k  \mid \rho_k\Pk\chi_k e^{k\phi} v_k\otimes s^k\right)_{\omega,k\phi}=\left(u\mid \rho\Plks\chi v\right)_{\omega_{(k)}} 
\end{align*}where $v_k(z):=k^{n/2}v(\sqrt{k}z)$ and $\chi_k(z):=\chi(\sqrt{k}z)$. By Proposition \ref{2.4 prop}  and (\ref{3.3 smooting pf 2}),
\[
|\left(u\mid \rho\Plks\chi v\right)_{0}|\leq\|u\|_{-2m}\|\rho\Plks\chi v\|_{2m}\lesssim\|u\|_{-2m}\|v\|_{0}.
\] By the arguments above and (\ref{3.3 compatibility 2}), we have
\begin{equation}\label{3.3 smoothing pf 3}
 \|\chi\Plksk\rho u\|_{0}\lesssim \|u\|_{-2m}  \end{equation}since $v$ is arbitrary. By (\ref{3.3 smooth pf 1}), it remains to show the following fact:
\begin{claim}
$\Tilde{\chi}\left(\Box^{(q)}_{(k),s}\right)^{m}\Plks \rho: H_{-2m}(\Cn,\T) \To L^2_{\omega_0}(\Cn,\T)$ is a k-uniformly bounded map.  
\end{claim}
To prove the claim, we observe that $\Tilde{\chi}_k u_k\in\Omega^{(0,q)}_c(B(1))\subset \Omega^{(0,q)}_c(M)$ for large enough $k$, and $\|\Tilde{\chi}_k u_k\|_{\omega,B(1)}=\|\Tilde{\chi}u\|_{\omega_{(k)}}$. By Proposition \ref{3.2 k Gard ineq} and (\ref{3.3 compatibility 1}), 
\[
\|\rho\Plks (\Box^{(q)}_{(k),s})^{m}\Tilde{\chi}u\|_{2m}\lesssim \|\Tilde{\rho}\Plks (\Box^{(q)}_{(k),s})^{m}\Tilde{\chi}u\|_{\omega_{(k)}}+\|\Tilde{\rho}(\Box^{(q)}_{(k),s})^{m}\Plks (\Box^{(q)}_{(k),s})^{m}\Tilde{\chi}u\|_{\omega_{(k)}}.
\]By rescaling, the first term on the right-hand side above is $k^{-m}\|\Tilde{\rho}_k\Pks(\Box^{(q)}_{k,s})^m\Tilde{\chi}_k u_k\|_{\omega,B(1)}$ where $\Tilde{\rho}_k(z):=\Tilde{\rho}(\sqrt{k}z)$. This can be dominated by 
\[
k^{-m}\|\Tilde{\rho}_k\Pk(\Box^{(q)}_{k})^m\Tilde{\chi}_k e^{k\phi}u_k\otimes s^k\|_{\omega,k\phi,M}\leq (\dfrac{c_k}{k})^m\|\Tilde{\chi}_k e^{k\phi}u_k\otimes s^k\|_{\omega,k\phi}\lesssim \|u\|_0
\]since $\limsup_{k\To\infty} c_k/k<\infty$. For the second term, by rescaling, we write
\[\|\Tilde{\rho}(\Box^{(q)}_{(k),s})^{m}\Plks (\Box^{(q)}_{(k),s})^{m}\Tilde{\chi}u\|_{\omega_{(k)}}=k^{-2m}\|\Tilde{\rho}_k(\Box^{(q)}_{k,s})^m\Pks(\Box^{(q)}_{k,s})^m\Tilde{\chi}_k u_k\|_{\omega,B(1)}\] which is smaller than
\[
k^{-2m}\|\Tilde{\rho}_k(\Box^{(q)}_{k})^m\Pk (\Box^{(q)}_{k})^m \Tilde{\chi}_k  e^{k\phi}u_k\otimes s^k\|_{\omega,k\phi,M}\leq (\dfrac{c_k}{k})^{2m}\|\Tilde{\chi}_k e^{k\phi}u_k\otimes s^k\|_{\omega,k\phi}\lesssim\|u\|_0.
\]Combining arguments above, we get
\begin{equation}\label{3.3 smoothing pf 4}
\|\rho\Plks (\Box^{(q)}_{(k),s})^{m}\Tilde{\chi}u\|_{2m}\lesssim \|u\|_0.    
\end{equation} By the self-adjointness of $\Box^{(q)}_{k}$ and $\Pk$, for any test function $v\in\Omega^{(0,q)}_c(\Cn)$,
\begin{align*}
\left(\Tilde{\chi}(\Box^{(q)}_{(k),s})^m\Plks\rho u\mid v\right)_{\omega_{(k)}}&=\left(k^{-m}\Tilde{\chi}_k(\Box^{(q)}_{k,s})^m\Pks\rho_k u_k\mid v_k\right)_{\omega}\\ =\left( u_k\mid k^{-m} \rho_k\Pks(\Box^{(q)}_{k,s})^m\Tilde{\chi}_k v_k\right)_{\omega}&=\left( u\mid \rho\Plks(\Box^{(q)}_{(k),s})^m\Tilde{\chi} v\right)_{\omega_{(k)}},    
\end{align*}where $v_k(z):=k^{n/2}v(\sqrt{k}z)$. Again, by Proposition \ref{2.4 prop}  and (\ref{3.3 smoothing pf 4}), we have 
\[
|\left( u\mid \rho\Plks(\Box^{(q)}_{(k),s})^m\Tilde{\chi} v\right)_{0}|\leq \|u\|_{-2m}\|\rho\Plks(\Box^{(q)}_{(k),s})^m\Tilde{\chi} v\|_{2m}\lesssim \|u\|_{-2m}\|v\|_{0}.
\]
Hence, we obtain $\left(\Tilde{\chi}(\Box^{(q)}_{(k),s})^m\Plks\rho u\mid v\right)_{\omega_{(k)}}\lesssim\|u\|_{-2m}\|v\|_0$ by combining above arguments. 
This completes the proof of the claim since $v$ is arbitrary. Finally, by estimates (\ref{3.3 smooth pf 1}), (\ref{3.3 smoothing pf 3}) and the claim,
\begin{align*}
\|\chi \Plks \rho u\|_{2m}   &\lesssim \| \Tilde{\chi}\Plks\rho v\|_{\omega_{(k)},B(\sqrt{k})}+ \|\Tilde{\chi}\left(\Box^{(q)}_{(k),s}\right)^m\Plks\rho v\|_{\omega_{(k)},B(\sqrt{k})}
\lesssim \|u\|_{-2m}
\end{align*}for all $u\in\Omega^{(0,q)}_c(\Cn)$. The theorem follows.
\end{proof}
With the preliminary work out of the way, we can now address the local uniform bound of $\Plksk(z,w)$. To do so, we represent $\Plksk(z,w)$ as the form
\[
\Plksk(z,w)=\sum_{I,J\in\mathcal{J}_{q,n}}\Plkski(z,w) d \zb^I \otimes (\dfrac{\p}{\p \Bar{w}})^J, 
\]where $\Plkski(z,w)\in \Cinf(B(\sqrt{k})\times B(\sqrt{k}))$. Also, we define the $\mathcal{C}^{d}$-norm of $\Plksk(z,w)$ on the bounded domain $B(r)\times B(r)$ as 
\[
\|\Plksk(z,w)\|^2_{\mathcal{C}^d(B(r)\times B(r))}:=\sup_{x,y\in B(r)}\sum_{I,J\in\mathcal{J}_{q,n}}\sum_{|\alpha|+|\beta|\leq d}\left(|\p^{\alpha}_x\p^{\beta}_{y}\Plkski(x,y)|^2\right)
\]
where the variables $x,y$ represent the real coordinates of $\Cn\simeq\R^{2n}$. 
For a $K(z,w)\in \Cinf(\Cn\times\Cn,\T\boxtimes \T)$, we say $\Plksk(z,w)\To K(z,w)$ as $k\To\infty$ locally uniformly in $\Cinf$ if $\|\Plksk(z,w)-K(z,w)\|_{\mathcal{C}^d(B(r)\times B(r))}\To 0$ as $k\To\infty$ for all $d\in\N$ and $r>0$. 
\begin{thm}[The local uniform bounds for scaled  spectral and Bergman kernels]\label{3.3 loc bdd thm} Suppose $c_k$ is a nonnegative sequence such that \[ \limsup_{k\To\infty}\frac{c_k}{k}<\infty.\]
   Fix a radius $r>0$.
   For any multi-indices $\alpha,\beta \in \N_0^{2n}$ and $I,J\in\mathcal{J}_{q,n}$, there exists a constant $C$ independent of $k$ such that 
   \begin{equation*}
\sup_{B(r)\times B(r)}|\p^{\alpha}_x\p^{\beta}_yP^{(q),I,J}_{(k),c_k,s}(x,y)|<C \quad \text{\rm for all }k \geq r^2.
\end{equation*}
\end{thm}

\begin{proof}
We start from the approximation of identity. For any fixed point $y_0 \in B(r)$, we set $f_l$ as an approximation of identity with its mass concentrated at $y_0$ as $l \rightarrow \infty$. For example, let $f_l =l^nf(\sqrt{l}(y-y_0))$ where $f \in \Cinf_c(B(r);[0,\infty))$ and $\int_{B(r)}f dm=1$. By the property of approximation of identity, it is sufficient to establish the following estimate:
\begin{equation*}
    \sup_{x\in B(r),k>r^2,\atop l\in \N}| \int_{B(r)}\p^\alpha_x \p^\beta_y \Plkski(x,y)f_l(y) dm(y)| < C.
\end{equation*} We hope the $C$ is independent of  the point $y_0\in B(r)$ chosen above. By integration by part, we just need to find an upper bound of

\begin{equation*}
    \sup_{x\in B(r),k>r^2,\atop l\in \N}|\p^\alpha_x  \int_{B(r)}\Plkski(x,y)\p^\beta_y f_l(y)dm(y)|.
\end{equation*}
Choose $\chi \in \Cinf_{c}(B(2r))$ so that $\chi\equiv1$ on $B(r)$. Observe that
\begin{align*}
    \sup_{x\in B(r),k>r^2,\atop l\in \N}|\p^\alpha_x  \int_{B(r)}\Plkski(x,y)(\p^\beta_y f_l(y))dm(y)| & \leq \sup_{k>r^2,l\in\N}\| \Plks\left((\p^{\beta} f_l)d\zb^J\right) \|_{\mathcal{C}^{|\alpha|}(B(r))}\\
    & \leq \sup_{k>r^2,l\in\N}\|\chi\Plks\left((\p^{\beta} f_l)d\zb^J\right) \|_{\mathcal{C}^{|\alpha|}(B(2r))},
\end{align*}where the norm $\|\cdot\|_{\mathcal{C}^{|\alpha|}(B(2r))}$ we adopted is defined in (\ref{2.4 C space}). By Sobolev inequality, for any integer $m$ with $2m \geq |\alpha|+n$, we have 
\begin{equation*}
    \| \chi \Plks \left((\p^{\beta} f_l)d\zb^J\right) \|_{\mathcal{C}^{|\alpha|}(B(2r))} \lesssim \| \chi \Plks\left((\p^{\beta} f_l)d\zb^J\right) \|_{2m}.
\end{equation*}
Note that $|\hat{f}_l(\xi)|\lesssim|\int_{\R^{2n}}e^{-\ii x\cdot \xi}f_l(x)dx|\lesssim O(1)$ and hence $|\widehat{(\p^{\beta}f_l)}|\lesssim |\xi|^{|\beta|}|\hat{f_l}|\lesssim |\xi|^{|\beta|}$. This implies that for large enough  $m\in\N$,
\[
\| (\p^{\beta} f_l)d\zb^J \|_{-2m} \lesssim O(1).
\]
  After combining this fact with Theorem \ref{3.3 k-smoothing prop}, we know that for large enough $m\in\N$, 
\begin{align*}
    \| \chi \Plks (\p^{\beta} f_l)d\zb^J \|_{2m}&\lesssim\| \chi \Plks \rho  (\p^{\beta} f_l)d\zb^J \|_{2m}\\ &\lesssim \|(\p^{\beta} f_l)d\zb^J\|_{-2m} \lesssim O(1),
\end{align*} where $\rho$ is a bump function which has value $1$ around the point $y_0$. We have completed the proof.
\end{proof}
 
We end this section with the following extremely important corollary, which is an immediate consequence of the Arzelà-Ascoli theorem and Theorem \ref{3.3 loc bdd thm}.
\begin{cor}\label{3.3 loc conv cor}
    If $c_k$ is a nonnegative sequence such that \[\limsup_{k\To\infty}\frac{c_k}{k}<\infty,\] then
    any subsequence of the scaled localized spectral kernel $\Plksk(z,w)$ ( or Bergman kernel $\Blksk(z,w)$ in the case $c_k=0$) has a $\Cinf$ 
    locally uniformly convergent subsequence in $\Cn$. \par By the identity (\ref{kernel relation 2}), we have the same results for $P^{(q),s}_{(k),c_k}(z,w)$ and $B^{(q),s}_{(k)}(z,w)$.   
\end{cor}

\newpage

\section{Asymptotics of spectral and Bergman kernels}\label{chapter 4}
Recall the Corollary \ref{3.3 loc conv cor}. We have established that $\Plksk(z,w)$ (or $P^{(q),s}_{(k),c_k}(z,w)$) is a sequence such that every subsequence has a $\Cinf$ uniformly convergent subsequence. To show that $\Plksk(z,w)$ (or $P^{(q),s}_{(k),c_k}(z,w)$) is itself a uniformly convergent sequence in $\Cinf$, it suffices to demonstrate that every convergent subsequence of $\Plksk(z,w)$ converges to the same limit. This follows from the fact that a sequence converges to a limit if and only if every subsequence has a subsequence converges to the same limit. Therefore, without loss of generality, we may assume that $\Plksk(z,w)$ is a $\Cinf$ uniformly convergent sequence, and our goal is to prove that the limit must be the Bergman kernel in a model case which will be thoroughly investigated in Section \ref{section 4.1}. From now on, we make the following assumption throughout this chapter.
\begin{assumption}\label{4.2 assumption}
The scaled localized spectral kernel $\Plksk(z,w)$ converges to $B^{(q)}_{s}(z,w)$ locally uniformly in $\Cinf$ on $\Cn$, where $B^{(q)}_s(z,w)\in\Cinf(\Cn\times\Cn,\T\boxtimes\T)$. Equivalently, the scaled spectral kernel $P^{(q),s}_{(k),c_k}(z,w)$ converges to $B^{(q),s}(z,w):=e^{\phi_0(z)}B^{(q)}_s(z,w)e^{-\phi_0(w)}$ locally uniformly in $\Cinf$.
 \end{assumption}
To maintain the validity of Corollary \ref{3.3 loc conv cor}, we must specify a nonnegative sequence $c_k$ such that
$
\limsup_{k\To\infty}\dfrac{c_k}{k}<\infty.
$ However, we will require a stronger condition that 
$\limsup_{k\To\infty}\dfrac{c_k}{k} =0$. We will see the reason in Lemma \ref{4.2 thm bd op BoxB=0}. Clearly, the scaled localized Bergman kernels $\Blksk(z,w)$ can be treated as a spacial case of 
 localized spectral kernels when $c_k=0$ and they satisfy the above condition. Before investigating the properties of $B^{(q)}_s(z,w)$, we study the space of sections in the \textbf{model case} on $\Cn$.  
\subsection{The model case}\label{section 4.1}
We now consider the trivial vector bundle $T^{*,(0,q)}\Cn  \otimes \Cs \rightarrow \Cn$ which is equipped with a pointwise Hermitian structure induced by the weight function $\phi_0 = \sum_{i=1}^n \lambda_i|z^i|^2$ on the trivial line bundle $\Cs\To\Cn$ and the standard Hermitian form $\omega_0=\ii\sum_{i=1}^{n}dz^i \wedge d\zb^i$ on $T^{*,(0,q)}\Cn\To\Cn$(cf.(\ref{3.2 inner product pt})).  \par 
We can define the Hilbert space $L^2_{\omega_0,\phi_0}(\Cn,\T)$ with an inner product $\left(\cdot|\cdot\right)_{\omega_0,\phi_0}$ (cf.(\ref{2.2 inner product 1})) and another space $L^2_{\omega_0}(\Cn,\T)$ with $\left(\cdot|\cdot\right)_{\omega_0}$ as its inner product (cf. (\ref{2.2 inner product 0})). There is an unitary identification 
\[L^2_{\omega_0,\phi_0}(\Cn,\T\otimes\Cs)\cong L^2_{\omega_0}(\Cn,\T) \quad \text{\rm by }\,\eta \leftrightarrow\eta e^{-\phi_0}.\] Let $\pb_0^{(q)}:\Omega^{(0,q)}(\Cn,\Cs)\To \Omega^{(0,q+1)}(\Cn,\Cs) $ and $\pb_0^{*,(q+1)}:\Omega^{(0,q+1)}(\Cn,\Cs)\To \Omega^{(0,q)}(\Cn,\Cs)$ be the Cauchy-Riemann operator and its formal adjoint with respect to $(\cdot|\cdot)_{\omega_0,\phi_0}$, respectively. Then \[
\Box^{(q)}_0:=\pb^{*,(q+1)}_{0}\pb^{(q)}_0+\pb^{(q-1)}_0\pb^{*,(q)}_{0}:\Dom\Box^{(q)}_0 \rightarrow L^2_{\omega_0,\phi_0}(\Cn,\T\otimes\Cs) 
\] is the Gaffney extension of the Kodaira Laplacian with respect to the Hermitian structure. As in (\ref{3.2 pbs eq}) and (\ref{3.2 pbs* eq}), we can define the localized Cauchy-Riemann operator $\pb_{0,s}$ and its formal adjoint  $\pb^*_{0,s}$ with respect to $(\cdot|\cdot)_{\omega_0}$ which are given by 
\begin{align}\label{4.1 pb eq}
\pb_{0,s}^{(q)}=\pb^{(q)}+(\pb \phi_0)\wedge\cdot  \quad;\quad  \pb^{*,(q)}_{0,s}=\pb^{*,(q)}_{\omega_0}+(\pb \phi_0)\wedge^*_{\omega_0},\notag
\end{align}respectively. The localized Kodaira Laplacian is
\[
\Box^{(q)}_{0,s}:=\pb^{*,(q+1)} _{0,s}\pb^{(q)}_{0,s}+\pb^{(q-1)}_{0,s}\pb^{*,(q)}_{0,s}:\Dom \Box^{(q)}_{0,s}\To L^2_{\omega_0}(\Cn,\T)
\]which satisfies $\Box^{(q)}_{0}(\eta \otimes 1)=e^{\phi_0}\Box^{(q)}_{0,s}(\eta e^{-\phi_0}) \otimes 1
$. Next, we denote 
\[
\mathcal{B}^{(q)}_{0}:L^{2}_{\omega_0,\phi_0}(\Cn,T^{*,(0,q)}\Cn\otimes\Cs) \rightarrow \Ker \Box^{(q)}_0\subset L^2_{\omega_0,\phi_0}(\Cn,\T\otimes\Cs)
\] to be the Bergman projection and \[\mathcal{B}^{(q)}_{0,s}:L^{2}_{\omega_0}(\Cn,T^{*,(0,q)}\Cn) \rightarrow \Ker \Box^{(q)}_{0,s}\subset L^2_{\omega_0}(\Cn,\T)\] to be the localized Bergman projection satisfying
$\mathcal{B}^{(q)}_{0}(\eta \otimes 1)=e^{\phi_0}\mathcal{B}^{(q)}_{0,s}(\eta e^{-\phi_0}) \otimes 1$. Furthermore, denote by $B^{(q)}_{0}(z,w)$ the \textbf{Bergman kernel} and $B^{(q)}_{0,s}(z,w)$ the \textbf{localized Bergman kernel} which are Schwartz kernels of $\mathcal{B}^{(q)}_{0}$ and $\mathcal{B}^{(q)}_{0,s}$, respectively. Our main goal of this section is to compute the localized Bergman kernel ${B}^{(q)}_{0,s}(z,w)$. Proposition \ref{4.1 prop hi} below tells us that $\Box^{(q)}_{0,s}$ 
 is $\textbf{diagonal}$ with respect to the basis $\{d\zb^I\}_{I\in\mathcal{J}_{q,n}}$. 

\begin{prop}\label{4.1 prop hi}
For $f\in\Cinf(\Cn,\Cs)$ and $I\in\mathcal{J}_{q,n}$, \begin{align*}
\Box^{(q)}_{0,s}fd\zb^I=\sum_{i=1}^{n}\left(-\dfrac{\p^2 f}{\p \zb^i \p z^i}+\dfrac{\p \phi_0}{\p z^i}\dfrac{\p f}{\p \zb^i}-\dfrac{\p \phi_0}{\p \zb^i}\dfrac{\p f}{\p z^i}\right)d\zb^I+\left(\sum_{i\in I}\lambda_i-\sum_{i\notin I}\lambda_i+|\pb \phi_0|_{\omega_0}^2\right)fd\zb^I.
    \end{align*}     
\end{prop}
\begin{proof}
    Since the metric $\omega_0$ is flat, the Hodge Laplacian $\Delta_{\omega_0}u$ of $u:=fd\zb^I$ is  
    \begin{align*}\left(\pb^{*}_{\omega_0}\pb+\pb\pb^{*}_{\omega_0}\right)u=-\sum_{i,j}\dfrac{\p^2 f}{\p\zb^i\p z^j}\left(d\zb^i\wedge^*_{\omega_0}(d\zb^j\wedge)+d\zb^j\wedge (d\zb^i\wedge^*_{\omega_0})\right)d\zb^I=- \sum_{i}\dfrac{\p^2 f}{\p \zb^i \p z^i}d\zb^I.  
    \end{align*}We compute the remaining terms of equation (\ref{3.2 Lemma Lap simple Box eq}) in Lemma \ref{3.2 Lemma Lap}.  
    \begin{align*}
        \pb\left((\pb\phi_0)\wedge^*_{\omega_0}u\right)&+(\pb \phi_0)\wedge^*_{\omega_0}\pb u=\pb\left(\sum_{i}\dfrac{\p \phi_0}{\p z^i}f d\zb^i \wedge^*_{\omega_0}d\zb^I\right)+(\pb \phi_0)\wedge^*_{\omega_0}\sum_{j}\dfrac{\p f}{\p \zb^j}d\zb^j\wedge d\zb^I\\&=\sum_{i,j}\left(\dfrac{\p^2 \phi_0}{\p \zb^j \p z^i}f+\dfrac{\p \phi_0}{\p z^i}\dfrac{\p f}{\p \zb^j}\right)d\zb^j\wedge(d\zb^i\wedge^*_{\omega_0})d\zb^I +\sum_{i,j}\dfrac{\p \phi_0}{\p z^i}\dfrac{\p f}{\p \zb^j} (d\zb^i\wedge^*_{\omega_0}) d\zb^j \wedge d\zb^I\\&=\sum_{i\in I}\lambda_i f d\zb^I+\sum_i \dfrac{\p \phi_0}{\p z^i}\dfrac{\p f}{\p \zb^i}d\zb^I.
    \end{align*}On the other hand, 
    \begin{align*}
        \pb^*_{\omega_0}\left((\pb\phi_0)\wedge u\right)+(\pb \phi_0)\wedge\pb^*_{\omega_0}u&=\pb^*_{\omega_0}\left(\sum_{j}\dfrac{\p \phi_0}{\p \zb^j}fd\zb^j\wedge d\zb^I\right)+(\pb \phi_0)\wedge \sum_{i}-\dfrac{\p f}{\p z^i}d\zb^i \wedge^*_{\omega_0}d\zb^I\\&=-\sum_{i\notin I}\lambda_if d\zb^I-\sum_i\dfrac{\p \phi_0}{\p \zb^i}\dfrac{\p f}{\p z^i}d\zb^I.
    \end{align*}Applying Lemma \ref{3.2 Lemma Lap}, we have
    \begin{multline*} \Box^{(q)}_{0,s}fd\zb^I=\sum_{i}\left(-\dfrac{\p^2 f}{\p \zb^i \p z^i}+\dfrac{\p \phi_0}{\p z^i}\dfrac{\p f}{\p \zb^i}-\dfrac{\p \phi_0}{\p \zb^i}\dfrac{\p f}{\p z^i}\right)d\zb^I+\left(\sum_{i\in I}\lambda_i-\sum_{i\notin I}\lambda_i+|\pb \phi_0|_{\omega_0}^2\right)fd\zb^I.
    \end{multline*}
\end{proof}
 We now try to find the complete orthonormal system of the space $\Ker\Box^{(q)}_{0,s}$. By Proposition \ref{4.1 prop hi}, to consider the equation $\Box^{(q)}_{0,s}u=0$, we can assume $u$ is of the form \[u(z):=f_I(z)d\zb^I \quad \text{where }\, d\zb^I:=d\zb^1\wedge\cdots\wedge d\zb^q.\] That is, we fix the multi-index $I:=(1,\cdots,q)\in\mathcal{J}_{q,n}$ and let $f_I\in L^2_{dm}(\Cn)\cap \Cinf(\Cn)$. Note that $\Box^{(q)}_{0,s}u=0$ if and only if $\pb^{(q)}_{0,s}u=0$ and $\pb^{*,(q)}_{0,s}u=0$ which are equivalent to
\[
   \dfrac{\p f_I}{\p z^i}-\lambda_i \zb^if_I=0 \,\,\,\forall\, i \in I \quad \text{\rm and}\quad 
    \dfrac{\p f_I}{\p \zb^i}+\lambda_i z^if_I=0 \,\,\,\forall \, i \notin I.\]
Thus, $\Box^{(q)}_{0,s}u=0$ if  and only if
\[
F_I(z):=f_I(\Bar{z^1},\cdots,\Bar{z^q},z^{q+1},\cdots,z^{n})e^{-\sum_{i=1}^{q}\lambda_i|z^i|^2+\sum_{i=q+1}^{n}\lambda_i|z^i|^2} 
\]is a holomorphic function on $\Cn$. If we write $F_I(z)$ as the form $
F_{I}(z)=\sum_{\alpha\in\N_0^n}a_{\alpha}z^{\alpha}$ for some coefficients  $a_{\alpha}\in\Cs$, then \[
f_{I}(\Bar{z^1},\cdots,\Bar{z^q},z^{q+1},\cdots,z^{n})=\sum_{\alpha\in\N_0^n}a_{\alpha}z^{\alpha}e^{\sum_{i=1}^{q}\lambda_i|z^i|^2-\sum_{i=q+1}^{n}\lambda_i|z^i|^2}. 
\]We can apply Fubini's theorem and introduce polar coordinates by setting $z^i=r_ie^{\ii\theta_i}$ to compute that for all  $\alpha,\alpha'\in\N_0^n$,
\begin{align*}
&\left(z^{\alpha}e^{\sum_{i=1}^{q}\lambda_i|z^i|^2-\sum_{i=q+1}^{n}\lambda_i|z^i|^2}\mid z^{\alpha'}e^{\sum_{i=1}^{q}\lambda_i|z^i|^2-\sum_{i=q+1}^{n}\lambda_i|z^i|^2}\right)_{\omega_0}\\&=2^n\int_{\Cn}z^{\alpha}\Bar{z}^{\alpha'}e^{2(\sum_{i=1}^{q}\lambda_i|z^i|^2-\sum_{i=q+1}^{n}\lambda_i|z^i|^2)}dm\\
&=2^n\left(\prod_{i=1}^{n}\int_{0}^{2\pi}e^{\ii(\alpha_i-\alpha'_i)\theta_i}d\theta_i\right)\left(\prod_{i=1}^{q}\int_{0}^{\infty}r_i^{\alpha_i+\alpha'_i+1}e^{2\lambda_i r_i^2} dr_i\right) \left(\prod_{i=q+1}^{n}\int_{0}^{\infty}r_i^{\alpha_i+\alpha'_i+1}e^{-2\lambda_i r_i^2}dr_i\right),
\end{align*}which is zero if $\alpha\neq\alpha'$. By the Parseval's identity, we can compute that:\begin{align*}
\|u\|^2_{\omega_0}&=2^n\int_{\Cn}|f_I|^2dm=2^n\sum_{\alpha\in\N_0^n}|a_{\alpha}|^2\int_{\Cn}|z^{\alpha}|^2 e^{2(\sum_{i=1}^{q}\lambda_i|z^i|^2-\sum_{i=q+1}^{n}\lambda_i|z^i|^2)}dm\\ &= 2^n(2\pi)^{n} \sum_{\alpha\in\N_0^n}|a_{\alpha}|^2\left(\prod_{i=1}^{q}\int_{0}^{\infty}r_i^{2\alpha_i+1}e^{2\lambda_ir_i^2}dr_i\cdot\prod_{i=q+1}^{n}\int_{0}^{\infty}r_i^{2\alpha_i+1}e^{-2\lambda_ir_i^2}dr_i \right).
\end{align*}
By the assumption that $\|u\|_{\omega_0}^2$ is a finite number, we can conclude that if $u$ is not identically zero, then $\lambda_i<0$ for all $i\in\{1,\cdots,q\}$ and $\lambda_i>0$ for all $i\in\{q+1,\cdots,n\}$. As a result, there exist nontrivial solutions of the equation $\Box^{(q)}_{0,s}u=0$ in $L^2_{\omega_0}(\Cn,\T)$ only when $p\in M(q)$. In other words, if $p\notin M(q)$, then \[\Ker\Box^{(q)}_{0,s}=\{0\}\quad \text{and} \quad B^{(q)}_{0,s}(z,w)\equiv 0.\] Now, we focus on case $p\in M(q)$. Suppose $\lambda_i<0$ for all $i\in\{1,\cdots,q\}$ and $\lambda_i>0$ for all $i\in\{q+1,\cdots,n\}$. For brevity, we set \[z^{\alpha}_{q}:=(\zb^1)^{\alpha_1} \cdots (\zb^q)^{\alpha_q}(z^{q+1})^{\alpha_{q+1}}\cdots (z^n)^{\alpha_n} \quad;\quad I:=(1,\cdots ,q)\in \mathcal{J}_{q,n}.\]Observe that $f_I d\zb^I$ is an element in $\Ker\Box^{(q)}_{0,s}$ if and only if $f_I$ is in the set \[
\{\Tilde{F}_I(\Bar{z^1},\cdots,\Bar{z^{q}},z^{q+1},\cdots,z^{n})e^{-\sum_{i=1}^n|\lambda_i||z^i|^2}; \Tilde{F}_I \,\text{is a holomorphic function on }\, \Cn\}.
\]
Therefore, we can see that the orthogonal basis of $\Ker\Box^{(q)}_{0,s}$ is given by 
\begin{equation*}
 \{z^{\alpha}_{q}e^{-\sum_{i=1}^n|\lambda_i||z^i|^2}d\zb^I\}_{\alpha\in\N_0^{n}}.  
\end{equation*}Next, we denote \[
[\lambda]:=(|\lambda_1|,\cdots ,|\lambda_n|)\in\R^{n}.
\] We now compute the length of the orthogonal basis to normalize them. By Fubini's theorem and changing the variables by letting $z^i=r_ie^{\sqrt{-1}\theta_i}$ and $u_i=2|\lambda_i|r_i^2$,

\begin{align*}
    \|z^{\alpha}_{q}e^{-\sum_{i=1}^n|\lambda_i||z^i|^2}d\zb^I\|^2_{\omega_0}&= 2^n\int_{\Cn}\prod_{i=1}^n|z^i|^{2\alpha_i}e^{-2\sum_{i=1}^n|\lambda_i||z^i|^2}dm = 2^n\prod_{i=1}^{n}(2\pi)\int_{0}^{\infty}r_i^{2\alpha_i+1}e^{-2|\lambda_i|r_i^2}dr_i\\&= 2^n\prod_{i=1}^{n}\dfrac{\pi}{(2|\lambda_i|)^{\alpha_i+1}}\int_{0}^{\infty}u_i^{\alpha_i}e^{-u_i}du_i= 2^n\prod_{i=1}^{n}\dfrac{\pi\,\Gamma(\alpha_i+1)}{(2|\lambda_i|)^{\alpha_i+1}}
    \\&=2^n\prod_{i=1}^{n}\dfrac{\pi \alpha_i !}{(2|\lambda_i|)^{\alpha_i+1}}
    =\dfrac{\pi^n \alpha !}{2^{|\alpha|}[\lambda]^{\alpha+1}}.
\end{align*}Consequently,
\begin{equation*}
    \{\Psi_{\alpha}:=\sqrt{\dfrac{2^{|\alpha|}[\lambda]^{\alpha+1}}{\pi^n \alpha!}}z^{\alpha}_{q}e^{-\sum_{i=1}^{n}|\lambda_i||z^i|^2}d\zb^I\}_{\alpha \in \N_0^n} 
\end{equation*} is the orthonormal basis of $\Ker \Box^{(q)}_{0,s}$ and 
\begin{align*}
    B^{(q)}_{0,s}(z,w)&=\sum_{\alpha \in \N_0^n}\Psi_{\alpha}(z) \otimes \Psi_{\alpha}^{*}(w)
    =\sum_{\alpha \in \N_0^n}\dfrac{2^{|\alpha|}[\lambda]^{\alpha+1}}{\pi^n \alpha!}z^{\alpha}_{q}\overline{w^{\alpha}_{q}}e^{-\sum_{i=1}^{n}|\lambda_i|(|z^i|^2+|w^i|^2)} d\zb^I \otimes (\pwb)^I \\
    &= \dfrac{|\lambda_1 \cdots \lambda_n|}{\pi^n}\sum_{|\alpha|\in \N_0^n}\left(\dfrac{2^{|\alpha|}}{\alpha!}[\lambda]^{\alpha}z^{\alpha}_{q}\Bar{w}^{\alpha}_{q}\right) e^{-\sum_{i=1}^{n}|\lambda_i|(|z^i|^2+|w^i|^2)} d\zb^I \otimes (\pwb)^I  \\
    &= \dfrac{|\lambda_1 \cdots \lambda_n|}{\pi^n}\,e^{2(\sum_{i=1}^{q}|\lambda_i|\zb^i w^i+\sum_{i=q+1}^{n}|\lambda_i|z^i\wb^i)-\sum_{i=1}^{n}|\lambda_i|(|z^i|^2+|w^i|^2)}d\zb^I \otimes (\pwb)^I.
\end{align*}We summarize the results in the following theorem.
\begin{thm}[Bergman kernel for the model case]\label{4.1 Model theorem}
Consider the trivial vector bundle $\T \otimes \Cs \rightarrow \Cn$ endowed with the standard Hermitian form $\omega_0$ and the weight function $\phi_0$ . In the case $p\in M(q)$, 
 we assume $\lambda_i<0$ for all $i\leq q$ and $\lambda_i>0$ for all $i>q$. The localized Bergman Kernel $B^{(q)}_{0,s}(z,w)$ for $(0,q)$-forms is given by
\begin{equation*}
    \dfrac{|\lambda_1 \cdots \lambda_n|}{\pi^n}\,e^{2(\sum_{i=1}^{q}|\lambda_i|\zb^i w^i+\sum_{i=q+1}^{n}|\lambda_i|z^i\wb^i)-\sum_{i=1}^{n}|\lambda_i|(|z^i|^2+|w^i|^2)}(d\zb^1\wedge \cdots \wedge d\zb^q) \otimes (\dfrac{\p}{\p \wb^{1}} \wedge \cdots \wedge \dfrac{\p}{\p \wb^{q}}).
\end{equation*}Furthermore,
\[
\{\Psi_{\alpha}:=\sqrt{\dfrac{2^{|\alpha|}[\lambda]^{\alpha+1}}{\pi^n \alpha!}}z^{\alpha}_{q}e^{-\sum_{i=1}^{n}|\lambda_i||z^i|^2}d\zb^1\wedge \cdots \wedge d\zb^q\}_{\alpha \in \N_0^n} 
\]is the orthonormal basis of $\Ker \Box^{(q)}_{0,s}\subset L^2_{\omega_0}(\Cn,\T)$.\par However, if $p\notin M(q)$, then
\[
\Ker \Box^{(q)}_{0,s}=\{0\}
\,\, \text{ and hence }\, \Bsko(z,w) \equiv 0.\]
\end{thm}\newpage

\subsection{Mapping properties of the approximated integral operator}\label{section 4.2}

Returning to Assumption \ref{4.2 assumption}, the kernel section $B^{(q)}_{s}(z,w)$ is unknown to us so far. Our objective is to demonstrate that it must be precisely the Bergman kernel $B^{(q)}_{0,s}(z,w)$ in the model case established above. We embark on the proof by the following definition and lemma which helps us to translate $B^{(q)}_{s}(z,w)$ from an unknown kernel section to an operator on the Hilbert space $L^2_{\omega_0}(\Cn,\T)$. 
\begin{defin}\label{4.2 def}Define the \textbf{approximated integral operator} as
 \begin{equation*}
        \Bs u(z) := \int_{\Cn}\Bsk(z,w)u(w)dm(w) \quad\text{\rm for all }\, u\in L^2_{\omega_0}(\Cn,\T).
    \end{equation*}   
\end{defin}
\begin{lem}[Well-definition of the integral operator]\label{4.2 Lem B op}
    For any $u \in L^2_{\omega_0}(\Cn,\T)$, the integral
    $\Bs u(z)$ converges for almost every $z\in \Cn$. Furthermore, the integral operator
    \begin{equation*}
        \Bs:L^2_{\omega_0}(\Cn,T^{*,(0,q)}\Cn) \rightarrow L^2_{\omega_0}(\Cn,T^{*,(0,q)}\Cn)
    \end{equation*}is a bounded linear map with its operator norm smaller than $1$.
\end{lem}
\begin{proof}
      Let $u\in\Omega^{0,q}_c(\Cn)$ and we aim to show $\|\Bs u\|_{\omega_{0}}\leq \|u\|_{\omega_0}$. Let $v\in \Omega^{(0,q)}_c(\Cn)$ be a test form.
    \begin{align*}
        \left(v\mid \Bs u\right)_{\omega_0}&=\int_{\supp v}\int_{\supp u}\la v(z) | \Bsk(z,w)u(w) \ra_{\omega_0}2^{2n}dm(w)dm(z).
    \end{align*}Let $\varepsilon>0$. By (\ref{3.3 compatibility 2}) and the fact that $\omega_{(k)}\rightarrow \omega_0$ and $\Plksk(z,w) \rightarrow \Bsk(z,w)$ uniformly on $\supp v \times \supp u$, the above integral can be dominated as
    \begin{multline*}
     |\left(v\mid \Bs u\right)_{\omega_0}|\leq(1+\varepsilon)|\left(v\mid \Plks u\right)_{\omega_0}|\\ \leq(1+\varepsilon)^2|\left(v\mid \Plks u\right)_{\omega_{(k)}}|\leq (1+\varepsilon)^2 \|v\|_{\omega_{(k)}}\|\Plks u\|_{\omega_{(k)}}   
    \end{multline*} 
    for large enough $k$. We apply (\ref{3.3 compatibility 2}) and Lemma \ref{3.3 thm pf 1.5} to  obtain
    \[
    \|v\|_{\omega_{(k)}}\|\Plk u\|_{\omega_{(k)}}\leq(1+\varepsilon)^2\|v\|_{\omega_0}\|u\|_{\omega_0}\quad \text{\rm for large enough }\, k.
    \]
    Since $\varepsilon>0$ is arbitrary, the estimates above mean $|\left(v\mid \Bs u\right)_{\omega_0}|\leq \|v\|_{\omega_0}\|u\|_{\omega_0}$ which implies $\| \Bs u\|_{\omega_0} \leq \|u\|_{\omega_0}$ because the test function $v$ is arbitrary. We have completed the proof by density argument. 
\end{proof}

Now, we state the key theorem of this section.
\begin{thm}\label{4.2 thm bd op BoxB=0}
 If  $\limsup_{k\To\infty}c_k/k=0$, then $\Bs $ is a bounded linear map \[\Bs:L^{2}_{\omega_0}(\Cn,\T)\To \Ker \Box^{(q)}_{0,s}\]with its operator norm smaller than $1$. 
\end{thm}
 \begin{proof}
By Lemma \ref{4.2 Lem B op}, it remains to show:
\begin{claim}
 If $\limsup_{k\To\infty} c_k/k=0$, then  $\Box^{(q)}_{0,s}\Bs u = 0
     $ for all $u \in L^2_{\omega_0}(\Cn,\T)$   
\end{claim}
We may assume $u \in \Omega^{(0,q)}_c(\Cn)$ by density argument. Fix $\rho\in \Cinf_{c}(\Cn)$ to be a cut-off function. By assumption $\Plksk\rightarrow\Bsk$ locally uniformly in $\Cinf$,
 \[
 \|\rho \Box^{(q)}_{0,s}\Bs u\|_{\omega_0} \lesssim \|\rho \Box^{(q)}_{0,s}\Plks u\|_{\omega_0}.
 \]
 Recall Lemma \ref{3.2 Lemma Lap} and the fact that $\omega_{(k)}\To \omega_0$ and $\phi_{(k)}\To \phi_0$ locally uniformly in $\Cinf$. We can immediately conclude that the coefficients of $\Box^{(q)}_{(k),s}$ converge to those of $\Box^{(q)}_{0,s}$ locally uniformly on $\Cn$. By this fact,
   \begin{multline*}
    \|\rho \Box^{(q)}_{0,s}\Plks u\|_{\omega_0}\lesssim\|\rho \Box^{(q)}_{(k),s}\Plks u\|_{\omega_0}\\ \lesssim\|\rho \Box^{(q)}_{(k),s}\Plks u\|_{\omega_{(k)},B(\sqrt{k})}\leq \|\Box^{(q)}_{(k),s}\Plks u\|_{\omega_{0},B(\sqrt{k})},   
   \end{multline*}
   where the second inequality is from (\ref{3.3 compatibility 1}). By the relations (\ref{3.1 scaled projection relation}) and (\ref{3.2 rescale relation 3}),
\[
\Box^{(q)}_{(k),s}\Plks u(\sqrt{k}z)=k^{-1}\Box^{(q)}_{k,s}\Pks \left(u(\sqrt{k}z)\right).
\]By changing the variable,
\[
\|\Box^{(q)}_{(k),s}\Plks u\|_{\omega_{(k)},B(\sqrt{k})}=k^{-1}\|k^{n/2}\Box^{(q)}_{k,s}\Pks \left(u(\sqrt{k}z)\right)\|_{\omega,B(1)}.
\]Define  $u_k:=k^{n/2}u(\sqrt{k}z)$ which is a section supported in $U\subset M$ for large enough $k$.  Furthermore, by changing the variable, observe that $\|u_k\|_{\omega}=\|u\|_{\omega_{(k)}}$ and hence
\[
\|\Box^{(q)}_{(k),s}\Plks u\|_{\omega_{(k)},B(\sqrt{k})}=k^{-1}\|\Box^{(q)}_{k,s}\Pks u_k\|_{\omega,B(1)}\leq \frac{c_k}{k}\|u_k\|_{\omega}= \frac{c_k}{k}\|u\|_{\omega_{(k)}}\lesssim \frac{c_k}{k}\|u\|_{\omega_0}.\]Then we apply the assumption $\limsup_{k\To\infty}\frac{c_k}{k}=0$ to conclude that $\|\rho \Box^{(q)}_{0,s}\Bs u\|_{\omega_0}=0$. Since $\rho$ is arbitrary, we have $\Box^{(q)}_{0,s}\Bs u\equiv 0$.
 \end{proof}
 
Next, our main objective is to demonstrate that $\Bs:L^2_{\omega_0}(\Cn,\T)\To\Ker\Box^{(q)}_{0,s}$ is an orthogonal projection. Once this is established, we can infer from the uniqueness of the Schwartz kernel and see that the section $\Bsk(z,w)$ coincides with the Bergman kernel $B^{(q)}_{0,s}(z,w)$. By Theorem \ref{4.2 thm bd op BoxB=0}, it remains to show the following statement (cf. \cite[theorem 3.1 in section 3.1]{Yos}):
\begin{statement}\label{4.2 statement 2} 
$\Bs u=u$ for all $u \in \Ker \Box^{(q)}_{0,s}$.
\end{statement}
\begin{rmk}\label{4.2 rmk}
If the statement holds, we are able to complete the proof of the main theorems as follows.\par 
Under Assumption \ref{4.2 assumption}, by Theorem \ref{4.2 thm bd op BoxB=0} and Statement \ref{4.2 statement 2}, we know that the operator $\Bs$ defined in Def.\ref{4.2 def} must be the Bergman projection $\mathcal{B}^{(q)}_{0,s}$ in the model case. By the uniqueness of the Schwartz kernel, we have $\Bsk(z,w)\equiv B^{(q)}_{0,s}(z,w)$.\par 
According to Corollary \ref{3.3 loc conv cor}, we know that each subsequence of $P^{(q),s}_{(k),c_k}(z,w)$ has a subsequence that converges locally uniformly to $B^{(q),s}_{0}(z,w)=e^{\phi_0(z)}B^{(q)}_{0,s}(z,w)e^{-\phi_0(w)}$ in $\Cinf$. This means that $P^{(q),s}_{(k),c_k}(z,w)$ converges to $B^{(q),s}_{0}(z,w)$ locally uniformly in $\Cinf$. Finally, by applying Theorem \ref{4.1 Model theorem} and 
 the relation (\ref{kernel relation 2}), we complete the proof of the main theorem.
\end{rmk}

 In the case $p\notin M(q)$, Theorem \ref{4.1 Model theorem} tells us that $\Ker\Box^{(q)}_{0,s}=\{0\}$ and therefore $\Bs$ is a zero map. As a consequence, Statement \ref{4.2 statement 2} automatically holds, and hence we have the main theorem for the case $p\notin M(q)$.
 \begin{thm}[main theorem for $p\notin M(q)$]\label{4.2 main thm}
    If $p\notin M(q)$ and $ \limsup_{k\To\infty}\dfrac{c_k}{k}=0$, then the scaled localized spectral (or Bergman if $c_k=0$) kernel $\Plksk(z,w)\To 0$ locally uniformly in $\Cinf$ on $\Cn$. Also, by (\ref{kernel relation 2}),  $P^{(q),s}_{(k),c_k}(z,w)\To 0$ locally uniformly in $\Cinf$ on $\Cn$.   
\end{thm}   
In the remaining sections, we pay full attention to proving Statement \ref{4.2 statement 2} in case $p \in M(q)$.
\subsection{Asymptotic of the function case}\label{section 4.3}The discussion in Sections \ref{section 4.1} and \ref{section 4.2} is mainly in the context of localized spectral and Bergman kernels with localized Kodaira Laplacian. In this section, we would like to stay in the context of $P^{(q),s}_{(k),c_k}$ and $B^{(q),s}_{(k)}$ defined in Def.\ref{8.2} rather than the localized kernels.
First, we establish some notations. Define $\mathcal{P}^{(q)}_{(k),c_k}:L^2_{\omega_{(k)},\phi_{(k)}}(B(\sqrt{k}),\T\otimes L^{(k)})\To L^2_{\omega_{(k)},\phi_{(k)}}(B(\sqrt{k}),\T\otimes L^{(k)})$ by \[
\left(\mathcal{P}^{(q)}_{(k),c_k}u\right)(\sqrt{k}z)=\mathcal{P}^{(q)}_{k,c_k}\left(u(\sqrt{k}w)\right).
\]Denote $\mathcal{B}^{(q)}_{(k)}:=\mathcal{P}^{(q)}_{(k),0}$. Then, for $\eta\otimes s^{(k)}\in \Omega_c^{(0,q)}(B(\sqrt{k}),L^{(k)})$, we have 
\[
\mathcal{P}^{(q)}_{(k),c_k}(\eta\otimes s^{(k)})(z)=\int_{B(\sqrt{k})}P^{(q),s}_{(k),c_k}(z,w)\eta(w)dV_{\omega_{(k)}}\otimes s^{(k)}.
\]We now treat $s^{(k)}$ as the trivial section $1$ of the trivial vector bundle $\Cs\To\Cn$ restricted on $B(\sqrt{k})$ and define $\mathcal{P}^{(q),s}_{(k),c_k}:L^2_{\omega_{(k)},\phi_{(k)}}(B(\sqrt{k}),\T)\To L^2_{\omega_{(k)},\phi_{(k)}}(B(\sqrt{k}),\T)$ by
\[
\mathcal{P}^{(q),s}_{(k),c_k}u:=\int_{B(\sqrt{k})}P^{(q),s}_{(k),c_k}(z,w)u(w)dV_{\omega_{(k)}}.
\]For the case $c_k=0$, denote $\mathcal{B}^{(q),s}_{(k)}:=\mathcal{P}^{(q),s}_{(k),0}$. Recall that Assumption \ref{4.2 assumption} means \[
 P^{(q),s}_{(k),c_k}(z,w)\To B^{(q),s}(z,w):=e^{\phi_0(z)}B^{(q)}_{s}(z,w)e^{-\phi_0(w)}\quad \text{locally uniformly in }\,\Cinf.
\] Next, define $\mathcal{B}^{(q),s}:L^2_{\omega_0,\phi_0}(\Cn,\T)\To L^2_{\omega_0,\phi_0}(\Cn,\T)$ by \[\mathcal{B}^{(q),s}u(z):=\int_{\Cn}B^{(q),s}(z,w)u(w)dV_{\omega_0}.\] By Theorem \ref{4.2 Lem B op}, we have $\mathcal{B}^{(q),s}:L^2_{\omega_0,\phi_0}(\Cn,\T)\To\Ker\Box^{(q)}_{0}$ is bounded with its operator norm smaller than $1$. Moreover, Statement \ref{4.2 statement 2} is equivalent to following statement: 
\begin{statement}\label{4.2 statement 2 (1)}
 \begin{equation}
\mathcal{B}^{(q),s}u=u\quad \text{for all }\,u\in\Ker\Box^{(q)}_{0},    
\end{equation}   
\end{statement}

In this section, we focus on the case $q=0$, $p\in M(0)$ and prove the Statement \ref{4.2 statement 2 (1)}. Note that $\lambda_i>0$ for all $i=1,\cdots,n$ under the assumption $p\in M(0)$. We impose the conditions that $\limsup_{k\To\infty}\dfrac{c_k}{k}=0   
$ and \begin{equation*}
 \exists c<1 \,\text{such that }\, \liminf e^{2c\min{\lambda_i}k^{1/2}}c_k>0.
\end{equation*}  Let $\chi\in\Cinf_c(B(1),[0,1])$ be a cut-off function such that $\chi\mid_{B(\sqrt{c'})}\equiv 1$ where $c'$ is a number with $c<c'<1$. Define \[\chi_k:=\chi(\dfrac{z}{k^{1/4}}).\]

We now embark on the proof of Statement \ref{4.2 statement 2 (1)}. Given  $u\in\Ker\Box^{(0)}_{0}$, our strategy is to construct a sequence $u_{(k)}$ converging to $u$ such that 
\begin{equation*}
\mathcal{P}^{(0),s}_{(k),c_k}u_{(k)}- u_{(k)}\To 0 \quad \text{\rm and}\quad \mathcal{P}^{(0),s}_{(k),s}u_{(k)}- \mathcal{B}^{(0),s}_{s}u\To 0. 
\end{equation*} Define 
\[
u_{(k)}:=\chi_ku
\]which is clearly satisfies that $\|u_{(k)}-u\|_{\omega_0,\phi_0}\To 0$.  
By Theorem \ref{4.1 Model theorem}, we can see that 
\[
\Ker\Box^{(0)}_{0}=\text{span}\{z^{\alpha}\}_{\alpha\in\N_0^n}.
\]It is enough to check Statement \ref{4.2 statement 2 (1)} holds for the basis $\{z^{\alpha}\}_{\alpha\in\N_0^n}$ and hence we assume that $u$ is of the form $z^{\alpha}$. Now, we are going to show that $\mathcal{P}^{(0),s}_{(k),c_k} u_{(k)}-u_{(k)}\To 0$.

\begin{thm}\label{4.3 thm 2}If $u=z^{\alpha}$ for some $\alpha \in \N_0^n$, we have 
\[
\|(\mathcal{P}^{(0),s}_{(k),c_k}u_{(k)}-u_{(k)})\|_{\omega_{(k)},\phi_{(k)},B(\sqrt{k})}:=\|(\mathcal{P}^{(0),s}_{(k),c_k}u_{(k)}-u_{(k)})e^{-\phi_{(k)}}\|_{\omega_{(k)},B(\sqrt{k})}\To 0.
\]As for the case $c_k=0$ for all $k$, it also holds under the spectral gap condition 2 (Def. \ref{1.2 spectral gap 2}).
\end{thm}
\begin{proof}
 Define $u_{k}(z):=k^{n/2}u_{(k)}(\sqrt{k}z)$.
     Observe that $u_{k}\in \Cinf_c(B(1))\subset \Cinf_c(M)$. By rescaling, we see
    \[\|(\mathcal{P}^{(0),s}_{(k),c_k} u_{(k)}-u_{(k)})e^{-\phi_{(k)}}\|_{\omega_{(k)},B(\sqrt{k})}\leq\|\mathcal{P}^{(0)}_{k,c_k}(u_k\otimes s^k)-u_k\otimes s^k\|_{\omega,k\phi}.
    \]By the property of spectral kernel, we have
    \begin{align*}
    \|\mathcal{P}^{(0)}_{k,c_k}(u_k\otimes s^k)-u_k\otimes s^k\|^2_{\omega,k\phi}\leq\dfrac{1}{c_k}\left(\Box^{(0)}_{k}u_k\otimes s^k|u_k\otimes s^k\right)_{\omega,k\phi}
    &=\dfrac{1}{c_k}\|(\pb u_k)\otimes s^k\|^2_{\omega,k\phi}\\&=\frac{k}{c_k}\|(\pb u_{(k)})e^{-\phi_{(k)}}\|^2_{\omega_{(k)}}.
    \end{align*}
    Recalling the setting that $\phi(z)=\phi_0(z)+O(|z|^4)$, we get \[
   |e^{\phi_0(z)-\phi_{(k)}(z)}-1|\lesssim |\phi_0(z)-\phi_{(k)}(z)| \lesssim \dfrac{|z|^4}{k} \quad \text{\rm for all }\, |z|\leq k^{1/4}.
    \]Because $\supp \chi_k \subset B(k^{1/4})$, we can change the metrics $e^{-\phi_{(k)}}$ and $\omega_{(k)}$ by the estimate: 
    \[
    \frac{k}{c_k}\|(\pb u_{(k)})e^{-\phi_{(k)}}\|^2_{\omega_{(k)}}\lesssim\frac{k}{c_k}\|(\pb u_{(k)})e^{-\phi_{0}}\|^2_{\omega_{(k)}}\lesssim \frac{k}{c_k}\|(\pb u_{(k)})e^{-\phi_0}\|^2_{\omega_0}.
    \]The last inequality is by (\ref{3.3 compatibility 1}). By direct computation,  
    \begin{equation}\label{123456789}
    \|(\pb u_{(k)})e^{-\phi_0}\|^2_{\omega_0}\lesssim \int_{\sqrt{c'}k^{1/4}<|z|<k^{1/4}}|\pb\chi_k|^2 |z^{\alpha}|^2e^{-2\phi_0}dm\lesssim k^N e^{-c'\cdot 2\min{\lambda_i}\cdot k^{1/2}},    
    \end{equation}
    where $N$ is an integer depends on $\alpha$. Hence, by the condition $\liminf e^{2c\min{\lambda_i}\cdot k^{1/2}}c_k>0$, we obtain 
    
    \[
    \frac{k}{c_k}\|(\pb u_{(k)})e^{-\phi_0}\|_{\omega_0}\lesssim k^{N+1}e^{2(c-c')\min \lambda_i\cdot k^{1/2}}\To 0 \quad \text{since }\, c<c'.
    \]
   For the Bergman kernel case $c_k=0$, by the spectral gap condition \ref{1.2 spectral gap 2}, we repeat that
    \begin{align*}
     \|(\mathcal{B}^{(0),s}_{(k)}u_{(k)}-u_{(k)})&e^{-\phi_{(k)}}\|^2_{\omega_{(k)},B(\sqrt{k})}\leq\|\mathcal{B}^{(0)}_{k}(u_k\otimes s^k)-(u_k\otimes s^k)\|^2_{\omega,k\phi} \\&\lesssim e^{2c\min{\lambda_i}\cdot k^{1/2}}\left(\Box^{(0)}_{k}u_k\otimes s^k \mid u_k\otimes s^k\right)_{\omega,k\phi}\lesssim k^{N+1}e^{2(c-c')\min \lambda_i\cdot k^{1/2}}\To 0.   
    \end{align*}
   
\end{proof}
Before proving the main theorem for the function case, we need another lemma
about convergence. The following lemma is in the context of scaled localized spectral or Bergman kernels and is applicable to the $(0,q)$-forms cases for all $q=0,\cdots,n$.
\begin{lem}\label{4.3 lem 3}
 Let $u=z^{\alpha}_qe^{-\sum |\lambda_i||z^i|^2}d\zb^I$ for some $\alpha\in\N_0^n$ and $I\in\mathcal{J}_{q,n}$. For any $v\in \Omega^{(0,q)}_c(\Cn)$,
 \[
 \left(v \mid \Plks\chi_ku -\Bs u\right)_{\omega_0} \rightarrow 0 \quad \text{as }\,k\To\infty.
 \] 
\end{lem}
\begin{proof}
    Let $v\in\Omega^{(0,q)}_c(\Cn)$. For any fixed positive integer $n_0\in\N$, observe that for each $k\in\N$ with $k>n_0$, we can estimate that 
    \begin{multline*}
      |\left(v \mid \mathcal{P}^{(q)}_{(k),c_k,s}\chi_ku -\mathcal{B}^{(q)}_{s}u\right)_{\omega_0}|\leq |\left(v\mid (\mathcal{P}^{(q)}_{(k),c_k,s}\chi_k-\mathcal{B}^{(q)}_{s})\chi_{n_0} u\right)_{\omega_0}| \\+ \|v\|_{\omega_0}\|(\mathcal{P}^{(q)}_{(k),c_k,s}\chi_k-\mathcal{B}^{(q)}_{s})(\chi_{n_0}-1)u\|_{\omega_0}.
      \end{multline*}
      Moreover, by Lemma \ref{3.3 thm pf 1.5}, $\mathcal{P}^{(q)}_{(k),c_k,s} \chi_k$ are uniformly bounded linear functionals on the space $L^2_{\omega_0}(\Cn,\T)$. For this reason, $\mathcal{P}^{(q)}_{(k),c_k,s}\chi_k-\mathcal{B}^{(q)}_{s}$ are also uniformly bounded linear functionals on $L^2_{\omega_0}(\Cn,\T)$. \par Given an arbitrary number $\varepsilon>0$, since $\chi_{n_0}u\To u$ in $L^2_{\omega_0}(\Cn,\T)$ as $n_0\To\infty$, we can fix $n_0$ large enough such that
    \[
    \|v\|_{\omega_0}\|(\mathcal{P}^{(q)}_{(k),c_k,s}\chi_k-\mathcal{B}^{(q)}_{s})(\chi_{n_0}-1)u\|_{\omega_0}<\varepsilon/2 \quad \text{for all}\,k\in\N.
    \]
    Furthermore, by the assumption that $P^{(q)}_{(k),c_k,s}(z,w) \rightarrow B^{(q)}_{s}(z,w)$ locally uniformly, 
    \[
    \mid\left(v\mid (\mathcal{P}^{(q)}_{(k),c_k,s}\chi_k-\mathcal{B}^{(q)}_{s})\chi_{n_0} u\right)_{\omega_0}\mid \rightarrow 0 \quad \text{as }\, k\To\infty.
    \] Finally, combining the above estimates, we obtain $|\left(v |\mathcal{P}^{(q)}_{(k),c_k,s}\chi_ku -\mathcal{B}^{(q)}_{s}u\right)_{\omega_0}|<\varepsilon$ for large enough $k$.

\end{proof}
To apply the Lemma in the context of $\mathcal{P}^{(0),s}_{(k),c_k}$ and $\mathcal{B}^{(0),s}_{(k)}$, we simply deduce the following corollary by the relation (\ref{kernel relation 2}) and the fact that $\phi_{(k)}\To\phi_0$ locally uniformly.
\begin{cor}\label{8.222}
    Let $u=z^{\alpha}$ for some $\alpha\in\N_0^n$. For any $v\in\Cinf_c(\Cn)$,\[
    \left(v|\mathcal{P}^{(q),s}_{(k),c_k}u_{(k)}-\mathcal{B}^{(q),s}u\right)_{\omega_0}\To 0 \,\, \text{as } \, k\To 0.
    \]
\end{cor}

Now, we are able to complete Statement \ref{4.2 statement 2} in Section \ref{section 4.2} for the function case when $p\in M(0)$.
\begin{thm}\label{4.3 main thm 1}
If $c_k$ satisfies the conditions $\limsup_{k\To\infty}\dfrac{c_k}{k}=0$ and \[\liminf_{k\To\infty} e^{2c\min{\lambda_i}\cdot k^{1/2}}c_k>0\] for some constant $c<1$, then
\begin{equation*}
    B^{(0),s} u=u \quad \text{\rm for all }\,u \in \Ker \Box^{(0),s}_{0}. 
\end{equation*}As for the Bergman kernel case $c_k=0$ for all $k$, it also holds under the spectral gap condition of suitable exponential rate (cf. Def. \ref{1.2 spectral gap 2}).
\end{thm}
\begin{proof}
   Assume that $u$ is of the form $u=z^{\alpha}$. To show $B^{(0)}_{s}u=u$, let $v\in\Cinf_{c}(\Cn)$ and observe that
    \begin{multline*}
       \left(v\mid \mathcal{B}^{(0),s} u-u\right)_{\omega_0} =\left(v\mid \mathcal{B}^{(0),s} u-\mathcal{P}^{(0),s}_{(k),c_k} u_{(k)}\right)_{\omega_0}+\left(v\mid \mathcal{P}^{(0),s}_{(k),c_k} u_{(k)}-u_{(k)}\right)_{\omega_0}\\+\left(v\mid u_{(k)}-u\right)_{\omega_0}.  
    \end{multline*} 
    By Theorem \ref{4.3 thm 2}, Corollary \ref{8.222} and the fact that $\omega_{(k)}\To\omega_0$ and $\phi_{(k)}\To\phi_0$ locally uniformly, the right-hand side of the equality above tends to zero. This means $B^{(0),s}u\equiv u$ because $v$ is arbitrary.
\end{proof}
By Remark \ref{4.2 rmk}, we obtain the main theorem for the function case when $p\in M(0)$.
\begin{thm}[main theorem for function case]\label{4.3 main thm 2}
    Suppose $c_k$ is a sequence with \[\limsup_{k\To\infty}\dfrac{c_k}{k}=0\] and $p \in M(0)$. If there exists a constant $c<1$ such that $\liminf_{k\To\infty}{e^{2c\min{\lambda_i}\cdot k^{1/2}}c_k}>0$, then 
    \begin{equation*}
        P^{(q),s}_{(k)}(z,w)\To\dfrac{\lambda_1 \cdots \lambda_n}{\pi^n}\,e^{2(\sum_{i=1}^{q}\lambda_i\zb^i w^i+\sum_{i=q+1}^{n}\lambda_iz^i\wb^i-\sum_{i=1}^{n}\lambda_i|w^i|^2)}.
        \end{equation*}locally uniformly in $\Cinf$ on $\Cn$.
    In the Bergman kernel case $c_k=0$, the convergence also holds under the 
 small spectral gap condition of suitable exponential rate in $U$ (cf. Def \ref{1.2 spectral gap 2}).
\end{thm}
\begin{rmk}
 The proof in this section is not valid for the $(0,q)$-forms if $q\neq 0$. The reason is that for $|z|<c'k^{1/4}$ and $u\in\Ker \Box^{(q)}_{0}$, the equation \[\pb^*_{k}u(\sqrt{k}z)=0\] may not be true if $q \neq 0$. In the context of localized Kodaira Laplacian, we need to adjust $u$ from the space $\Ker \Box^{(q)}_{0,s}$ to the space $\Ker \Box^{(q)}_{(k),s}$. It is natural to orthogonally project $u$ from $\Ker \Box^{(0)}_{0,s}$ into space $\Ker \Box^{(q)}_{(k),s}$. However, we encounter a difficulty as we lack information about the Bergman projection corresponding to $\Box^{(q)}_{(k),s}$. One potential solution is to extend the Laplacian $\Box^{(q)}_{(k),s}$ to $\Box^{(q)\sim}_{(k),s}$ defined on the whole $\Cn$, where the Bergman kernel with respect to the extended Laplacian $\Box^{(q)\sim}_{(k),s}$ is tractable. This is the main idea of Section  \ref{section 4.4} and Section  \ref{section 4.5}.
 \end{rmk}
\newpage
\subsection{The spectral gap of the extended Laplacian on $\Cn$}\label{section 4.4}
In this section, we will extend the localized scaled Laplacian $\Box^{(q)}_{(k),s}$ which is defined on $B(\sqrt{k})$ to the whole $\Cn$. The extended localized Laplacian is identical to $\Box^{(q)}_{(k),s}$ in $B(k^{\epsilon})$ where $\epsilon$ will be determined later in Section \ref{section 4.5}. \par
From now on, we fix a cut-off function denoted by $\chi \in \Cinf_c(\Cn)$ such that its support is contained within the ball $B(2)$, and is identical to $1$ on the ball $B(1)$.  Let us choose a number $\epsilon$ such that $0<\epsilon <1/6$ and define the extended metric data on $\Cn$ by
\begin{equation*}
    \tphi_{(k)}(z):=\chi(\dfrac{z}{k^{\epsilon}})\phi_{(k)}(z)+\left(1-\chi(\dfrac{z}{k^{\epsilon}})\right)\phi_0(z)
\end{equation*} and the extended Hermitian form by 
\begin{equation*}
    \tomega_{(k)}(z):=\chi(\dfrac{z}{k^{\epsilon}})\omega_{(k)}(z)+\left(1-\chi(\dfrac{z}{k^{\epsilon}})\right)\omega_0(z). 
\end{equation*}Recall the observations (\ref{3.1 ob phi}) and (\ref{3.1 ob omega}). Since $\epsilon<1/6$, we have the uniform convergences
\begin{equation*}
    \| \tphi_{(k)}-\phi_0\|_{\Ct} \rightarrow 0 \quad \text{\rm and } \quad \| \tomega_{(k)}-\omega_0\|_{\Ct} \rightarrow 0.
\end{equation*}Denote 
\begin{align*}
\tpb^{(q)}_{(k),s}:\Omega^{(0,q)}(\Cn) \To \Omega^{(0,q+1)}(\Cn)\quad;\quad
\tpb^{(q)}_{(k),s}:\Omega^{(0,q)}(\Cn) \To \Omega^{(0,q-1)}(\Cn)
\end{align*} to be the localized Cauchy-Riemann operator and its formal adjoint which are given by
\begin{align*}
\tpb^{(q)}_{(k),s}&=\pb^{(q)}+(\pb \tphi_{(k)})\wedge\cdot;\\
  \tpb^{*,(q)}_{(k),s}&=\pb^{*,(q)}_{\tomega_{(k)}}+(\pb \tphi_{(k)})\wedge^*_{\tomega_{(k)}},
\end{align*}respectively. Denote
\[\Box^{(q)\sim}_{(k),s}=\tpb_{(k),s}^*\tpb_{(k),s}+\tpb_{(k),s}\tpb_{(k),s}^*:\Dom{\Tilde{\Box}^{(q)\sim}_{(k),s}}\subset L^2_{\omega_0}(\Cn,\T) \rightarrow L^2_{\omega_0}(\Cn,\T)\] as the Gaffney extension of the localized Kodaira Laplacian with respect to the Hermitian form $\tomega_{(k)}$ and the weight function $\tphi_{(k)}$. It follows immediately from the constructions that $\pb_{(k),s}\equiv\tpb_{(k),s}$, 
 $\pb^*_{(k),s}\equiv\tpb^*_{(k),s}$ and $\Box^{(q)\sim}_{(k),s}\equiv\Box^{(q)}_{(k),s}$ in $B(k^{\epsilon})$. Reasonably, we call the $\Box^{(q)\sim}_{(k),s}$ \textbf{extended localized  Laplacian}.\\  \par 
We suppose $\lambda_i <0$ for all $i=1,\cdots ,q_0$ ; $\lambda_i>0$ for all $i=q_0+1,\cdots ,n$. Then there exists a constant $c>0$ such that for all $z\in\Cn$, 
\begin{equation}\label{4.4 phi est}
    \dfrac{\p^2 \tphi_{(k)}}{\p z^i \p \zb^i}(z)<-c\quad \forall \, i=1,\cdots ,q_0
\quad \text{\rm and} \quad 
     \dfrac{\p^2 \tphi_{(k)}}{\p z^i \p \zb^i}(z)>c\quad \forall \, i=q_0+1,\cdots ,n.
\end{equation}The following results tell us these estimates  create a uniform lower bound of the first eigen-values of $\Box^{(q)\sim}_{(k),s}$. 
\begin{lem}\label{4.4 lem 1}
    For $q \neq q_0$, there is a constant $c>0$ such that for all $u \in \Dom \Box^{(q)\sim}_{(k),s}$,
    \begin{equation*}\left(\Box^{(q)\sim}_{(k),s}u\mid u\right)_{\tomega_{(k)}}=\|\tpb_{(k),s}u\|^2_{\tomega_{(k)}}+ \|\tpb^{*}_{(k),s}u\|^2_{\tomega_{(k)}}>c\|u\|^2_{\tomega_{(k)}}.
    \end{equation*}Therefore, $\|\Box^{(q)\sim}_{(k),s}u\|_{\tomega_{(k)}}>c\|u\|_{\tomega_{(k)}}$.
\end{lem}
\begin{proof}
    Note that 
\begin{align}
 \label{4.4 lem 1 (1)}    
\|\tpb_{(k),s}u\|^2_{\tomega_{(k)}}&=\|\left(\pb+(\pb \tphi_{(k)})\wedge\right)u\|^2_{\tomega_{(k)}} \gtrsim \|\left(\pb+(\pb \tphi_{(k)})\wedge\right)u\|^2_{\omega_{0}};\\
\|\tpb^*_{(k),s}u\|^2_{\tomega_{(k)}}&=\|\left(\pb^*_{\tomega_{(k)}}+(\pb \tphi_{(k)})\wedge^*_{\tomega_{(k)}}\right)u\|^2_{\tomega_{(k)}} \gtrsim \|\left(\pb^*_{\omega_{0}}+(\pb \tphi_{(k)})\wedge^*_{\omega_0}\right)u\|^2_{\omega_{0}}.\notag
\end{align} \par 
   Let $u =fd\zb^I$ for some $f \in \Cinf_c(\Cn)$ and $I\in \mathcal{J}_{q,n}$. Since $q \neq q_0$, there exists $i \in \{1,\cdots , n\}$ such that at least one of the following two cases holds:
    \begin{itemize}
        \item $i \notin I$ and  $\lambda_i<0$;
        \item $i \in I$ and  $\lambda_i>0$.
    \end{itemize}If the first case holds, \begin{align}\label{4.4 lem 1 (2)}
 \| \left(\pb+(\pb \tphi_{(k)})\wedge\right)u\|^2_{\omega_0} & \geq \int_{\Cn} \mid \dfrac{\p f}{\p \zb^i}+\dfrac{\p \tphi_{(k)}}{\p \zb^i}f \mid ^2 dm  = \int_{\Cn} (\dfrac{\p f}{\p \zb^i}+\dfrac{\p \tphi_{(k)}}{\p \zb^i}f)(\dfrac{\p \Bar{f}}{\p z^i}+\dfrac{\p \Bar{\tphi}_{(k)}}{\p z^i}\Bar{f}) dm \notag \\
 &=\int_{\Cn}|\dfrac{\p f}{\p \zb^i}|^2+\Bar{f}\dfrac{\p f}{\p \zb^i}\dfrac{\p \Bar{\tphi}_{(k)}}{\p z^i}+f\dfrac{\p \Bar{f}}{\p z^i}\dfrac{\p \tphi_{(k)}}{\p \zb^i}+|\dfrac{\p \tphi_{(k)}}{\p \zb^i}|^2|f|^2dm\notag. 
\end{align}By integration by part, we compute that $\int_{\Cn}|\dfrac{\p f}{\p \zb^i}|^2dm=\int_{\Cn}|\dfrac{\p f}{\p z^i}|^2dm$ and
\begin{align*}
    \int_{\Cn}\Bar{f}\dfrac{\p f}{\p \zb^i}\dfrac{\p \Bar{\tphi}_{(k)}}{\p z^i}+f\dfrac{\p \Bar{f}}{\p z^i}\dfrac{\p \tphi_{(k)}}{\p \zb^i}dm = \int_{\Cn}-2|f|^2\dfrac{\p^2 \tphi_{(k)}}{\p z^i \p \zb^i}-f\dfrac{\p \Bar{f}}{\p \zb^i}\dfrac{\p\Bar{\tphi}_{(k)}}{\p z^i}-\Bar{f}\dfrac{\p f}{\p z^i}\dfrac{{\p\tphi}_{(k)}}{\p \zb^i}dm.
\end{align*}Applying these two equations and  
$|\dfrac{\p f}{\p z^i}|^2+|\dfrac{\p \tphi_{(k)}}{\p \zb^i}|^2|f|^2-2|f||\dfrac{\p f}{\p z^i}||\dfrac{{\p\tphi}_{(k)}}{\p \zb^i}| \geq 0$ , we have 
 \begin{equation}\label{4.4 lem 1 (3)}
  \| \left(\pb+(\pb \tphi_{(k)})\wedge\right)u\|^2_{\omega_0} \geq -2\int_{\Cn}|f|^2\dfrac{\p^2 \tphi_{(k)}}{\p z^i \p \zb^i}dm  \gtrsim -\inf\left(\dfrac{\p^2 \tphi_{(k)}}{\p z^i \p \zb^i}\right)\|f\|^2_{\tomega_{(k)}} .  
 \end{equation}On the other hand, if the second case holds,    
\begin{align*}
    \|\left(\pb^*_{\omega_0}+(\pb \tphi_{(k)})\wedge^*_{\omega_0}\right)u\|^2_{\omega_0} &\geq \int_{\Cn} (-\dfrac{\p f}{\p z^i}+\dfrac{\p \tphi_{(k)}}{\p z^i}f)(-\dfrac{\p \Bar{f}}{\p \zb^i}+\dfrac{\p \Bar{\tphi}_{(k)}}{\p \zb^i}\Bar{f}) dm \notag\\
 &=\int_{\Cn}|\dfrac{\p f}{\p z^i}|^2-\Bar{f}\dfrac{\p f}{\p z^i}\dfrac{\p \Bar{\tphi}_{(k)}}{\p \zb^i}-f\dfrac{\p \Bar{f}}{\p \zb^i}\dfrac{\p \tphi_{(k)}}{\p z^i}+|\dfrac{\p \tphi_{(k)}}{\p z^i}|^2|f|^2dm. \notag
\end{align*}By integration by part again, we have  $\int_{\Cn}|\dfrac{\p f}{\p z^i}|^2dm=\int_{\Cn}|\dfrac{\p f}{\p \zb^i}|^2dm$ and 
\begin{equation*}
 \int_{\Cn}-\Bar{f}\dfrac{\p f}{\p z^i}\dfrac{\p \Bar{\tphi}_{(k)}}{\p \zb^i}-f\dfrac{\p \Bar{f}}{\p \zb^i}\dfrac{\p \tphi_{(k)}}{\p z^i}dm = \int_{\Cn}2|f|^2\dfrac{\p^2 \tphi_{(k)}}{\p z^i \p \zb^i}+f\dfrac{\p \Bar{f}}{\p z^i}\dfrac{\p\Bar{\tphi}_{(k)}}{\p \zb^i}+\Bar{f}\dfrac{\p f}{\p \zb^i}\dfrac{{\p\tphi}_{(k)}}{\p z^i}dm.   
\end{equation*}Combining equations above and  $|\dfrac{\p f}{\p \zb^i}|^2+|\dfrac{\p \tphi_{(k)}}{\p z^i}|^2|f|^2-2|f||\dfrac{\p f}{\p \zb^i}||\dfrac{{\p\tphi}_{(k)}}{\p z^i}| \geq 0$, 
\begin{equation}\label{4.4 lem 1 (4)}
 \|\left(\pb^*_{\omega_0}+(\pb \tphi_{(k)})\wedge^*_{\tomega_{(k)}}\right)u\|^2_{\omega_0} \geq 
 2 \int_{\Cn}|f|^2 \dfrac{\p^2 \tphi_{(k)}}{\p z^i \p \zb^i} dm \gtrsim \inf\left(\dfrac{\p^2 \tphi_{(k)}}{\p z^i \p \zb^i}\right)\|f\|^2_{\tomega_{(k)}} .  
\end{equation}By (\ref{4.4 lem 1 (1)}),(\ref{4.4 lem 1 (3)}) and (\ref{4.4 lem 1 (4)}), we have completed the proof for the case $u\in \Omega^{(0,q)}_c(\Cn)$. Next, we are able to prove the lemma by density argument. The density argument here is somehow technical and based on the \textbf{Friedrich's Lemma} (cf.\cite[Chapter 7, Lemma 3.3]{Dem}). For the details of approximation, readers may consult \cite[Lemma 5]{Hou}. 
\end{proof}

\begin{cor}\label{4.4 cor N}
 For $q \neq q_0$, the extended Laplacians $\Box^{(q)\sim}_{(k),k}$ is bijective and has  inverses \[N^{q}_{k}:L^2_{\tomega_{(k)}}(\Cn,\T) \rightarrow \Dom \Box^{(q)\sim}_{(k),s}\] which is a $k$-uniformly bounded operator.    
\end{cor}
\begin{proof}
    According to Lemma \ref{4.4 lem 1}, $\Box^{(q)\sim}_{(k),s}$ is injective. To show the surjectivity, we choose an arbitrary $v\in L^2_{\tomega_{(k)}}(\Cn,\T)$ and consider the linear functional $\Tl_{v}$ on $\Rang \, \Box^{(q)\sim}_{(k),s}$ given by 
    \[
    \Tl_{v}(\Box^{(q)\sim}_{(k),s}\,u)=\left(u \mid v\right)_{\tomega_{(k)}} \quad \forall u\in \Dom \, \Box^{(q)\sim}_{(k),s}.
    \]Lemma \ref{4.4 lem 1} implies that $\|\Tl_{v}\|_{\tomega_{(k)}} \leq \dfrac{\|v\|_{\tomega_{(k)}}}{c}$ for a constant $c$ independent of $v$ and $k$. By the Hahn-Banach Theorem, the functional $\Tl_{v}$ can be extended to a bounded linear functional on $L^2_{\tomega_{(k)}}(\Cn,\T)$ with the same norm. By Riesz representation theorem, there exists a representative $\Tilde{v} \in L^2_{\tomega_{(k)}}(\Cn,\T)$ such that 
    \[
    \left(u \mid v \right)_{\tomega_{(k)}}=\Tl_{v}(\Box^{(q)\sim}_{(k),s}u)=\left(\Box^{(q)\sim}_{(k),s}u\mid\Tilde{v}\right)_{\tomega_{(k)}} \quad \forall u\in \Dom \, \Box^{(q)\sim}_{(k),s}.
    \]This means $\Box^{(q)\sim}_{(k),s}\Tilde{v}=v$ which proves the surjectivity. Define $N^{q}_{k}$ such that $N^{q}_{k}v=\Tilde{v}$.  Lemma \ref{4.4 lem 1} implies $\|N^{q}_{k}\|_{\tomega_{(k)}} \leq C$ for a constant $C$ independent of $k$.
\end{proof}
We have shown that when $q\neq q_0$, the extended Laplacian $\Box^{(q)\sim}_{(k),s}$ has a uniform spectral gap $\text{ spec }\Box^{(q)\sim}_{(k),s}\subset [c,\infty)$ for a positive constant $c$ independent of $k$. Next, in the case $q=q_0$, we should prove that the uniform spectral gap also holds in the sense that $\text{\rm spec }\Box^{(q)\sim}_{(k),s}\subset \{0\}\cup [c,\infty)$. Define \[
\tBlks:L^2_{\tomega_{(k)}}(\Cn,\T) \rightarrow \Ker \, \Box^{(q)\sim}_{(k),s} \subset L^2_{\tomega_{(k)}}(\Cn,\T)
\]to be the Bergman projection and $\tBlksk$ to be the Bergman kernel. The following representation of $\tBlksk$ is standard.
\begin{thm}[Hodge decomposition]\label{4.4 thm B=I-} We have the expression
\begin{equation}\label{4.4 thm B (1)}
\Tilde{\mathcal{B}}^{(q_0)}_{(k),s}=\Id-\tpb^{(q_0-1)}_{(k),s}N^{q_0-1}_{k}\tpb^{*,(q_0)}_{(k),s}-\tpb^{*,(q_0+1)}_{(k),s}N^{q_0+1}_{(k)}\tpb^{(q_0)}_{(k),s} \quad \textit{on}\,\, \Omega^{(0,q)}_c(\Cn).    
\end{equation}
Here, $N^{q}_{k}$ is the inverse of the Laplacian $\Box^{(q)\sim}_{k,s}$ established in Corollary \ref{4.4 cor N}.
\end{thm}
\begin{proof}Note that
    \begin{align*}
&\Box^{(q_0)\sim}_{(k),s}\left(\Id-\tpb^{(q_0-1)}_{(k),s}N^{q_0-1}_{k}\tpb^{*,(q_0)}_{(k),s}-\tpb^{*,(q_0+1)}_{(k),s}N^{q_0+1}_{k}\tpb^{(q_0)}_{(k),s}\right)\\&=\tpb_{(k),s}\tpb^*_{(k),s}+\tpb_{(k),s}^*\tpb_{(k),s}-\tpb_{(k),s}\tpb^*_{(k),s}\tpb_{(k),s} N^{q_0-1}_{k}\tpb^*_{(k),s}-\tpb^{*}_{(k),s}\tpb_{(k),s}\tpb_{(k),s}^{*}N^{q_0+1}\tpb_{(k),s} \\
       &=\tpb_{(k),s}\tpb^*_{(k),s}+\tpb^*_{(k),s}\tpb_{(k),s}-\tpb_{(k),s} \Box^{(q_0-1)\sim}_{(k),s} N^{q_0-1}_k\tpb^*_{(k),s}-\tpb^{*}_{(k),s}\Box^{(q_0+1)\sim}_{(k),s}N^{q_0+1}_k\tpb_{(k),s}\\&=\tpb_{(k),s}\tpb^*_{(k),s}+\tpb^*_{(k),s}\tpb_{(k),s}-(\tpb_{(k),s}\tpb^*_{(k),s}+\tpb^*_{(k),s}\tpb_{(k),s})=0.
    \end{align*}So the right-hand side of (\ref{4.4 thm B (1)}) has its image in $\Ker \Box^{(q)\sim}_{(k),s}$. It remains to show that $\Rang \left(\tpb^{(q_0-1)}_{(k),s}N^{q_0-1}\tpb^{*,(q_0)}_{(k),s}-\tpb^{*,(q_0+1)}_{(k),s}N^{q_0+1}\tpb^{(q_0)}_{(k),s}\right) \perp \Ker \Box^{(q_0)\sim}_{(k),s}$. Given $u\in \Omega^{(0,q)}_c(\Cn)$ and $v\in \Ker \Box^{(q_0)\sim}_{(k),s}$, since $\tpb^*_{(k),s}v=\tpb_{(k),s}v=0$,
    \begin{multline*}
       \left((\tpb_{(k),s} N^{q_0-1}_k\tpb^*_{(k),s}-\pb^{*}N^{q_0+1}_k\pb)u \mid v\right)_{\tomega_{(k)}}\\ =\left(N^{q_0-1}_k\tpb_{(k),s} u \mid \tpb^*_{(k),s} v\right)_{\tomega_{(k)}}+\left(N^{q_0+1}_k\tpb_{(k),s} u \mid \tpb_{(k),s}v\right)_{\tomega_{(k)}}=0. 
    \end{multline*}
\end{proof}
We now deduce some identities which will be frequently utilized. Compute that 
\begin{align*}
\|\tpb_{(k),s}\tpb^*_{(k),s}N^{q_0-1}_{k}\tpb^{*}_{(k),s}u\|^2_{\tomega_{(k)}}&= \left(\tpb^*_{(k),s}\tpb_{(k),s}\tpb^*_{(k),s}N^{q_0-1}_{k}\tpb^{*}_{(k),s}u \mid \tpb^*_{(k),s}N^{q_0-1}_{k,s}\tpb^{*}_{(k),s}u \right)_{\tomega_{(k)}}\\
     &=\left(\tpb^*_{(k),s}\Box^{(q_0-1)\sim}_{(k),s}N^{q_0-1}_{k}\tpb^{*}_{(k),s}u \mid \tpb^*_{(k),s}N^{q_0-1}_{k}\tpb^{*}_{(k),s}u \right)_{\tomega_{(k)}}\\&=\left(\tpb^*_{(k),s}\tpb^{*}_{(k),s}u \mid \tpb^*_{(k),s}N^{q_0-1}_{k}\tpb^{*}_{(k),s}u \right)_{\tomega_{(k)}}=0,    
\end{align*}for all $u\in\Omega^{(0,q_0)}_c(\Cn)$. Similarly, we can compute that $\|\tpb^{*}_{(k),s}\tpb_{(k),s}N^{q_0+1}_{k}\tpb_{(k),s}u\|^2_{\tomega_{(k)}}=0$ for all $u\in\Omega^{(0,q_0)}_c(\Cn)$. Hence, we have 
\begin{equation}\label{4.4 eq 1}
\tpb_{(k),s}\tpb^{*}_{(k),s}N^{q_0-1}_{k}\tpb^{*}_{(k),s}=0\quad;\quad \tpb^{*}_{(k),s}\tpb_{(k),s}N^{q_0+1}_{k}\tpb_{(k),s}=0 \quad \text{on }\, \Omega^{(0,q_0)}_c(\Cn).  
\end{equation}
Moreover, we can apply the two equations above to see that 
\begin{equation}\label{4.4 eq 2}
\tpb^*_{(k),s}\tpb_{(k),s}N^{q_0-1}_{k}\tpb^*_{(k),s}=\tpb^*_{(k),s}\quad ;\quad \tpb_{(k),s}\tpb^*_{(k),s}N^{q_0-1}_{k}\tpb_{(k),s}=\tpb_{(k),s}\quad\text{on }\, \Omega^{(0,q_0)}_c(\Cn).\end{equation}

\begin{thm}[uniform spectral gap for $\Box^{(q_0)\sim}_{(k),s}$]\label{4.4 thm spectral gap}
There exists a constant $c$ independent of $k$ such that
 \[\|\tBlkso u-u\|^2_{\tomega_{(k)}} \leq c\left(\|\tpb^{*,(q_0)}_{(k),s}u\|^2_{\tomega_{(k)}}+\|\tpb^{(q_0)}_{(k),s}u\|^2_{\tomega_{(k)}}\right)  \quad \text{\rm on }\, \Omega^{(0,q_0)}_c(\Cn).\]  
\end{thm}
\begin{proof}
  By Lemma \ref{4.4 thm B=I-}, \[\tBlkso-I=-\tpb^{(q_0-1)}_{(k),s}N^{q_0-1}_{k}\tpb^{*,(q_0)}_{(k),s}-\tpb^{*,(q_0+1)}_{(k),s}N^{q_0+1}_{k}\tpb^{(q_0)}_{(k),s} \quad \text{\rm on } \, \Omega^{(0,q_0)}_c(\Cn).\]Given $u \in \Omega^{(0,q_0)}_c(\Cn)$, 
  \begin{align*}
   \|\tpb^{(q_0-1)}_{(k),s}N^{q_0-1}_{k}\tpb^{*,(q_0)}_{(k),s}u\|^2_{\tomega_{(k)}} &= \left(N^{q_0-1}_{k}\tpb^{*}_{(k),s}u\mid \tpb^{*}_{(k),s} \tpb_{(k),s}N^{q_0-1}_{k}\tpb^{*}_{(k),s}u\right)_{\tomega_{(k)}}\\&\leq \|N^{q_0-1}_{k}\tpb^{*}_{(k),s}u\|_{\tomega_{(k)}} \|\tpb^{*}_{(k),s}\tpb_{(k),s}N^{q_0-1}_{k}\tpb^{*}_{(k),s}u\|_{\tomega_{(k)}}\\&\lesssim \|\tpb^{*}_{(k),s}u\|^2_{\tomega_{(k)}}.
  \end{align*}The last inequality is from Corollary \ref{4.4 cor N} and (\ref{4.4 eq 2}). Symmetrically, we can show that
  \begin{equation*}
\|\tpb^{*,(q_0+1)}_{(k),s}N^{q_0+1}_k\tpb^{(q_0)}_{(k),s}u\|^2_{\tomega_{(k)}}\leq c\|\tpb^{(q_0)}_{(k),s}u\|^2_{\tomega_{(k)}}.
  \end{equation*}The two estimates above imply the theorem.
\end{proof}

\subsection{Asymptotics of the general $(0,q)$-forms cases}\label{section 4.5}

In this section, we adopt Assumption \ref{4.2 assumption} in Section \ref{section 4.2} and consider the case $p\in M(q)$ where $q\in \{1,\cdots,n\}$. We strengthen the condition of $c_k$ by imposing $\limsup_{k\To\infty}\dfrac{c_k}{k}=0$ and  
\[
\exists  d\in\ \N\quad \text{\rm such that}\quad \liminf_{k\To\infty}{k^dc_k}>0.
\]The goal of this section is to show Statement \ref{4.2 statement 2} in Section \ref{4.2 statement 2}.
 
By rearrangement, we let $\lambda_i <0$ for all $i=1,\cdots,q$ and $\lambda_i>0$ for all $i=q+1,\cdots,n$ for simplicity. Recall \[z^{\alpha}_q:=(\zb^1)^{\alpha_1}\cdots(\zb^q)^{\alpha_q}(z^{q+1})^{\alpha_{q+1}}\cdots (z^{n})^{\alpha_{n}} \quad;\quad I:=(1,\cdots,q)\in \mathcal{J}_{q,n}.\]
We now adopt the settings in Section \ref{section 4.4}. It is important to note that in the construction of $\tomega_{(k)}$ and $\tphi_{(k)}$, we impose the condition that $0<\epsilon<1/6$. Now, we require 
\[
0<\epsilon<\min\{\dfrac{1}{2n+1},\dfrac{1}{6}\}.
\] The reason is in the proof of Theorem \ref{4.5 thm 1}.\par 
We establish the notations of cut-off functions. Recall that $\chi \in \Cinf_c(\Cn)$ is the cut-off function which is fixed at the beginning of Section \ref{section 4.4}. Choose $\rho \in \Cinf_c(\Cn)$ as another cut-off function such that $\overline{\supp \rho} \subset\{z\in\Cs; 2/7<|z|<1\}$ and $\rho\equiv 1$ on $\{z\in\Cs; 3/7<|z|<6/7\}$. Construct a sequence of cut-off functions by
\begin{equation}
    \chi_k(z):=\chi(\dfrac{7z}{k^{\epsilon}})\quad;\quad\Tilde{\chi}_k(z):=\chi(\dfrac{7z}{3k^{\epsilon}})\quad;\quad\rho_k(z):=\rho(\dfrac{z}{k^{\epsilon}}).
\end{equation}Observe that $\supp \chi_k \subset \{z\in\Cs;\,|z|<(2/7)k^{\epsilon}\}$ and $\supp \Tilde{\chi}_k \subset \{z\in\Cs;\,|z|<(6/7)k^{\epsilon}\}$. Moreover, the derivatives of $\Tilde{\chi}_k$ are supported in the annuli $\{z\in\Cs;\,(3/7)k^{\epsilon}<|z|<(6/7)k^{\epsilon}\}$ and the support of $\rho_k$ are in the annuli $\{z\in\Cs;\,(2/7)k^{\epsilon}<|z|<k^{\epsilon}\}$. 
 Next, we define the following convention.
\begin{defin}
 For any $u\in L^2_{\omega_0}(\Cn,\T)$, define 
 \begin{equation*}
u_{(k)}:=\Tilde{\chi}_k \ \tBlks \ \chi_k u.
\end{equation*}
\end{defin}
Now, the stage is set to demonstrate Statement \ref{4.2 statement 2}. The strategy is similar to Section \ref{section 4.3} except for the different constructions of $u_{(k)}$ resulting in the different estimates. First, our objective is to show the convergence $u_{(k)}\To u$ in $L^2_{\tomega}(\Cn,\T)$ as $k\To\infty$ if $u\in\Ker \Box^{(q)}_0$. 

\begin{lem}\label{4.5 thm 1 lem 1}
    Let $u=z^{\alpha}_q e^{-\sum_{i=1}^{n}|\lambda_i||z^i|^2}d\zb^I$ for some $\alpha\in \N_0^{n}$. There exists a constant $C$ such that for large enough $k$,
    \begin{equation}
        |\tpb^{(q)}_{(k),s}u|_{\omega_0}+|\tpb^{*,(q)}_{(k),s}u|_{\omega_0}\leq \dfrac{C}{\sqrt{k}}
    \end{equation}for all $|z|< k^{\epsilon}$.
\end{lem}
\begin{proof}
 Denote $u=:fd\zb^I$ where $f=z^{\alpha}_{q}e^{-\sum_{i=1}^n|\lambda_i||z^i|^2}$. By the formulas (\ref{3.2 pbs eq}),(\ref{3.2 pbs* eq}) and (\ref{3.2 Lemma Lap pb* eq}), we can write down the expression of $(\tpb_{(k),s}-\pb_{(k),s})u$ and $(\tpb^*_{(k),s}-\pb^*_{(k),s})u$ as 
 \begin{align*}
  (\tpb_{(k),s}-\pb_{0,s})u&=\left(\pb+(\pb \tphi_{(k)})\wedge\right)u-\left(\pb-(\pb \phi_0)\wedge\right)u=\left(\pb(\tphi_{(k)}-\phi_0)\right)\wedge u;
  \\(\tpb^*_{(k),s}-\pb^*_{0,s})u&=\left(\pb^*_{\tomega_{(k)}}+(\pb \tphi_{(k)})\wedge^*_{\tomega_{(k)}}\right)u-\left(\pb^*_{\omega_0}+(\pb \phi_0)\wedge^*_{\omega_{0}}\cdot\right)u\\ =&\left(-\dfrac{\p f}{\p z^i}-f\dfrac{\p \tvarphi_{(k)}}{\p z^i}\right)(d\zb^i)\wedge^*_{\tomega_{(k)}}d\zb^I-
  \left(-\dfrac{\p f}{\p z^i}-f\dfrac{\p \varphi_{0}}{\p z^i}\right) (d\zb^i)\wedge^*_{\omega_0}d\zb^I \\
 -f&\left((d\zb^i\wedge^*_{\tomega_{(k)}})\overline{{\theta}}_{\p/\p \zb^i,\tomega_{(k)}}^*-(d\zb^i\wedge^*_{\omega_0})\overline{{\theta}}_{\p/\p \zb^i,\omega_0}^*\right)d\zb^I+f\left((\pb \tphi_{(k)})\wedge^*_{\tomega_{(k)}}-(\pb\phi_0)\wedge^*_{\omega_0}\right)d\zb^I.
 \end{align*}Denote $a_1(z)$ and $a_2(z)$ as the absolute maximum of the coefficients of the differential operators $\tpb_{(k),s}-\pb_{0,s}$ and $\tpb^*_{(k),s}-\pb^*_{0,s}$ at a point $z\in\Cn$, respectively. By (\ref{3.1 ob phi}) and (\ref{3.1 ob omega}), 
    \begin{eqnarray*}
      |\tphi_{(k)}-\phi_0|_{\Ct}(z) \lesssim \dfrac{|z|^3+1}{\sqrt{k}}\quad;\quad
      |\tomega_{(k)}-\omega_0|_{\Ct}(z) \lesssim \frac{|z|+1}{\sqrt{k}}\quad \text{for all }\, |z|<2k^{\epsilon}.
      \end{eqnarray*}The coefficients of the differential operators $\tpb_{(k),s}-\pb_{0,s}$ and $\tpb^*_{(k),s}-\pb^*_{0,s}$ consist of the zero and first derivatives of $\phi_0-\tphi_{(k)}$. Moreover, the matrix of connection forms $\theta$ and the operator $\wedge^*_{\cdot}$ are smoothly depend on the zero and first derivatives of components of Hermitian forms. We can see that \begin{equation*}
        |a_i(z)|\lesssim \dfrac{|z|^3+1}{\sqrt{k}}\quad \forall\, |z|<2k^{\epsilon} \quad \text{and}\quad |a_i(z)|=0\quad\forall\,|z|>2k^{\epsilon}. \end{equation*}Because any derivatives of $u$ decay exponentially as $|z|$ goes to infinity, there is a constant $c > 0$ such that 
     \begin{equation*}
    |(\tpb_{(k),s}-\tpb_{0,s})u(z)|_{\omega_0} \lesssim \dfrac{|z|^3+1}{\sqrt{k}}e^{-c|z|^2} \quad \text{\rm and} \quad |(\tpb^*_{(k),s}-\tpb^*_{0,s})u(z)|_{\omega_0} \lesssim \dfrac{|z|^3+1}{\sqrt{k}}e^{-c|z|^2} 
     \end{equation*} for all $z\in\Cn$. Since $|z|^3 e^{-c|z|^2}$ is a bounded function, we have completed the proof.
\end{proof}
We can apply Lemma \ref{4.5 thm 1 lem 1} to establish the following theorem which claims that $u_{(k)}\To u$.
\begin{thm}\label{4.5 thm 1}If $u=z^{\alpha}_q e^{-\sum_{i=1}^{n}|\lambda_i||z^i|^2}d\zb^I$ for some $\alpha\in \N_0^{n}$, then
  \[\|u_{(k)}-u\|_{\tomega_{(k)}} \rightarrow 0\quad \text{as }\,k\To\infty.\]
\end{thm}
\begin{proof}
     Note that
     $
     \|u_{(k)}-u\|_{\tomega_{(k)}} \leq \|u_{(k)}-\tchi_k u\|_{\tomega_{(k)}}+\|\tchi_{k}u-u\|_{\tomega_{(k)}}
     $. Clearly, the second term tends to zero by the decreasing of $u$ as $z\To \infty$. For the first term,
     \begin{align*}
     \|u_{(k)}-\tchi_ku\|_{\tomega_{(k)}}= \|\tchi_k (\tBlks \chi_k u-u)\|_{\tomega_{(k)}}&\leq \|\tBlks \chi_k u-u\|_{\tomega_{(k)}}\\ &\leq \|\tBlks \chi_k u-\chi_k u\|_{\tomega_{(k)}}+\|\chi_k u-u\|_{\tomega_{(k)}}.    \end{align*}
      Since the second term of the right-hand side tends to zero, we only need to estimate $\|\tBlks \chi_k u-\chi_k u\|_{\tomega_{(k)}}$. By Theorem \ref{4.4 thm spectral gap},
     \[
     \|\tBlks \chi_k u-\chi_k u\|^2_{\tomega_{(k)}} \lesssim \|\tpb^*_{(k),s}\chi_ku\|^2_{\tomega_{(k)}}+\|\tpb_{(k),s}\chi_ku\|^2_{\tomega_{(k)}}.
     \]It remains to claim $\|\tpb^*_{(k),s}\chi_ku\|^2_{\tomega_{(k)}} \To 0$ and $\|\tpb_{(k),s}\chi_ku\|^2_{\tomega_{(k)}} \To 0$.\par 
     For $\|\tpb_{(k),s}\chi_ku\|^2_{\tomega_{(k)}}$, we compute that $\tpb_{(k),s}\chi_ku=(\pb \chi_k)\wedge u+\chi \tpb_{k,s}u$ and then 
\begin{align*}\|\tpb_{(k),s}\chi_ku\|^2_{\tomega_{(k)}} &=\int_{\{|z|<k^{\epsilon}/7\}}|\tpb_{k,s}u|_{\tomega_{(k)}}^2dV_{\tomega_{(k)}}+\int_{\{k^{\epsilon}/7<|z|<2k^{\epsilon}/7\}}|(\pb \chi_k)\wedge u+\chi \tpb_{k,s}u|^2_{\tomega_{(k)}}dV_{\tomega_{(k)}} \\
     &\lesssim
    \int_{\{|z|<2k^{\epsilon}/7\}}|\tpb_{k,s}u|_{\omega_{0}}^2dm+\int_{\{k^{\epsilon}/7<|z|<2k^{\epsilon}/7\}}|u|^2_{\omega_{0}}dm.
     \end{align*}Clearly, the second term $\int_{\{k^{\epsilon}/7<|z|<2k^{\epsilon}/7\}}|u|^2_{\omega_{0}}dm$ tends to zero  by the decreasing of $u$ as $z\To\infty$. By Lemma \ref{4.5 thm 1 lem 1}  and the setting $\epsilon<1/(2n)$, the first term can be dominated by
   \[
   \int_{\{|z|<2k^{\epsilon}/7\}}|\tpb_{(k),s}u|_{\omega_{0}}^2dm\lesssim \dfrac{(k^{\epsilon})^{2n}}{k}\To 0.
   \] We have proven that $\|\tpb_{(k),s}\chi_ku\|^2_{\tomega_{(k)}} \To 0$. 
   Next, we will show $\|\tpb^*_{(k),s}\chi_ku\|_{\tomega_{(k)}}\To 0$ in a similar way. Compute $\tpb^*_{(k),s}\chi_ku_k=\sum_{i=1}^{n}\dfrac{\p \chi_k}{\p z^i}(d\zb^i)\wedge^*_{\tomega_{(k)}}u+\chi_k \ \tpb^*_{(k),s}u$ and repeat the above process to get
   \begin{align*}\|\tpb^*_{(k),s}\chi_ku\|^2_{\tomega_{(k)}}\lesssim
    \int_{\{|z|<2k^{\epsilon}/7\}}|\tpb^*_{(k),s}u|_{\omega_{0}}^2dm+\int_{\{k^{\epsilon}/7<|z|<2k^{\epsilon}/7\}}|u|^2_{\omega_{0}}dm.
     \end{align*}The second term clearly tends to zero and the first term also tends to zero by the fact that $\epsilon<1/(2n)$ and Lemma \ref{4.5 thm 1 lem 1}. We also have $\|\tpb^*_{k,s}\chi_ku\|_{\tomega_{(k)}} \to 0$.
\end{proof}

In the next step, we will display $\Plks u_{(k)} -u_{(k)}\To 0$ in $L^2_{\tomega_{(k)}}(\Cn,\T)$. To do this, we need to estimate the decreasing rate of $\|\Box^{(q)}_{(k),s}u_{(k)}\|_{\omega_{(k)}}$ as (\ref{123456789}) in the proof of Theorem \ref{4.3 thm 2}. Since $\left(\Box^{(q)}_{(k),s}u_{(k)}\right)(z)=0$ for all $|z|<(1/7)k^{\epsilon}$, we only need to analyze $u_{(k)}$ on the annuli $\{(2/7)k^{\epsilon}<|z|<k^{\epsilon}\}$. The following lemma tells us that $u_{(k)}$ are \textbf{small} on the annuli. Notably, the proof effectively utilizes the property that $\supp \rho_k\cap\supp \tchi_k=\emptyset$. 

\begin{lem}\label{4.5 thm2 lem1}
Consider the functional
\[
\rho_k \tBlks \chi_k: L^2_{\tomega_{(k)}}(\Cn,\T) \rightarrow L^2_{\tomega_{(k)}}(\Cn,\T).
\]For any $d\in \N$, there exists a constant $C$ and $n_0\in \N$ such that the operator norm
\[
\|\rho_k \tBlks \chi_k\|_{\tomega_{(k)}} \leq \dfrac{C}{k^{d}} \quad \text{\rm for all }\, k\geq n_0.
\]
\end{lem}
\begin{proof}
For any $u\in \Omega^{(0,q)}_c(\Cn)$, by Theorem \ref{4.4 thm B=I-},
\begin{align}\label{4.5 lem1 (0)}
\rho_k\tBlks\chi_ku&=\rho_k\left(\Id-\tpb^*_{(k),s}N^{q+1}_k\tpb_{(k),s}-\tpb_{(k),s}N^{q-1}_k\tpb^*_{(k),s}\right)\chi_ku \\ &= -\rho_k\tpb^*_{(k),s}N^{q+1}_k\tpb_{(k),s}\chi_ku-\rho_k\tpb_{(k),s}N^{q+1}_k\tpb^*_{(k),s}\chi_k u.\notag 
 \end{align}Now, we aim to estimate $\|\rho_k\tpb^{*}_{(k),s}N^{q+1}_k\tpb_{(k),s}\chi_ku\|_{\tomega_{(k)}}$. Observe that
 \begin{align*}
\|\rho_k\tpb^{*}_{(k),s}N^{q+1}_k\tpb_{(k),s}\chi_ku\|^2_{\tomega_{(k)}}&=\left(\rho_k\tpb^{*}_{(k),s}N^{q+1}_k\tpb_{(k),s}\chi_ku\mid \rho_k\tpb^{*}_{(k),s}N^{q+1}_k\tpb_{(k),s}\chi_ku\right)_{\tomega_{(k)}}\\&
=\left(N^{q+1}_k\tpb_{(k),s}\chi_ku\mid \tpb_{(k),s}\rho^2_k\tpb^{*}_{(k),s}N^{q+1}_k\tpb_{(k),s}\chi_ku\right)_{\tomega_{(k)}}\notag\\&=\left(\Tilde{\rho}_kN^{q+1}_k\tpb_{(k),s}\chi_ku\mid \tpb_{(k),s}\rho^2_k\tpb^{*}_{(k),s}N^{q+1}_k\tpb_{(k),s}\chi_ku\right)_{\tomega_{(k)}}\notag\\&\leq\|\Tilde{\rho}_k N^{q+1}_k\tpb_{(k),s}\chi_ku\|_{\tomega_{(k)}}\|\tpb_{(k),s}\rho^2_k\tpb^{*}_{(k),s}N^{q+1}_k\tpb_{(k),s}\chi_ku\|_{\tomega_{(k)}}, \notag
 \end{align*}
where $\Tilde{\rho}_k\in \Cinf_c(\Cn)$ is another cut-off function such that $\supp \Tilde{\rho}_k\supset \supp \rho_k$ and $\supp \Tilde{\rho}_k \cap \supp\chi_k=\emptyset$. By direct computation,
\begin{align*}
\tpb_{(k),s}\rho^2_k\tpb^{*}_{(k),s}N^{q+1}_k\tpb_{(k),s}\chi_k u&=(\pb \rho^2_k)\wedge \tpb^{*}_{(k),s}N^{q+1}_k\tpb_{(k),s}\chi_ku+ \rho^2_k\tpb_{(k),s}\tpb^{*}_{(k),s}N^{q+1}_k\tpb_{(k),s}\chi_k u\\&=(\pb \rho^2_k)\wedge \tpb^{*}_{(k),s}N^{q+1}_k\tpb_{(k),s}\chi_ku+ \rho^2_k\tpb_{(k),s}\chi_k u \notag\\&= (\pb \rho^2_k)\wedge \tpb^{*}_{(k),s}N^{q+1}_k\tpb_{(k),s}\chi_ku,\notag\end{align*}
where the second equality is from (\ref{4.4 eq 2}) and the third is by the fact that $\supp\rho_k \cap \supp \chi_k=\emptyset$.
We apply this computation to continue the previous estimate and get
\begin{align}\label{4.5 lem1 (3)}
\|\rho_k\tpb^{*}_{(k),s}N^{q+1}_k\tpb_{(k),s}\chi_k u\|^2_{\tomega_{(k)}}&\leq\|\Tilde{\rho}_k N^{q+1}_k\tpb_{(k),s}\chi_ku\|_{\tomega_{(k)}}\|(\pb \rho^2_k)\wedge \tpb^{*}_{(k),s}N^{q+1}_k\tpb_{(k),s}\chi_ku\|_{\tomega_{(k)}}\\&\lesssim k^{-\epsilon}\|\Tilde{\rho}_kN^{q+1}_k\tpb_{(k),s}\chi_k u\|_{\tomega_{(k)}}\|\Tilde{\rho}_k\tpb^*_{(k),s}N^{q+1}_{k}\tpb_{(k),s}\chi_k u\|_{\tomega_{(k)}},\notag 
\end{align} 
where the term $k^{-\epsilon}$ arises during the computation of $\pb\rho_k$ since $\sup_{|\alpha|=1} |\p^{\alpha}\rho_k|\lesssim k^{-\epsilon}$. Moreover, the sequence $\Tilde{\rho}_k$ can be taken to satisfy the condition $\sup_{|\alpha|=1} |\p^{\alpha}\Tilde{\rho}_k|\lesssim k^{-\epsilon}$.  To iterate the preceding process , we show the following claim:

\begin{claim}
   There exists $\Tilde{\Tilde{\rho}}_k\in \Cinf_c(\Cn)$ with $\supp \Tilde{\Tilde{\rho}}_k\supset \supp \Tilde{\rho}_k$ and $\supp \Tilde{\Tilde{\rho}}_k \cap \supp\chi_k=\emptyset$ such that $\sup_{|\alpha|=1} |\p^{\alpha}\Tilde{\Tilde{\rho}}_k|\lesssim k^{-\epsilon}$ and \begin{multline*}
   \|\Tilde{\rho}_kN^{q+1}_{k}\tpb_{(k),s}\chi_k u\|_{\tomega_{(k)}}\|\Tilde{\rho}_k\tpb^*_{(k),s}N^{q+1}_{k}\tpb_{(k),s}\chi_k u\|_{\tomega_{(k)}} \\ \lesssim k^{-\epsilon}\|\Tilde{\Tilde{\rho}}_kN^{q+1}_k\tpb_{(k),s}\chi_k u\|_{\tomega_{(k)}}\|\Tilde{\Tilde{\rho}}_k\tpb^*_{(k),s}N^{q+1}_{k}\tpb_{(k),s}\chi_k u\|_{\tomega_{(k)}}.    
   \end{multline*}
    
\end{claim}
To show the claim, by Lemma \ref{4.4 lem 1}, we get 
\[\|\Tilde{\rho}_kN^{q+1}_k\tpb_{(k),s}\chi_k u\|_{\tomega_{(k)}}\lesssim \|\tpb_{(k),s}\Tilde{\rho}_kN^{q+1}_{k}\tpb_{(k),s}\chi_k u\|_{\tomega_{(k)}}+\|\tpb^*_{(k),s}\Tilde{\rho}_kN^{q+1}_{k}\tpb_{(k),s}\chi_k u\|_{\tomega_{(k)}}.\]
Moreover, we compute directly that
\begin{align*}
\tpb_{(k),s}\Tilde{\rho}_kN^{q+1}_k\tpb_{(k),s}\chi_k u&=(\pb \Tilde{\rho}_k)\wedge N^{q+1}_k\tpb_{(k),s}\chi_k u+\Tilde{\rho}_k\tpb_{(k),s}N^{q+1}_k\tpb_{(k),s}\chi_k u =(\pb \Tilde{\rho}_k)\wedge N^{q+1}_k\tpb_{(k),s}\chi_k u;\\
\tpb^*_{(k),s}\Tilde{\rho}_kN^{q+1}_k\tpb_{(k),s}\chi_k u&=-\sum_{i=1}^n\dfrac{\p \Tilde{\rho}_k}{\p z^i}d\zb^i\wedge^*_{\tomega_{(k)}} N^{q+1}_k\tpb_{(k),s}\chi_k u+\Tilde{\rho}_k\tpb^*_{(k),s}N^{q+1}_k\tpb_{(k),s}\chi_k u.
\end{align*}Substitute these equations into the estimate and then dominate $\|\Tilde{\rho}_kN^{q+1}_k\tpb_{(k),s}\chi_k u\|_{\tomega_{(k)}}$ by   \begin{align*}
 \|(\pb \Tilde{\rho}_k)\wedge N^{q+1}_k\tpb_{(k),s}\chi_k u\|_{\tomega_{(k)}}&+ \|\sum_{i=1}^n\dfrac{\p \Tilde{\rho}_k}{\p z^i}d\zb^i\wedge^*_{\tomega_{(k)}} N^{q+1}_k\tpb_{(k),s}\chi_k u\|_{\tomega_{(k)}}+\|\Tilde{\rho}_k\tpb^*_{(k),s}N^{q+1}_k\tpb_{(k),s}\chi_k u\|_{\tomega_{(k)}}\\&\lesssim k^{-\epsilon}\|\Tilde{\Tilde{\rho}}_k N^{q+1}_k\tpb_{(k),s}\chi_k u\|_{\tomega_{(k)}}+\|\Tilde{\Tilde{\rho}}_k\tpb^*_{(k),s}N^{q+1}_k\tpb_{(k),s}\chi_k u\|_{\tomega_{(k)}}     
\end{align*}for some $\Tilde{\Tilde{\rho}}_k$ as described above. So, \begin{multline*}
\|\Tilde{\rho}_kN^{q+1}_{k}\tpb_{(k),s}\chi_k u\|_{\tomega_{(k)}}\|\Tilde{\rho}_k\tpb^*_{(k),s}N^{q+1}_{k}\tpb_{(k),s}\chi_k u\|_{\tomega_{(k)}}\\ \lesssim k^{-\epsilon}\|\Tilde{\Tilde{\rho}}_k N^{q+1}_k\tpb_{(k),s}\chi_k u\|_{\tomega_{(k)}}\|\Tilde{\Tilde{\rho}}_k\tpb^*_{(k),s}N^{q+1}_{k}\tpb_{(k),s}\chi_k u\|_{\tomega_{(k)}} +\|\Tilde{\Tilde{\rho}}_k\tpb^*_{(k),s}N^{q+1}_k\tpb_{(k),s}\chi_k u\|^2_{\tomega_{(k)}}.
\end{multline*}
For the last term of the right-hand side, we replace the $\rho_k$ by $\Tilde{\Tilde{\rho}}_k$ in (\ref{4.5 lem1 (3)}) and get
 \begin{equation*}
\|\Tilde{\Tilde{\rho}}_k\tpb^*_{(k),s}N^{q+1}_k\tpb_{(k),s}\chi_k u\|^2_{\tomega_{(k)}} \lesssim k^{-\epsilon}\|\Tilde{\Tilde{\Tilde{\rho}}}_kN^{q+1}_k\tpb_{(k),s}\chi_k u\|_{\tomega_{(k)}}\|\Tilde{\Tilde{\Tilde{\rho}}}_k\tpb^*_{(k),s}N^{q+1}_{k}\tpb_{(k),s}\chi_k u\|_{\tomega_{(k)}}.
 \end{equation*}Combining the above estimates, we have completed the claim. Next, by (\ref{4.5 lem1 (3)}) and  iterating the claim , we can conclude that for any integer $N\in\N$, there exists a constant $C$ and $\Tilde{\rho}_k\in \Cinf_c(\Cn)$ with $\supp \Tilde{\rho}_k\supset \supp \rho_k$ and $\supp \Tilde{\rho}_k \cap \supp\chi_k=\emptyset$ such that
 \[
\|\rho_k\tpb^{*}_{(k),s}N^{q+1}_k\tpb_{(k),s}\chi_k u\|^2_{\tomega_{(k)}}\leq Ck^{-N}\|\Tilde{\rho}_kN^{q+1}_{k}\tpb_{(k),s}\chi_k u\|_{\tomega_{(k)}}\|\Tilde{\rho}_k\tpb^*_{(k),s}N^{q+1}_{k}\tpb_{(k),s}\chi_k u\|_{\tomega_{(k)}}.\]
 Finally, we need to show the following fact:
 \begin{claim}
For all $v\in \Omega^{(0,q)}_c(\Cn)$
\begin{equation*}
\|\tpb^*_{(k),s}N_k^{q+1}\tpb_{(k),s}v\|_{\tomega_{(k)}}\leq \|v\|_{\tomega_{(k)}} \quad;\quad \|N_k^{q+1}\tpb_{(k),s}v\|_{\tomega_{(k)}}\lesssim \|v\|_{\tomega_{(k)}}.   
\end{equation*}
  \end{claim}For the first term, by (\ref{4.4 eq 2}), we compute that
 \begin{align*}
\|\tpb^*_{(k),s}N_k^{q+1}\tpb_{(k),s}v\|^2_{\tomega_{(k)}}&=\left(N_k^{q+1}\tpb_{(k),s}v\mid \tpb_{(k),s}\tpb^*_{(k),s}N_k^{q+1}\tpb_{(k),s}v\right)_{\tomega_{(k)}}=\left(N_k^{q+1}\tpb_{(k),s}v\mid \tpb_{(k),s}v\right)_{\tomega_{(k)}}\\ &=\left(\tpb^*_{(k),s}N_k^{q+1}\tpb_{(k),s}v\mid v\right)_{\tomega_{(k)}}\leq\|\tpb^*_{(k),s}N_k^{q+1}\tpb_{(k),s}v\|_{\tomega_{(k)}}\|v\|_{\tomega_{(k)}}.\end{align*} We get $\|\tpb^*_{(k),s}N_k^{q+1}\tpb_{(k),s}v\|_{\tomega_{(k)}}\leq \|v\|_{\tomega_{(k)}}$.  
 The second term follows by Lemma \ref{4.4 lem 1} that
 \begin{align*}
\|N_k^{q+1}\tpb_{(k),s}v\|_{\tomega_{(k)}}&\lesssim\|\tpb_{(k),s}N_k^{q+1}\tpb_{(k),s}v\|_{\tomega_{(k)}}+\|\tpb^*_{(k),s}N_k^{q+1}\tpb_{(k),s}v\|_{\tomega_{(k)}}\\ &=\|\tpb^*_{(k),s}N_k^{q+1}\tpb_{(k),s}v\|_{\tomega_{(k)}}\leq\|v\|_{\tomega_{(k)}},    
 \end{align*}since $\tpb_{(k),s}N_k^{q+1}\tpb_{(k),s}=0$. We completed the proof of the second claim. After combining all the above results, we know that for any integer $N\in\N$, there exists a constant $C$ such that
 \[
\|\rho_k\tpb^{*}_{(k),s}N^{q+1}_k\tpb_{(k),s}\chi_k u\|_{\tomega_{(k)}}\leq Ck^{-N}\|u\|_{\tomega_{(k)}}.
 \]Symmetrically, we can literally repeat the process to show the analogous statement:
\[
\|\rho_k\tpb_{(k),s}N^{q+1}_k\tpb^{*}_{(k),s}\chi_ku\|_{\tomega_{(k)}} \lesssim Ck^{-N}\|u\|_{\tomega_{(k)}}.
\]Then the lemma follows by (\ref{4.5 lem1 (0)}) and a density argument. 
 \end{proof}
\begin{cor}\label{4.5 thm2 cor}
    For any $u\in L^2_{\omega_0}(\Cn,\T)$ and $d\in \N$, \[k^{d}\left(\|\tpb^{*}_{(k),s}u_{(k)}\|_{\tomega_{(k)}}^2+\|\tpb_{(k),s}u_{(k)}\|_{\tomega_{(k)}}^2\right) \rightarrow 0. 
    \]
\end{cor}
\begin{proof}
    Recall the fact that $\tpb_{(k),s}\tBlks=0$ and $\tpb^*_{(k),s}\tBlks=0$. 
    \begin{align*}
    \tpb_{(k),s}u_{(k)}&=(\pb \Tilde{\chi}_k)\wedge\tBlks \chi_ku+\Tilde{\chi}_k\tpb_{(k),s}\tBlks \chi_ku=(\pb \Tilde{\chi}_k)\wedge\tBlks \chi_ku;\\
\tpb^*_{(k),s}u_{(k)}&=-\sum_{i=1}^{n}\dfrac{\p \Tilde{\chi}_k}{\p z^i}(d\zb^i)\wedge^*_{\tomega_{(k)}}\tBlks \chi_ku+\Tilde{\chi}_k \ \tpb^*_{(k),s}\tBlks \chi_ku=-\sum_{i=1}^{n}\dfrac{\p \Tilde{\chi}_k}{\p z^i}(d\zb^i)\wedge^*_{\tomega_{(k)}}\tBlks \chi_ku.   
    \end{align*}Observe that derivatives of $\Tilde{\chi}_k$ are supported in the annuli $\{3k^{\epsilon}/7<|z|<6k^{\epsilon}/7\}$ and $\rho_k\equiv 1$ on the annuli. We can see \begin{align*}
\|\tpb_{(k),s}u_{(k)}\|^2_{\tomega_{(k)}} \lesssim \|\rho_k\tBlks\chi_k u\|^2_{\tomega_{(k)}} \quad;\quad 
\|\tpb^*_{(k),s}u_{(k)}\|^2_{\tomega_{(k)}} \lesssim \|\rho_k\tBlks\chi_k u\|^2_{\tomega_{(k)}}.
\end{align*} By Lemma \ref{4.5 thm2 lem1}, we can immediately derive the corollary.
\end{proof}
Now, we show the theorem claiming that $\Plks u_{(k)}- u_{(k)}\To 0$  which is similar to Lemma \ref{4.3 thm 2} in Section \ref{section 4.3}.
\begin{thm}\label{4.5 thm2}
If there exists $d \in \R$ such that $\liminf_{k\To\infty} k^{d}c_k>0$, then we have
\begin{equation*}
\|\Plks u_{(k)}-u_{(k)}\|_{\omega_{(k)}} \rightarrow 0 \quad \text{as }\,k\To\infty.
\end{equation*}
In the case $c_k=0$ for all $k$, the convergence holds under the local small spectral gap condition of polynomial rate in $U$ (cf. Def. \ref{1.2 spectral gap 1}).
\end{thm}
\begin{proof}
   Define $
    u_{k}(z):=k^{n/2}u_{(k)}(\sqrt{k}z)$ which is a section with compact support in $U$. Then we have
    \[
    \|\Plks u_{(k)}-u_{(k)}\|_{\omega_{(k)},B(\sqrt{k})}=\|\Pks u_k-u_k\|_{\omega,B(1)}.
    \]By the property of spectral kernel,
    \begin{equation*}
     \|\Pks u_k-u_k\|^2_{\omega,B(1)} \leq \dfrac{1}{c_k}\left(\Box^{(q)}_{k,s}u_k \mid u_k\right)_{\omega}=\dfrac{1}{c_k}\left(\|\pb^*_{k,s}u_k\|^2_{\omega}+\|\pb_{k,s}u_k\|^2_{\omega}\right).
    \end{equation*}
    Moreover, note that \[\left(\|\pb^*_{k,s}u_k\|^2_{\omega}+\|\pb_{k,s}u_k\|^2_{\omega}\right)=k\left(\|\tpb^*_{(k),s}u_{(k)}\|^2_{\tomega_{(k)}}+\|\tpb_{(k),s}u_{(k)}\|^2_{\tomega_{(k)}}\right)\] by the relation (\ref{3.2 rescale relation}). Combine them and get
    \begin{align*}
    \|\Plks u_{(k)}-u_{(k)}\|_{\omega_{(k)},B(\sqrt{k})}&\leq \dfrac{k}{c_k}\left(\|\tpb^*_{(k),s}u_{(k)}\|^2_{\tomega_{(k)}}+\|\tpb_{(k),s}u_{(k)}\|^2_{\tomega_{(k)}}\right).
    \end{align*}By the assumption that $\liminf_{k\To\infty} k^{N}c_k>0$ for some $N\in\N$ and Corollary \ref{4.5 thm2 cor}, the right-hand sides of equations above must tend to zero.\par
    In the Bergman kernel case $c_k=0$, we apply the spectral gap condition \ref{1.2 spectral gap 2} and get
    \begin{multline*}
        \|\Blks u_{(k)}-u_{(k)}\|_{\omega_{(k)}}=\|\Bks u_k-u_k\|^2_{\omega,B(1)}\\ \lesssim k^d\left(\Box^{(q)}_{k,s}u_k \mid u_k\right)_{\omega}=k^{d+1}\left(\|\tpb^*_{(k),s}u_k\|^2_{\tomega_{(k)}}+\|\tpb_{(k),s}u_k\|^2_{\tomega_{(k)}}\right).
    \end{multline*} 
   We apply Corollary \ref{4.5 thm2 cor} to complete the proof.
   \end{proof}
Now, we are ready to overcome Statement \ref{4.2 statement 2} for the general cases of $(0,q)$-forms.
\begin{thm}\label{4.5 thm Bu=u}
If there exists $d \in \R$ such that $\liminf_{k \rightarrow \infty} k^{d}c_k>0$,
\begin{equation*}
    \Bs u=u \quad \text{\rm for all }\, u \in \Ker \Box^{(q)}_{0,s}.
\end{equation*}As for the Bergman kernel case $c_k=0$, it also holds under the local small spectral gap condition of polynomial rate in $U$ (cf. Def. \ref{1.2 spectral gap 1}).
\end{thm}
\begin{proof}
   By Theorem \ref{4.1 Model theorem}, we may assume that $u$ is of the form $u=z^{\alpha}_qe^{-\sum|\lambda_i||z^i|^2}d\zb^I$ for some $\alpha\in\N_0^n$ by density argument. By Lemma \ref{3.3 thm pf 1.5}, Lemma \ref{4.5 thm 1} and the decreasing of $u$,  \begin{align}\label{4.5 thm Bu=u (1)}
       \|\Plks (\chi_k u-u_{(k)})\|_{\omega_0} \lesssim \|\chi_k u-u_{(k)}\|_{\omega_{0}}\leq\|\chi_k u-u\|_{\omega_0}+\|u-u_{(k)}\|_{\omega_{0}}\To 0.
   \end{align}
   To show $\Bs u=u$, let $v\in\Omega^{(0,q)}_{c}(\Cn)$ and observe that
    \begin{multline*}
       \left(v\mid\Bs u-u\right)_{\omega_0} =\left(v\mid\Bs u-\Plks \chi_ku\right)_{\omega_0}+\left(v\mid \Plks(\chi_ku-u_{(k)})\right)_{\omega_0}\\+ \left(v\mid\Plks u_{(k)}-u_{(k)}\right)_{\omega_0}+\left(v\mid u_{(k)}-u\right)_{\omega_0}.  
    \end{multline*}
    By Lemma \ref{4.3 lem 3} , Theorem \ref{4.5 thm 1}, Theorem \ref{4.5 thm2} and (\ref{4.5 thm Bu=u (1)}), the right-hand side of the above equation must tend to zero.  
\end{proof}
Eventually, we are able to complete the proof main theorem for the case $p\in M(q)$ by Remark \ref{4.2 rmk}.
\begin{thm}\label{4.5 main thm}
 Suppose $c_k$ is a sequence such that \[\limsup_{k\To\infty}\dfrac{c_k}{k}=0.\] If $p\in M(q)$ and $\liminf_{k\To\infty} k^dc_k>0$ for some $d\in \N$, then  
    \begin{equation*}
        P^{(q),s}_{(k),c_k}(z,w) \To \dfrac{|\lambda_1 \cdots \lambda_n|}{\pi^n}\,e^{2(\sum_{i=1}^{q}|\lambda_i|\zb^i w^i+\sum_{i=q+1}^{n}|\lambda_i|z^i\wb^i-\sum_{i=1}^{n}|\lambda_i||w^i|^2)}d\zb^I \otimes (\pwb)^I 
       \end{equation*}locally uniformly
       in $\Cinf$ on $\Cn$.
    In the case $c_k=0$ for all $k\in\N$, the convergence also holds if $\Box^{(q)}_{k}$ satisfies the local small spectral gap condition of polynomial rate in $U$ (cf. Def. \ref{1.2 spectral gap 1}).
\end{thm}

\end{document}